\documentclass[11pt,a4paper,reqno]{amsart}
\usepackage{a4wide,amsmath}
\usepackage{amsthm}
\usepackage{amsfonts}
\usepackage{amssymb}
\usepackage{graphicx}
\usepackage{mathrsfs}
\usepackage{bbm,pdfsync}
\usepackage{color}
\definecolor{colorGreen}{rgb}{0.,0.67,0}
\definecolor{colorRed}{rgb}{0.67,0.,0.}
\definecolor{colorBlue}{rgb}{0.,0.,0.67}

\newtheorem{theorem}{Theorem}[section]
\newtheorem{lemma}[theorem]{Lemma}

\newtheorem{proposition}[theorem]{Proposition}

\newtheorem{definition}[theorem]{Definition}
\theoremstyle{remark}
\newtheorem{remark}[theorem]{\it \bf{Remark}\/}

\numberwithin{equation}{section}
\catcode`@=11
\def\section{\@startsection{section}{1}%
  \z@{1.5\linespacing\@plus\linespacing}{.5\linespacing}%
  {\normalfont\bfseries\large\centering}}
\catcode`@=12
\newcommand{\be}{\begin{equation}}
\newcommand{\ee}{\end{equation}}
\newcommand{\bea}{\begin{eqnarray}}
\newcommand{\eea}{\end{eqnarray}}
\newcommand{\bee}{\begin{eqnarray*}}
\newcommand{\eee}{\end{eqnarray*}}

\def\pa{\partial}

\def\RR{\mathbb{R}}

\def\fref#1{{\rm (\ref{#1})}}

\catcode`@=11
\def\supess{\mathop{\operator@font Sup\,ess}}
\catcode`@=12

\def\RR{\mathbb{R}}

\def\e{\varepsilon}

\def\fref#1{{\rm (\ref{#1})}}

\def\R2+{\RR ^2_+}

\def\lsl{\frac{\lambda_s}{\lambda}}

\def\pa{\partial}

\def\lim{\mathop{\rm lim}}

\def\sup{\mathop{\rm sup}}
\def\exp{{\rm exp}}
\def\e{\varepsilon}
\def\l{\lambda}

\def\log{{\rm log}}

\def\lsl{\frac{\lambda_s}{\lambda}}

\def\eb{\e}
\def\ebo{\e}

\def\qbt{P_{B_1}}

\def\qbtb{P_{B_1}}

\def\pbt{\tilde{P}_{B_0}}

\def\pa{\partial}

\def\vt{\tilde{V}}

\def\pa{\partial}

\title[]{Smooth type II blow up solutions to the four dimensional energy critical wave equation}
\author[M. Hillairet]{Matthieu Hillairet}
\address{Ceremade, Universit\'e Paris Dauphine, France}
\email{hillairet@ceremade.dauphine.fr}\author[P. Rapha\"el]{Pierre Rapha\"el}
\address{Institut de Math\'ematiques de Toulouse, Universit\'e Toulouse III, France}
\email{pierre.raphael@math.univ-toulouse.fr}

\begin{document}
\maketitle

\begin{abstract}
We exhibit $\mathcal C^{\infty}$ type II blow up solutions to the focusing energy critical wave equation in dimension $N=4$. These solutions admit near blow up  time a decomposiiton $$u(t,x)=\frac{1}{\lambda^{\frac{N-2}{2}}(t)}(Q+\e(t))(\frac{x}{\lambda(t)}) \ \ \mbox{with} \ \ \|\e(t),\pa_t\e(t)\|_{\dot{H}^1\times L^2}\ll1 $$ where $Q$ is the extremizing profile of the Sobolev embedding $\dot{H}^1\to L^{2^*}$, and a blow up speed $$\lambda(t)=(T-t)e^{-\sqrt{|\log (T-t)|}(1+o(1))} \ \ \mbox{as} \ \ t\to T.$$
\end{abstract}


\section{Introduction}


\subsection{Setting of the problem}

We deal in this paper with the energy critical focusing wave equation 
\be
\label{wave}
\left\{\begin{array}{ll}\pa_{tt}u-\Delta u-f(u)=0 \ \ \ \mbox{with} \ \ f(t)=t^{\frac{N+2}{N-2}}, \\ (u,\pa_tu)_{|t=0}=(u_0, u_1), \ \ (t,x)\in \RR\times \RR^N.
\end{array} \right .
\ee
in dimension $$N=4.$$ This is a special case of the nonlinear wave equation 
\be
\label{energycri}
\pa_{tt}u-\Delta u-f(u)=0
\ee
which since the pioneering works by J\"orgens \cite{Jurgens} has attracted a considerable amount of works. For the energy critical nonlinearity $f(u)=\pm t^{\frac{N+2}{N-2}}$, the Cauchy problem is locally well posed in the energy space $\dot{H}^1\times L^2$ and the solution propagates regularity, see for example Sogge \cite{sogge} and references therein. Recall that in this case, \fref{energycri} admits a conserved energy $$E(u(t))=E(u_0,u_1)=\frac12 \int (\pa_tu)^2+\frac 12\int|\nabla u|^2\mp\frac{N-2}{2N}\int u^{\frac{2N}{N-2}}$$ which is left invariant by the scaling symmetry of the flow: $$u_{\lambda}(t,x)=\frac{1}{\lambda^{\frac{N-2}{2}}}u(\frac{t}{\lambda},\frac{x}{\lambda^2}).$$ Global existence in the defocusing case was proved by Struwe \cite{struweone} for radial data and Grillakis \cite{grillakis} for general data. For focusing nonlinearities, a sharp threshold criterion of global existence and scattering or finite time blow up is obtained by Kenig and Merle \cite{KMacta} based on the solitonic solution to \fref{wave}:
\be
\label{defqexplicite}
Q(r)=\left(\frac{1}{1+\frac{r^2}{N(N-2)}}\right)^{\frac{N-2}{2}}
\ee 
which is the extremizing profile of the Sobolev embedding $\dot{H}^1\to L^{2^*}$. Indeed, for initial data $(u_0,u_1)$ such that $E(u_0,u_1)<E(Q,0)$, those with $\|\nabla u_0\|_{L^2}<\|\nabla Q\|_{L^2}$ have global solutions and scatter, while those with $\|\nabla u_0\|_{L^2}>\|\nabla Q\|_{L^2}$ lead to finite time blow up.\\
Note that like in the works by Levine \cite{levine}, see also Strauss \cite{strauss}, and as is standard in a nonlinear dispersive setting, blow up is derived through obstructive convexity arguments, see also Karageorgis and Strauss \cite{karastrauss} for refined statements near the soliton $Q$. However, this approach gives very little insight into the description of the blow up mechanism and the description of the flow even just near the ground state soliton $Q$ is still only at its beginning.


\subsection{On the energy critical wave map problem}


There is an important litterature devoted to the construction of blow up solutions for nonlinear wave equations, see e.g. Alinhac \cite{Alinhac}, and Merle and Zaag \cite{Merlezaag1}, \cite{Merlezaag2} for the study of the ODE type of blow up for subcritical nonlinearities. For energy critical problems like \fref{wave}, recent important progress has been made through the study of the two dimensional energy critical corotational wave map to the 2-sphere:
\be
\label{wavemap}
\pa_{tt}u-\pa_{rr}u-\frac{\pa_ru}{r}-\frac{k^2\sin{2u}}{2r^2}=0,
\ee 
where $k\in \Bbb N^*$ is the homotopy number. The ground state is given there by $$Q(r)=2\tan^{-1}(r^k).$$ After the pioneering works by Christodoulou, Tahvildar-Zadeh \cite{ChrsitZadeh}, Shatah and Tahvildar-Zadeh \cite{shattahtahvildar} and Struwe \cite{Struwewm} and their detailed study of the concentration of energy scenario, the first explicit description of singularity formation for the $k=1$ case is derived by Krieger, Schlag and Tataru \cite{KSTinvent} who construct finite energy finite time blow up solutions of the form 
\be
\label{decompuontro}
u(t,x)=(Q+\e)(t,\frac{x}{\lambda(t)}) \ \ \mbox{with} \ \ \|\e(t),\pa_t\e(t)\|_{\dot{H^1}\times L^2}\ll1
\ee 
 with a blow up speed given by $$\lambda(t)=(T-t)^{\nu}, \ \ \forall \ \nu>\frac{3}{2},$$ see also \cite{KSTyangmills}. The spectacular feature of this result is to exhibit arbitrarily slow blow up regimes further and further from self similarity which would correspond to the --forbidden, see \cite{Struwewm}-- self similar law 
\be
\label{selfsimilarlaw}
\lambda(t)\sim T-t.
\ee
Numerics suggest \cite{bizon} that this blow up scenario is non generic and corresponds to finite codimensional manifolds. After the pioneering works \cite{RodSter} for large homotopy number $k\geq 4$, Rapha\"el and Rodnianski \cite{RaphRod} give a complete description of a {\it stable} blow up dynamics which originates from smooth data and for all homotopy number $k\geq 1$. The blow up speed obeys in this regime a {\it universal law which depends in an essential way on the rate of convergence} of the ground state $Q$ to its asymptotic value $$\pi-Q\sim \frac{1}{r^k} \ \ \mbox{as} \ \ r\to\infty,$$ and indeed the stable blow up regime corresponds to a decomposition \fref{decompuontro} with the blow up speed 
\be
\label{reatewm}
\lambda (t)\sim \left\{\begin{array}{ll} c_k\frac{T-t}{|\log (T-t)|^{\frac{1}{2k-2}}} \ \ \mbox{for} \ \ k\geq 2,\\ (T-t)e^{-\sqrt{|\log (T-t)|}} \ \ \mbox{for} \ \ k=1.\end{array}\right .
\ee
Note that this work draws an important analogy with another critical problem, the $L^2$ critical nonlinear Schr\"odinger equation, where a similar universality of the stable singularity formation near the ground state is proved by Merle and Rapha\"el in the series of papers \cite{MR1}, \cite{MR2}, \cite{R1}, \cite{MR3}, \cite{MR4}, \cite{MR5}.


\subsection{Statement of the result}


For the power nonlinearity energy critical problem  \fref{wave}, there has been recent progress towards the understanding of the flow near the solitary wave $Q$.  In \cite{KSwave}, Krieger and Schlag  construct in dimension $N=3$ a codimension one manifold of initial data near $Q$ which yield global solutions asymptotically converging to the soliton manifold. The strategy  developed  by Krieger, Schlag, Tataru in \cite{KSTinvent} for the wave map problem has been adapted in \cite{KSTwave} to show in dimension $N=3$ the existence of finite energy finite time blow up solutions of the form $$u(t,x)=\frac{1}{\lambda^{\frac{N-2}{2}}(t)}(Q+\e)(t,\frac{x}{\lambda(t)}) \ \ \mbox{with} \ \ \|\e(t),\pa_t\e(t)\|_{\dot{H^1}\times L^2}\ll1 $$ and with a blow up speed given by
\be
\label{continiohf}
\lambda(t)=(T-t)^{\nu}, \ \ \forall \ \nu>\frac{3}{2}.
\ee
The quantization of the energy at blow up for small type II blow up solutions in dimension $N\in \{3,5\}$ is proved in \cite{DMKquanta}, \cite{DMKquantabis} in the radial and non radial cases. In particular, for radial data, if $T<+\infty$ and $$\sup_{t\in [0,T]}\left[|\nabla u(t)|^2_{L^2}+\pa_tu|^2_{L^2}\right]\leq |\nabla Q|_{L^2}^2+\alpha^*, \ \ \alpha^*\ll 1, $$ then there exists a dilation parameter $\lambda(t)\to 0$ as $t\to T$ and asympotic profiles $(u^*,v^*)\in H^1\times L^2$ such that 
$$\left(u(t,x)- \frac{1}{\lambda^{\frac{N-2}{2}}(t)}Q(\frac{x}{\lambda(t)}),\partial_tu(t)\right)\to (u^*,v^*) \ \ \mbox{in} \ \ \dot{H}^1\times L^2\ \ \mbox{as} \ \ t\to T,
$$ 
see \cite{MR5} for related classification results for the $L^2$ critical (NLS).\\
These works however leave open the question of the existence of {\it smooth} type II blow up solutions. We claim that such smooth type II blow up solutions can be constructed in dimension $N=4$ as the formal analogue of the singular dynamics exhibited by Rapha\"el and Rodnianski \cite{RaphRod} for the wave map problem in the least homotopy number class $k=1$. The following theorem is the main result of this paper:

\begin{theorem}[Existence of smooth type II blow up solutions in dimension $N=4$]
\label{thmmain}
Let $N=4$. Then for all $\alpha^*>0$, there exist $\mathcal C^{\infty}$ initial data $(u_0,u_1)$ with $$E(u_0,u_1)<E(Q,0)+\alpha^* $$ such that the corresponding solution to the energy critical focusing wave equation \fref{wave} blows up in finite time $T=T(u_0,u_1)<+\infty$ in a type II regime according to the following dynamics: there exist $(u^*,v^*)\in \dot{H}^1\times L^2$ such that 
\be
\label{energyquantization}
\left(u(t,x)- \frac{1}{\lambda^{\frac{N-2}{2}}(t)}Q(\frac{x}{\lambda(t)}),\partial_tu(t)\right)\to (u^*,v^*) \ \ \mbox{in} \ \ \dot{H}^1\times L^2\ \ \mbox{as} \ \ t\to T
\ee 
with a blow up speed given by \be
\label{lawlambda}
\lambda(t)=(T-t)e^{-\sqrt{|\log (T-t)|}(1+o(1))} \ \ \mbox{as} \ \ t\to T.
\ee
\end{theorem}

\medskip

\noindent{\it Comments on the result}\\
{\it 1. On the smoothness of the initial data}: An important feature of Theorem \ref{thmmain} is to exhibit a new blow up speed which is valid for {\em $\mathcal C^{\infty}$} solutions. Indeed, while the Krieger, Schlag, Tataru \cite{KSTwave} approach provides a continuum of blow up speeds, the exact regularity of the obtained solutions is not known, which is an unpleasant consequence of their construction scheme. In fact, it is expected that $\mathcal C^{\infty}$ initial data should lead to {\it quantize} blow up rates hence breaking the continuum of blow up speeds \fref{continiohf},  we refer to \cite{papierheat} for a related discussion in the context of the energy critical harmonic heat flow. Hence we expect the blow up rate \fref{lawlambda} to correspond to the minimal type II blow up speed of smooth solutions with small super critical energy. Such a general lower bound on blow up rate in the spirit of the one obtained by Merle and Raphael for the $L^2$ critical NLS \cite{R1}, \cite{MR4} is an open problem. The construction of excited blow up solutions with other speeds and $\mathcal C^{\infty}$ regularity also remains to be done. This problematic is related to the understanding of the structure of the flow near $Q$ which is still at its beginning.\\ 

{\it 2. On the codimension one manifold}:  The proof of Theorem \ref{thmmain} involves a detailed description of the set of initial data leading to the type II blow up with speed \fref{lawlambda}. Indeed, given a small enough parameter $b_0>0$ and a suitable deformation $Q_b$ of the soliton with $$Q_{b_0}\to Q \ \ \mbox{as} \ \ b_0\to 0$$ in some strong sense, we show that for any smooth and radially symmetric excess of energy $$\|\eta_0,\eta_1\|_{H^2\times H^1}\lesssim   \frac{b_0^2}{|\log(b_0)|},$$ we can find $d_+(b_0,\eta_0,\eta_1)\in \Bbb R$ such that the solution to \fref{wave} with initial data $$u_0=Q_{b_0}+\eta_0+d_+\psi, \ \ u_1=b_0\left(\frac{N-2}{2} Q_{b_0}+y\cdot\nabla Q_{b_0}\right)+\eta_1$$ blows up in finite time in the regime described by Theorem \ref{thmmain}. Here $\psi$ is the bound state of the linearized operator close to $Q$ and generates the unstable mode, we refer to Definition \ref{defexittime} and Proposition \ref{propcle} for precise statements. Hence the set of blow up solutions we construct live on a codimension one manifold in the radial class in some weak sense. Following \cite{KSwave},  \cite{KSnls}, the proof that this set is indeed a codimension one manifold relies on proving some Lipschitz regularity of the map $(b_0,\eta_0,\eta_1)\to d_+(b_0,\eta_0,\eta_1)$, and in particular some local uniqueness to begin with. The analysis in \cite{KSnls} shows that this may be a delicate step in some cases. Our solution is constructed using a soft continuous topological argument of Brouwer type coupled with suitable monotonicity properties in the spirit of Cote, Marte and Merle \cite{cotemartelmerle}, and in other related settings, see e.g. Martel \cite{martelmulti}, Rapha\"el and Szeftel \cite{RaphSzef}, this strategy has proved to be quite powerful to eventually achieve strong uniqueness results. This interesting question in our setting will require additional efforts and needs to be adressed separately in details.\\

{\it 3. Extension to higher dimensions}: We focus onto the case of dimension $N=4$ for the sake of simplicity, and our main objective is to provide a robust framework to construct $\mathcal C^{\infty}$ type II blow up solutions. However, following the heurisitic  developed  in \cite{RaphRod}, the blow up speed \fref{lawlambda} corresponds to the $k=1$ case in \fref{reatewm}, and we similarly conjecture in dimension $N\geq 5$ the existence type II finite type blow up solutions close to $Q$ with blow up speed $$\lambda(t)\sim c_N\frac{T-t}{|\log (T-t)|^{\frac{1}{N-4}}}.$$ Note from \fref{defqexplicite} that the higher the dimension, the fastest the decay of the ground state $Q$, and this should avoid some difficulties which occur only in low dimension like in \cite{RaphRod} for large homotopy number $k\geq 4$.  We expect the strategy  developed  in this paper to carry over to the case $N=5,6$, but the extension to large dimension will be confronted in particular to the difficulty of the lack of smoothness of the nonlinearity. Let us also insist onto the fact that the case $N=4$ is in many ways the more delicate one in terms of the strong coupling of the main part of the solution and the outgoing tail due to the slow decay of $Q$, which results in the somewhat pathological blow up speed \fref{lawlambda}. This comment becomes even more dramatic in dimension $N=3$ where we expect our analyis to be applicable to construct $\mathcal C^{\infty}$ type II blow up solutions, but this seems to require a slightly different approach.


\subsection{Strategy of the proof}


Let us briefly summarize the strategy of the proof of Theorem \ref{thmmain}. \\

{\bf step 1} Approximate self similar solution.\\

Let $D,\Lambda$ denote the differential operators \fref{defdl}. Exact self similar solutions to \fref{wave} of the form $$u(t,x)=\frac{1}{\lambda^{\frac{N-2}{2}}(t)}Q_{b}\left(\frac{x}{\lambda(t)}\right) \ \ \mbox{with} \ \ b=-\lambda_t$$  where $Q_b$ satisfies the self similar equation 
\be
\label{voehoeigr}
\Delta Q_b-b^2D\Lambda Q_b+Q_b^3=0
\ee 
are known to develop a singularity on the light cone $y=\frac{T-t}{\lambda(t)}=\frac{1}{b}$ leading to an unbounded Dirichlet energy $\|\nabla Q_b\|_{L^2}=+\infty$, see Kavian, Weissler \cite{KaWe}. We therefore assume $0<b\ll 1$ and consider a one term expansion approximation $$Q_b=Q+b^2T_1$$ which injected into \fref{voehoeigr} yields at the order $b^2$: 
\be
\label{eqtoneintro}
HT_1=-D\Lambda Q.
\ee
Here $H$ is the linearized operator close to Q given by 
\be
\label{defhclosetoq}
H=-\Delta -\frac{N+2}{N-2}Q^{\frac{4}{N-2}}.
\ee
The spectral structure of $H$ is well knwon in connection to the fact that $Q$ is an extremizer of the Sobolev embedding $\dot{H}^1\to L^{2^*}$, and in the radial sector, $H$ admits one non positive eigenvalue with well localized eigenvector $\psi$:
\be
\label{defmupsi}
H\psi=-\zeta\psi, \ \ \zeta>0,
\ee
and a {\it resonance} at the boundary of the continuum spectrum generated by the scaling invariance of \fref{wave}:
\be
\label{resonance}
H(\Lambda Q)=0, \ \ \Lambda Q(r)\sim \frac{C}{r^{N-2}}\ \  \mbox{as} \ \ r\to +\infty.
\ee
In order to solve \fref{eqtoneintro}, we first remove the leading order growth in the exact solution $T_1=\frac{1}{4}|y|^2Q$ which is consequence of the flux computation: 
\be
\label{fouxcomputation}
(D\Lambda Q,\Lambda Q)=\frac12\lim_{y\to +\infty}y^4|\Lambda Q|^2>0
\ee due to the slow decay of $Q$ in dimension $N=4$ from \fref{defqexplicite}. For this, we solve $$HT_1=-D\Lambda Q+c_b\Lambda Q{\bf 1}_{y\leq \frac1b} \ \ \mbox{with} \ \ c_b=\frac{(D\Lambda Q,\Lambda Q)}{\int_{y\leq \frac1b}|\Lambda Q|^2}\sim\frac{1}{2|\log b|} \ \ \mbox{as} \ \ b\to 0.$$ The purpose of this construction is to yield after a suitable localization  process an $o(b^2)$ approximate solution to the self similar equation \fref{voehoeigr} which dominant term near and past the light cone is still given by $Q$ itself in the sense that: $$b^2|T_1|\ll Q \ \ \mbox{for} \ \ y\geq \frac1b.$$ This identifies $Q$ as the leading order radiation term\footnote{see \cite{RaphRod} for a further discussion on this issue and the role played by the non vanishing Pohozaev integration \fref{fouxcomputation}}.\\

{\bf step 2} Bootstrap estimates.\\

We now roughly consider initial data of the form 
\be
\label{intiialdata}
u_0=Q_{b_0}+d_+\psi+\eta_0,\ \ u_1=b_0\Lambda Q_{b_0}+\eta_1, \ \text{with} \ |d_+|+\|\eta_0,\eta_1\|_{H^2\times H^1}\ll b_0^2,
\ee and introduce a modulated decomposition of the flow $$u(t,x)=\frac{1}{\lambda^{\frac{N-2}{2}}(t)}(Q_{b(t)}+\e)\left(t,\frac{x}{\lambda(t)}\right), \ \ b(t)=-\lambda_t.$$ Here we face the major difference between the power nonlinearity wave equation \fref{wave} and the critical wave map problem \fref{wavemap} which is the presense of a negative eigenvalue in the first case \fref{defmupsi} for the linearized operator $H$ close to $Q$. This induces an instability in the modulation equations for $b,\lambda$ which is absent in the wave map case, leading to stable blow up dynamics. However, we claim that the ODE type instability generated by \fref{defmupsi} is the {\it only} instability mechanism.\\
The situation is conceptually similar to the one studied in \cite{cotemartelmerle} where multisolitary wave solutions are constructed in the supercritical regime despite the presence of exponentially growing modes for the linearized operator which are absent in the subcritical regime. We adapt a similar scheme of proof which does not rely on a fixed point argument to solve the problem from infinity in time\footnote{after renormalization of the time}, but by directly following the flow for {\it any} initial data of the form \fref{intiialdata}. This reduces the full problem to a one dimensional dynamical system for which a classical clever continuity argument yields the existence of $d_+(b_0,\eta_0,\eta_1)$ such that the unstable mode is extinct, see section \ref{sectionfour}.\\
The key is hence to control the flow under the a priori control of the unstable mode, and here we adapt the technology  developed  in \cite{RaphRod} which relies on {\it monotonicity properties} of the linearized Hamiltonian at the $H^2$ level of regularity. However, the analysis in \cite{RaphRod} heavily relies on the existence of a decomposition of the Hamiltonian $$H=A^*A, \ \ A=-\partial_y+V(y)$$ which is central in the proof of the main monotonicity property and is lost in our setting. This forces us to revisit the approach in several ways, and to rely in particular on fine algebraic properties of the flow\footnote{see in particular \fref{signone}, \fref{signtwo}} near $Q$ and coercitivity properties of suitable quadratic forms in the spirit of \cite{MartelMerlekdv}, \cite{MR1}, see Lemma \ref{lemmacoercitivity},  which remarkably turn out to be almost explicit thanks to the formula \fref{defqexplicite}. We are eventually able to find $d_+(b_0,\eta_0,\eta_1)$ for which to leading order $$b_s\sim -c_bb^2\sim-\frac{b^2}{2|\log b|}, \ \  b=-\lambda_t, \ \ \frac{ds}{dt}=\frac{1}{\lambda}, \ \ |d_+|+\|\pa_{yy}\e\|_{L^2}\ll b^2$$ which reintegration in time yields finite time blow up in the regime described by Theorem \ref{thmmain}.


\subsection{Notations} 


We define the differential operators: 
\be
\label{defdl}
\Lambda f=\frac{N-2}{2}f+y\cdot\nabla f \ \ (\mbox{$\dot{H}^1$ scaling}), \ \ Df=\frac{N}{2}f+y\cdot\nabla f\ \  (\mbox{$L^2$ scaling}).
\ee
Denoting  $$(f,g)=\int fg=\int_{0}^{+\infty}f(r)g(r)r^{N-1}dr$$ the $L^2(\RR^N)$ radial inner product, we observe the integration by parts formula:
\be
\label{adjoinctionfrimula}
(Df,g)=-(f,Dg), \ \ (\Lambda f,g)+(\Lambda g,f)=-2(f,g).
\ee
Given $f$ and $\lambda>0$, we shall denote: $$f_{\lambda}(t,r)=\frac{1}{\lambda^{\frac{N-2}{2}}}f\left(t,\frac{r}{\lambda}\right),$$ 
and the space rescaled variable will always be denoted by $$y=\frac{r}{\lambda}.$$ We let $\chi$ be a smooth positive radial cut off function $\chi(r)=1$ for $r\leq 1$ and $\chi(r)=0$ for $r\geq 2$. For a given parameter $B>0$, we let 
\be
\label{defchib}
\chi_B(r)=\chi\left(\frac{r}{B}\right).
\ee
Given $b>0$, we set
\be
\label{defbnot}
 B_0=\frac{2}{b}, \quad B_1=\frac{|\log b|}{b}.
\ee

To clarify the exposition we use the notation $a \lesssim b$ when there exists a constant $C$ with no relevant dependency on $(a,b)$  such that $a \leq C b.$ In particular, we do not allow constants $C$ to depend on the parameter $M$ except in Appendix A.

{\bf Aknowledegments}: The authors would like to thank Igor Rodnianski for stimulating discussions about this work. P.R is supported by the French ANR Jeune Chercheur SWAP.
	

\section{Computation of the modified self-similar profile}
\label{sectiontwo}


This section is devoted to the construction of an approximate self-similar solution $Q_b$ which describes the dominant part of the blow up profile inside the backward light cone from the singular point $(0,T)$ and displays a slow decay at infinity which is eventually responsible for the modifications to the blow up speed with respect to the self similar law. The key to this construction is the fact that the structure of the linearized operator $H$ close to $Q$ is completely explicit in the radial sector thanks to the explicit formulas at hand for the elements of the kernel.\\

We introduce the direction 
\be
\label{defphi}
\Phi=D\Lambda Q 
\ee
which displays the cancellation
\be
\label{cancelaationphi}|\Phi(y)|\lesssim \frac{1}{1+y^4}
\ee
and the crucial nondegeneracy which follows from the Pohozaev integration by parts formula:
\be
\label{degenfond}
(\Phi,\Lambda Q)=\lim_{y\to +\infty}\left(\frac12 y^4|\Lambda Q|^2\right)=32>0.
\ee

\begin{proposition}[Approximate self-similar solution]
\label{propqb}
Let $M$ denote a large enough constant. Then there exists $b^*(M)>0$ small enough such that for all $0<b<b^*(M)$, there exists a smooth radially symmetric profile $T_1$ satisfying the orthogonality condition \be
\label{orthotone}
(T_1,\chi_M\Phi)=0
\ee
such that 
\be
\label{defpb}
P_{B_1}=Q+\chi_{B_1}b^2T_1
\ee is an approximate self similar solution in the following sense. Let 
\be
\label{defpsib}
\Psi_{B_1}=-\Delta P_{B_1}+b^2D\Lambda P_{B_1}-f(P_{B_1}),
\ee 
then for all  $k \geq 0, 0\leq y\leq \frac{1}{b^2}$,
\be
\label{esttone}
\left|  \dfrac{d^k T_1}{dy^k}(y) \right|
\lesssim \frac{1}{1+y^k}\left[\frac{1+|\log(by)|}{|\log b|}{\bf 1}_{2\leq y\leq \frac{B_0}{2}}+\frac{1}{b^2y^2|\log b|}{\bf 1}_{y\geq \frac{B_0}{2}}+\frac{\log (M)+|\log (1+y)| }{1+y^2}
\right],
\ee
\be
\label{estpbcut}
\left|\frac{d^k}{dy^k}\frac{\partial P_{B_1}}{\partial b}\right|  \lesssim  \frac{b{\bf 1}_{ y\leq 2 B_1}}{1+y^k}\left[\frac{1+|\log(by)|}{|\log b|}{\bf 1}_{2\leq y\leq \frac{B_0}{2}}+\frac{1}{b^2y^2|\log b|}{\bf 1}_{y \geq \frac{B_0}{2}}+\frac{\log (M)+|\log (1+ y)|}{1+y^2}
\right],
\ee
and, {for all $k \geq 0,$ $y\geq 0$,}
\begin{multline}
\label{estpsibloin} 
\left| \frac{d^k}{dy^k}(\Psi_{B_1} - c_b b^2\chi_{\frac{B_0}{4}} \Lambda Q) \right| \\[4pt] 
\begin{array}{cl}
\lesssim & \dfrac{b^4}{1+y^k}\left[\dfrac{1+|\log(by)|}{|\log b|}{\bf 1}_{2\leq y\leq \frac{B_0}{2}}+\dfrac{1}{b^2y^2|\log b|}{\bf 1}_{2B_1 \geq y\geq \frac{B_0}{2}}+\dfrac{\log(M)+|\log (1+ y)|}{1+y^2}{\bf 1}_{ y\leq 2 B_1}\right] \\[8pt]
+& \dfrac{b^2}{(1+y^{4+k})}{\bf 1}_{ y\geq  B_1/2},  
\end{array}
\end{multline}

for some constant 
\be
\label{computationcb}
c_b=\frac{1}{2|\log b|}\left(1+O\left(\frac{1}{|\log b|}\right)\right).
\ee	
\end{proposition}

{\bf Proof of Proposition \ref{propqb}}\\

\noindent{\bf Step 1} Inversion of $H$.\\

The first green function of $H$ is given from scaling invariance by 
\be
\label{deflambdaq}
\Lambda Q(y)= \frac{N-2}{2\left(1+ \frac{y^2}{N(N-2)}\right)^{\frac{N}{2}}}\left(1 - \frac{y^2}{N(N-2)}\right), 
\ee which admits the following asymptotics:
\be
\label{ezpansionLQ}
\forall \ k\geq 0, \ \ \frac{d^k (\Lambda Q)}{dy^k}(y)=\left\{\begin{array}{lll}
O(1) & \mbox{as} \ \  y \to 0,\\ O(y^{-(N-2 + k)})  & \mbox{as} \ \ y \to \infty
\end{array}
\right .
\ee
Let now 
$$
\Gamma(y)= - \Lambda Q (y)\int_{1}^y\frac{ds}{s^{N-1}(\Lambda Q)^2(s)},
$$ 
be another (singular at the origin\footnote{Note that $\Gamma$ must be smooth at $y=\sqrt{N(N-2)}$ where $\Lambda Q$ vanishes from the radial ODE $H\Gamma=0$}) element of the kernel of $H$ which can be found from the Wronskian relation: 
$$\Gamma' \Lambda Q- \Gamma (\Lambda Q)'=\frac{-1}{y^{N-1}}.
$$
From this we easily find the asymptotics of $\Gamma^{(k)}$ for any integer $k$: 
\be
\label{ezpansionH}
\dfrac{d^k \Gamma}{dy^k}(y) =\left\{ \begin{array}{ll} O(y^{-(N -2 + k)}) \ \ \mbox{as}  \ \ y \to 0\\
 O(y^{-k}) \ \  \mbox{as} \ \ y\to \infty.	
\end{array}
\right .
\ee
A smooth solution to $Hw=F$ is given by: 
 \be
 \label{defulinearsolver}
 w(y)=\Gamma(y)\int_0^yF(s)\Lambda Q(s)s^{N-1}ds - \Lambda Q(y)\int_{0}^yF(s)\Gamma(s)s^{N-1}ds.
 \ee
 
 We now look for  a solution to the self similar equation in the form $Q+b^2T_1.$ This yields:
 \bea
 \label{computationpsib}
  \Psi_b& =&-\Delta Q_b+b^2D\Lambda Q_b-f(Q_b)\\[4pt]
 \nonumber & = & b^2(HT_1+D\Lambda Q)+b^4D\Lambda T_1-\left[f(Q+b^2T_1)-f(Q)-b^2f'(Q)T_1\right].
 \eea

 \noindent{\bf Step 2} Computation of $T_1$.\\
 
Thanks to the anomalous decay \fref{cancelaationphi}, we chose $T_1$ solution to 
 \be
 \left\{
 \begin{array}{rcl}
 \label{eqtone}
 HT_1 & = & F=-D\Lambda Q+c_b\chi_{\frac{B_0}{4}}\Lambda Q, \\ 
(T_1,\chi_M \Phi) &=& 0,
\end{array}
\right.
 \ee
 with $c_b$ chosen such that: 
 \be
 \label{eqcb}
 (F,\Lambda Q)=0
 \ee \emph{i.e.} from Pohozaev integration by parts formula,
  see \eqref{defbnot} and \eqref{degenfond} ,
 \bee
 \nonumber c_b & = & \frac{(D\Lambda Q,\Lambda Q)}{(\chi_{\frac{B_0}{4}}\Lambda Q,\Lambda Q)}=  \frac{1}{2}\frac{\lim_{y\to+\infty}y^4|\Lambda Q(y)|^2}{\int\chi_{\frac{B_0}{4}}|\Lambda Q|^2}\\
 & = & \frac{1}{2|\log b|}\left(1+O\left(\frac{1}{|\log b|}\right)\right) \ \ \mbox{as} \ \ b\to 0.
 \eee
 This yields \fref{computationcb}. Following \fref{defulinearsolver}, we first consider
 \be
 \label{defttilde}
 \tilde{T}_1(y)=\Gamma(y)\int_0^yF(s)\Lambda Q(s)s^{3}ds - \Lambda Q(y)\int_{0}^yF(s)\Gamma(s)s^{3}ds 
 \ee 
The  smoothness of $\tilde{T}_1$ at the origin follows from \fref{defttilde} together with elliptic regularity from \fref{eqtone}. We now examine the behavior of $\tilde{T}_1$ at large $y$. 

We first observe that, from the orthogonality \fref{eqcb}: 
$$
\tilde{T}_1(y)  = -\left[ \Gamma(y)\int_y^{+\infty}F(s)\Lambda Q(s)s^{3}ds +\Lambda Q(y)\int_{0}^yF(s)\Gamma(s)s^{3}ds\right]
$$

Hence, from the degeneracy $|D\Lambda Q|=O(y^{-4}),$ this yields that, for $\frac{B_0}{2}\leq y\leq \frac{1}{b^2}$:

\bea
\label{boundttilde}
|\tilde{T}_1(y)|\nonumber & \lesssim & \int_y^{+\infty}\frac{s^3}{(1+s^4)(1+s^2)}ds+\frac{1}{y^2}\left[\int_0^y\frac{1+s^3}{1+s^4}ds+|c_b|\int_0^{B_0}\frac{s^3}{1+s^2}ds\right]\\
& \lesssim & \frac{|\log (1+y)|}{ 1 + y^2}+\frac{1}{b^2y^2|\log b|}.
\eea
similarly, for $1\leq y\leq \frac{B_0}{2}$, 
\bea
\label{boundttildebis}
\nonumber |\tilde{T}_1(y)| & = & \left|\Gamma(y)\int_y^{+\infty}F(s)\Lambda Q(s)s^{3}ds +\Lambda Q(y)\int_{0}^yF(s)\Gamma(s)s^{3}ds\right|\\
\nonumber & \lesssim & \int_y^{+\infty}\frac{s^3}{(1+s^4)(1+s^2)}ds+|c_b|\int_y^{B_0} \frac{s^3}{(1+s^2)^2}ds\\
\nonumber & + & \frac{1}{1+y^2}\left[\int_0^y\frac{s^3}{1+s^4}ds+|c_b|\int_0^y\frac{s^3}{1+s^2}ds\right]\\
& \lesssim & \frac{1+|\log (by)|}{|\log b|}+\frac{|\log (1+y)|}{1+y^2}.
\eea
We now choose thanks to \fref{degenfond}:
 $$ T_1(y)=\tilde{T}_1(y)-c\Lambda Q  \ \ \mbox{with} \ \  c=\frac{(\tilde{T}_1,\chi_M\Phi)}{(\chi_M\Phi,\Lambda Q)} $$
 so that the orthogonality condition \fref{orthotone} is fulfilled. 
 We note that so that the bounds \fref{boundttilde} and \fref{boundttildebis} 
 ensure that $c$ remains bounded by $\log(M)$ uniformly in $M$ and $b,$ provided $b$ is chosen sufficiently small w.r.t. $M.$  
 
This yields \fref{esttone} for $k=0$, the other cases follow similarly.\\

 {\bf Step 4} Estimate on $\Psi_{B_1}$ and $\partial_b \Psi_{B_1}$\\
 
 We now cut off the slow decaying tail $T_1$ according to \fref{defpb} and estimate the corresponding error to self similarity $\Psi_{B_1}$ given by \fref{defpsib}. 
 
 We compute: 
 
 \bee
 \Psi_{B_1}  &= & b^2\chi_{B_1}(HT_1+D\Lambda Q)\\
 	&& \ + \ b^2\left[-2\chi_{B_1}'T'_1-T_1\Delta \chi_{B_1}+(1-\chi_{B_1})D\Lambda Q+b^2D\Lambda(\chi_{B_1}T_1)\right]\\
 && \ - \ \left[  f(Q+b^2\chi_{B_1}T_1)-f(Q)-\chi_{B_1}f'(Q)T_1\right].
\eee 

Outside the support of $\chi_{B_1}$ we have thus $\Psi_{B_1} = b^2 D\Lambda Q.$ 
On the other hand, in dimension $N=4$, we have the Taylor expansion :
 $$
 f(Q+b^2\chi_{B_1}T_1)-f(Q)-\chi_{B_1}f'(Q)T_1 = b^4\chi_{B_1}^2T^2_1(y) \int_{0}^1 (1-\tau) (Q(y) + \tau b^2 \chi_{B_1}T_1(y)) d\tau.
 $$

We thus estimate from \fref{esttone}, \fref{computationpsib}, \fref{eqtone} and the degeneracy \fref{cancelaationphi}  for $y\leq 2B_1:$
 \begin{multline*}
 \left|\Psi_{B_1} - b^2 c_b  \chi_{\frac{B_0}{4}} \Lambda Q \right| 
  \lesssim  b^2\mathbf{1}_{y\geq B_1/2}\left( \dfrac{T'_1}{1+y} + \dfrac{T_1}{1+y^2} + \dfrac{1}{1+y^4}\right)\\
   +  b^4  |D\Lambda (\chi_{B_1}T_1)|+b^4|T^2_1(y)|\int_{0}^1 (1-\tau) |Q(y) + \tau b^2 T_1(y)| d\tau.
 \end{multline*}
\fref{esttone} now yields \fref{estpsibloin} for $k=0.$
 Further derivatives are estimated similarly thanks to the smoothness of the nonlinearity. 
We emphasize here that, given $B>0$ large, we have $ 1/(1+y) \lesssim 1/B \lesssim 1/(1+y)$
 on the support of $\chi'_{{B}}$, so that differentiating $\chi_{B}$ acts as a multiplication by $1/(1+y).$ Furthermore, there holds $1/B_1 = o(b)$ so that we can always
 dominate $1/(1+y)$ by $b$ on the support of $\chi'_{B_1}.$

Finally, we compute $\partial_b P_{B_1}$ from \fref{defpb}.
 
To this end, we note that $\partial_b c_b = O(1/b|\log(b)|^2)$ when $b\to 0$ so that the source term for $T_1$ 
in \eqref{eqtone} satisfies 
$$
\partial_b F = \left[ O\left(\frac{1}{b |\log b|}\right) \chi_{B_0/4} + O\left(\frac{1}{b|\log b|}\right) \rho_{B_0/4} \right] \Lambda Q
$$
where $\rho(z) = z \chi'(z) \in \mathcal{C}^{\infty}_c(0,\infty)$ and we keep the convention for function dilation.
Hence, the same arguments as for ${T}_1$ enable to show that $\partial_b \tilde{T}_1$ and then $\partial_b T_1$ satisfy the estimates:
\be
\label{esttbine}
\left|  \dfrac{d^k \partial_b T_1}{dy^k}(y) \right|
\lesssim \frac{1}{b(1+y^k)}\left[\frac{1+|\log(by)|}{|\log b|}{\bf 1}_{2\leq y\leq \frac{B_0}{2}}+\frac{1}{b^2y^2|\log b|}{\bf 1}_{y\geq \frac{B_0}{2}}+\frac{1+|\log (1+y)|}{1+y^2}\right].
\ee
Finally, we compute from \eqref{defpb}
\be\label{decompdbT1}
\partial_b P_{B_1}=2b\chi_{B_1}T_1+b^2 \partial_b \log(B_1) \rho_{B_1} T_1+ b^2 \chi_{B_1} \partial_b T_1.
\ee
This decomposition together with \fref{esttone} and the previous computation yield \fref{estpbcut}.

\noindent This concludes the proof of Proposition \ref{propqb}.


\section{Description of the trapped regime}
\label{sectionthree}


We display in this section the regime which leads to the blow up dynamics described by Theorem \ref{thmmain}.\\


\subsection{Modulation of solutions to \eqref{wave}}


Let us start with describing the set of  solutions among which the finite time blow up scenario described by Theorem \ref{thmmain} is likely to arise. 
We recall from \fref{defmupsi} that $\psi$ denotes the bound state of $H$ with eigenvalue $-\zeta<0$. 
The following lemma is a standard consequence of the implicit function theorem and the smoothness of the flow, see Appendix \ref{app_modulation}.

\begin{lemma} [Modulation theory]
\label{definitionadmissible}
Let $M$ be a sufficiently large constant to be chosen later and $0<b_0<b_0^*(M)$ small enough. Let $(\eta_0,\eta_1,d_+)$ satisfying the smallness
condition:
\be
\label{defmotionintia}
|d_+|+\|\eta_0,\nabla \eta_0,\eta_1+b_0(1-\chi_{B_1(b_0)})\Lambda Q,\nabla \eta_1\|_{\dot{H}^1\times\dot{H^1}\times L^2\times L^2}\lesssim \frac{b_0^2}{|\log b_0|},
\ee
then, there exists a time $T_0$ such that the unique solution $u \in \mathcal{C}^{2}([0,T_0];L^2(\mathbb R^N)) \cap \mathcal{C}([0,T_0]; H^2(\mathbb R^N))$
to \eqref{wave} with initial data :
\be
\label{intiialdtata}
u_0=P_{B_1(b_0)}+\eta_0+d_+\psi,\ \ u_1=b_0\Lambda P_{B_1(b_0)}+\eta_1,
\ee
 admits  on $[0,T_0]$ a unique decomposition 
\be
\label{decompuintro}
u(t)=(P_{B_1(b(t))}+\e(t))_{\lambda(t)}
\ee 
with \\
{\em 1.} $\lambda\in \mathcal C^2([0,T_0],\Bbb R^*_+)$ with  
\be
\label{orthoeetb}
\forall \ t\in [0,T_0], \ \ (\e(t),\chi_M\Phi)=0\ \ \mbox{and} \ \ b(t) = - \lambda_t ;\ \ 
\ee
{\em 2.} there holds the smallness:
\be
\label{initialbounds}
\|\nabla \varepsilon(t)\|_{L^2} \lesssim  b_0|\log b_0|\qquad  
|b(t) - b_0| + |\lambda(t)-1|+ \|\nabla^2 \varepsilon(t) \|_{L^2} \lesssim \dfrac{b_0^2}{|\log b_0|} \quad \forall \ t \in [0,T_0].
\ee
\end{lemma}

\begin{remark} Recall that the slow decay of $Q$ and the choice of $P_{B_1}$ induces an unbounded tail of $P_{B_1}$ in the energy norm, and more specifically $\|\Lambda Q\|_{L^2}=+\infty$, hence the need for the compensation in the norm for the time derivative in \fref{defmotionintia}.
\end{remark}


\subsection{Decomposition of the flow and modulation equations}
\label{geomdecomp}


Considering initial data satisfying the assumption of the above lemma, we now write the evolution equation induced by \fref{wave} in terms of the decomposition \fref{decompuintro}. 
Let 
\be
\label{defwou}
u(t,r)=\frac{1}{[\lambda(t)]^{\frac{N}{2}-1}}\left(P_{B_1(b(t))}+\e\right)(t,\frac{r}{\lambda(t)})=\left(P_{B_1(b(t))}\right)_{\lambda(t)}+w(t,r)
\ee 
where $b=-\lambda_t.$ Let us derive the equations for $w$ and $\e$. Let
\be
\label{defreacledtime}
s(t)=\int_0^t\frac{d\tau}{\lambda(\tau)} \ee
be the rescaled time. We shall make an intensive use of the following rescaling formulas: 
\be
\label{formularesacoling}
u(t,r)=\dfrac{1}{\lambda^{N/2-1}}v(s,y), \ \ y=\frac{r}{\lambda}, \ \ \frac{ds}{dt}=\frac{1}{\lambda},
\ee
\be
\label{dliatationone}
\partial_tu=\frac{1}{\lambda}\left(\partial_sv+b\Lambda v\right)_{\lambda},
\ee
\be
\label{dilationtowo}
\partial_{tt}u=\frac{1}{\lambda^2}\left[\partial_s^2v+b(\partial_sv+2\Lambda\partial_sv)+b^2D\Lambda v+b_s\Lambda v\right]_{\lambda}.
\ee
In particular, we derive from \fref{wave} the equation for $\e$:
\bea
\label{eqeqb}
\nonumber \partial_s^2\eb+H_{B_1}\eb & = & -\Psi_{B_1}-b_s\Lambda P_{B_1}-b(\partial_sP_{B_1}+2\Lambda\partial_sP_{B_1})-\partial^2_sP_{B_1}\\
& - & b(\partial_s\eb+2\Lambda\partial_s\eb)-b_s \Lambda \eb+N(\eb)
\eea
where, implicitly, $B_1 = B_1(b(t))$ and $H_{B_1}$ is the linear operator associated to the profile  $P_{B_1}$
\be
\label{defhbl}
H_{B_1} \eb=-\Delta \eb+b^2D\Lambda\eb-f'(P_{B_1})\eb,
\ee
and the nonlinearity:
\be
\label{defnw}
N(\eb)=f(P_{B_1}+\eb)-f(P_{B_1})-f'(P_{B_1})\eb.
\ee
Alternatively, the equation for $w$ takes the form:
$$\partial_t^2w+\tilde{H}_{B_1} w=-\left[\partial_t^2(P_{B_1})_{\lambda}-\Delta(P_{B_1})_{\lambda}-f((P_{B_1})_{\lambda})\right]+N_{\lambda}(w)$$
with 
\be
\label{defhbltwo}
\tilde{H}_{B_1} w=-\Delta w-f'((P_{B_1})_{\lambda})w,
\ee
\be
\label{defnwbis}
N_{\lambda}(w)=f((P_{B_1})_{\lambda}+w)-f((P_{B_1})_{\lambda})-f'((P_{B_1})_{\lambda})w.
\ee
We then expand using (\ref{dliatationone}), (\ref{dilationtowo}):
\bee
\partial_t^2(P_{B_1})_{\lambda}-\Delta(P_{B_1})_{\lambda}-f((P_{B_1})_{\lambda})& = & \frac{1}{\lambda^2}\left[\partial_{ss}P_{B_1}+b(\partial_sP_{B_1}+2\Lambda\partial_sP_{B_1})+b_s\Lambda P_{B_1}+\Psi_B\right]_{\lambda}\\
& = &  \frac{1}{\lambda^2}\left[b\Lambda\partial_sP_{B_1}+b_s\Lambda P_{B_1}+\Psi_B\right]_{\lambda}+\partial_t\left[\frac{1}{\lambda}(\partial_sP_{B_1})_{\lambda}\right]
\eee
and rewrite the equation for $w$:
\be
\label{eqwfinal}
\partial_t^2w+\tilde{H}_{B_1}w =   -\frac{1}{\lambda^2}\left[b\Lambda\partial_sP_{B_1}+b_s\Lambda P_{B_1}+\Psi_B\right]_{\lambda}-\partial_t\left[\frac{1}{\lambda}(\partial_sP_{B_1})_{\lambda}\right]+N_{\lambda}(w).
\ee
For most of our arguments we prefer to view the linear operator $H_{B_1}$ acting on $w$ in (\ref{eqwfinal}) as a perturbation of the linear operator $H_\lambda$ associated to $Q_\lambda$. Then
\bea
\label{oeioehoe}
\partial_t^2w+H_\lambda w&=&F_{B_1}\\
\nonumber &   = &   -\frac{1}{\lambda^2}\left[b\Lambda\partial_sP_{B_1}+b_s\Lambda P_{B_1}+\Psi_{B_1}\right]_{\lambda}-\partial_t\left[\frac{1}{\lambda}(\partial_s\qbt)_{\lambda}\right] \\
\nonumber & & -   \left[f'(Q_{\lambda})-f'((P_{B_1})_{\lambda})\right]w+N_{\lambda}(w)
\eea
with 
\be
\label{defh}
H_{\lambda} w=-\Delta w+f'(Q_{\lambda})w.
\ee


\subsection{The set of bootstrap estimates}

At first, we fix some notations. We introduce the energy  $\mathcal E(t)$ associated to the Hamiltonian $H_\lambda$: 
\be
\label{poitnwiseboundWbis}
\mathcal E(t)  =  \lambda^2\int\left[(H_{\lambda}\pa_tw,\partial_tw)+(H_{\lambda} w)^2\right].
\ee
Given $\zeta \in (0,\infty)$ the unstable eigenvalue, we set:
\be \label{defvplus}
V_+ =
\left| 
\begin{array}{c}
1 \\
\sqrt{\zeta}
\end{array}
\right. 
\qquad V_- = 
\left| 
\begin{array}{c}
1 \\
-\sqrt{\zeta}
\end{array}
\right. 
\ee
and, we introduce the decomposition of the unstable direction 
\be
\label{defaplus}
\left |\begin{array}{ll}(\e,\psi)\\(\pa_s\e,\psi)\end{array}\right .=\tilde{a}_+(s)V_++\tilde{a}_-(s)V_-
\ee 
Let us denote:
\be
\label{aplusamoins}
\kappa_+(s)=\tilde{a}_+(s)+\frac{b_s}{2\sqrt{\zeta}}(\partial_b P_{B_1},\psi), \ \ \kappa_-(s)=\tilde{a}_-(s)-\frac{b_s}{2\sqrt{\zeta}}(\partial_b P_{B_1},\psi).
\ee
We note that the vectors $V_+,V_-$ given by \fref{defvplus} yield an eingenbasis of 
$$\left(\begin{array}{ll} 0&1\\{\zeta}&0\end{array}\right)$$ 
and hence correspond respectively to the unstable and stable mode of the two dimensional dynamical system 
$$\frac{dY}{ds}=\left(\begin{array}{ll} 0&1\\{\zeta}&0\end{array}\right) Y$$ which to first order in $b$ is verified by the projection onto the unstable mode $(\e,\psi)$, see \fref{defunstbakle}. The deformation term $b_s(\partial_b P_{B_1},\psi)$ in \fref{aplusamoins} is present to handle some possible time oscillations induced by the $\partial_s^2P_{B_1}$ term in the RHS of \fref{eqeqb} which cannot be estimated in absolute value but will be proved to be lower order.\\

With these conventions,  we may now paramaterize the set of initial data described by Lemma \ref{definitionadmissible} by $a_+=\kappa_+(0)$, and then reformulate the initial smallness properties in terms of suitable initial bounds for $\e$, see Appendix \ref{app_modulation} for the proof which is standard.

\begin{lemma}[Inital parametrization of the unstable mode and initial bounds]
 \label{smalldata}
Let $M$ and $b_0$ be given as in Lemma \ref{definitionadmissible} and denote by $C(M)$
a sufficiently large constant. Then,  given $(\eta_0,\eta_1,a_+)$ satisfying 
\be
\label{defmotionintia2}
|a_+|+\|\eta_0,\nabla \eta_0,\eta_1+b_0(1-\chi_{B_1(b_0)})\Lambda Q,\nabla \eta_1\|_{\dot{H}^1\times\dot{H^1}\times L^2\times L^2}\leq \frac{b_0^2}{|\log b_0|},
\ee
there exists a unique $d_+$ with $|d_+| \lesssim b_0^2/|\log(b_0)|$ and $T_0>0$ such that the unique decomposition 
$$
u(t) = (P_{B_1(b(t))} + \varepsilon)_{\lambda(t)} = (P_{B_1(b(t))})_{\lambda(t)} + w(t),
$$ 
of the unique smooth solution $u$ to \eqref{wave} on $[0,T_0]$ with initial data \eqref{intiialdtata} satisfies the initialization
\be
\label{intialisaifjeo}
\kappa_+(0) = a_+, 
\ee
and the following smallness condition on $[0,T_0]:$
\begin{itemize}
\item Smallness and positivity of $b$:
\be
\label{controllambdabootsd}
0<b(t)<5b_0;
\ee
\item Pointwise bound on $b_s$:
\be
\label{poitwisebsbootsd}
|b_s(t)|^2\leq C(M)\frac{[b(t)]^{4}}{|\log b(t)|^{2}};
\ee
\item Smallness of the energy norm:
\be
\label{bootneergynormsd}
\|(\nabla w(t),\pa_tw(t)+\frac{b(t)}{\lambda(t)}((1-\chi_{B_1(b(t))})\Lambda Q))_{\lambda(t)}\|_{L^2\times L^2} \leq \sqrt{b_0};
\ee
\item {\it Global} $\dot{H}^2$ bound:
\be
\label{poitnwiseboundbootsd}
|\mathcal E(t)|\leq C(M)\frac{[b(t)]^{4}}{|\log b(t)|^{2}};
\ee
\item A priori bound on the stable mode: 
\be
\label{boundunstablebootsd}
|\kappa_-(t)|\leq (C(M))^{\frac18}\frac{[b(t)]^2}{|\log b(t)|}.
\ee
\item A priori bound of the unstable mode:
\be \label{boundinitialsd}
\ \  |\kappa_+(t)|\leq  2\frac{[b(t)]^{2}}{|\log b(t)|}.
\ee
\end{itemize}\end{lemma}

We may now describe the bootstrap regime as follows:

\begin{definition}[Exit time]
\label{defexittime}
Let $K(M)$ denote some large enough constant.

Given $a_+\in[-\frac{b_0^2}{|\log b_0|}, \frac{b^2_0}{|\log b_0|}]$, we let $T(a_+)$ be the life time of the solution to \fref{wave} with initial data \fref{intiialdtata}, and  $T_{1}(a_+)>0$ be the supremum of $T\in(0,T(a_+))$ such that for all $t\in [0,T]$, the following estimates hold:
\begin{itemize}
\item Smallness and positivity of $b$:
\be
\label{controllambdaboot}
0<b(t)<5b_0;
\ee
\item Pointwise bound on $b_s$:
\be
\label{poitwisebsboot}
|b_s|^2\leq K(M)\frac{[b(t)]^{4}}{|\log b(t)|^{2}};
\ee
\item Smallness of the energy norm:
\be
\label{bootneergynorm2}
\|(\nabla w(t),\pa_tw(t)+\frac{b(t)}{\lambda(t)}((1-\chi_{B_1(b(t))})\Lambda Q))_{\lambda(t)}\|_{L^2\times L^2}\leq\sqrt{b_0};
\ee
\item {\it Global} $\dot{H}^2$ bound:
\be
\label{poitnwiseboundboot}
|\mathcal E(t)|\leq K(M)\frac{[b(t)]^{4}}{|\log b(t)|^{2}};
\ee
\item A priori bound on the stable and unstable modes: 
\be
\label{boundunstableboot}
|\kappa_+(t)|\leq  2\frac{[b(t)]^{2}}{|\log b(t)|}, \ \ |\kappa_-(t)|\leq (K(M))^{\frac18}\frac{[b(t)]^2}{|\log b(t)|}.
\ee
\end{itemize}
\end{definition}

The existence of blow up solutions in the regime described by Theorem \ref{thmmain} now follows from the following:

\begin{proposition}
\label{propcle}

There exists $a_+\in \left[-\frac{b^2_0}{|\log b_0|}, \frac{b^2_0}{|\log b_0|}\right]$ such that $$T_1(a_+) =T(a_+)$$ 
and then corresponding solution to \fref{wave} blows up in finite time in the regime described by Theorem \ref{thmmain}.
\end{proposition}

The proof of Proposition \ref{propcle} relies on a monotonicity argument on the energy $\mathcal E$ which is the core of the analysis, see Proposition \ref{lemmaenergy}, and the strictly outgoing behavior of the unstable mode induced by the non trivial eigenvalue $-\zeta<0$ of $H$, see Lemma \ref{lemmaunstable}. The fact that the regime described by the bootstrap bounds \fref{controllambdaboot}, \fref{poitwisebsboot}, \fref{bootneergynorm2}, \fref{poitnwiseboundboot}, \fref{boundunstableboot} corresponds to a finite blow up solution with a specific blow up speed will then follow from the modulation equations and the sharp derivation of the blow speed as in \cite{RaphRod}.


\section{Improved bounds}
\label{core}


This section is devoted to the derivation of the main dynamical properties of the flow in the bootstrap regime described by Definition \ref{defexittime}. The three main steps are first the derivation of a monotonicity property on $\mathcal E$ which allows us to improve the bounds \fref{controllambdaboot}, \fref{poitwisebsboot}, \fref{bootneergynorm2}, \fref{poitnwiseboundboot}  in $[0,T_1(a_+)]$, second the derivation of the dynamics of the eigenmode and the outgoing behavior of the unstable direction, and eventually the derivation of the sharp law for the parameter $b$ which allows to bootstrap its smallness \fref{controllambdaboot} and will eventually allow us to derive the sharp blow up speed.

\begin{remark}
\label{remarkcofm}
All along the proof, we will introduce various constants $C(M),\delta(M)>0$ {\em which do not depend} on the bootstrap constant $K(M)$. 
An important feature of all these constants is that, up to a smaller choice of  $b^*(M)$ or a larger choice of $K(M)$, we assume that any product of the form $C(M) \, f(b)$ where $\lim_{b \to 0} f(b) = 0$ or any ratio $\delta(M)/K(M)$ is small in the trapped regime. 
This will be used implicitly in this section.

\end{remark}


\subsection{Coercitivity of $\mathcal E$}


Let us start with showing that the linearized energy $\mathcal E$ yields a control of suitable  weighted norms of $(w,\e)$ in the regime $t\in [0,T_1(a_+)]$.

\begin{lemma}[Coercitivity of $\mathcal E$]
\label{coerceenergy}
There exists $M_0\geq 1$ such that for all $M\geq M_0$, there exists\footnote{recall remark \ref{remarkcofm}} $\delta(M)>0$ and $C(M)< \infty$ 
such that in the interval $[0,T_1(a_+)),$ there holds:
\be
\label{estlowerenergy}
\mathcal E  \geq  \frac12\lambda^2\int(H_{\lambda}w)^2+\delta(M)\lambda^2\left[\int(\nabla \pa_t w)^2+\int \frac{(\pa_rw)^2}{r^2}\right]- C(M) [{K(M)}]^{\frac{1}{4}} 
\frac{b^{4}}{|\log b|^{2}}.
\ee
\end{lemma}

\medskip

\noindent{\bf Proof of Lemma \ref{coerceenergy}.}\\ This is a consequence of the explicit distribution of the negative eigenvalues of $H$ and the a priori bound on the unstable mode \fref{boundunstableboot}.
Indeed, let $t\in [0,T_1(a_+))$, then first observe from \fref{defaplus}, \fref{aplusamoins},  \fref{boundunstableboot} that 
\bea \label{estpsiscal}
\nonumber |(\e,\psi)|^2+|(\partial_s\e,\psi)|^2 & \lesssim & |\kappa_+|^2+|\kappa_-|^2+|b_s|^2({\partial_b P_{B_1}},\psi)^2\\
& \lesssim & [K(M)]^{\frac14}\frac{b^{4}}{|\log b|^{2}}+C(M)b^2|b_s|^2\lesssim [K(M)]^{\frac14}\frac{b^{4}}{|\log b|^{2}}
\eea
where we used the estimates of Proposition \ref{propqb} and the well localization of $\psi$. This yields 
\bea
\label{estprofuiocaliceni}
\nonumber \frac{1}{\lambda^4}(w,\psi_{\lambda})^2+\frac{1}{\lambda^2}(\pa_tw,\psi_{\lambda})^2 & = & (\e,\psi)^2+(\pa_s\e+b\Lambda \e,\psi)^2\\
& \lesssim  & [K(M)]^{\frac14}\frac{b^4}{|\log b|^{2}}+b^2\left[ \int\frac{\e^2}{ y^4(1+|\log(y)|)^2 }  
+  \int \frac{|\nabla \e|^2}{y^2}  \right]
\eea
and similarly using the orthogonality condition \fref{orthoeetb}:
\bea
\label{estprofuiocalicenibis}
\nonumber \frac{1}{\lambda^4}(w,(\chi_M\Phi)_{\lambda})^2+\frac{1}{\lambda^2}(\pa_tw,(\chi_M\Phi)_{\lambda})^2 & = & (b\Lambda \e,\chi_M\Phi)^2\\
& \lesssim  &b^2M^C\left[ \int\frac{\e^2}{ y^4(1+|\log(y)|)^2 } 
+ \int \frac{|\nabla \e|^2}{y^2} 
\right].
\eea

Moreover, applying Lemma \ref{lemmahardy1} yields:
\bea
\nonumber \lambda^2 \int |H_\lambda w|^2 &=& \int |H \e|^2 \\
\nonumber						&\ge& \delta(M) \left[ \int \dfrac{|\nabla \e|^2}{y^2}  + \dfrac{\e^2}{y^4(1+|\log(y)|)^2}\right]
\eea

Introducing the rescaled version \eqref{coercHlambda} of Lemma \ref{lemmahardy1}, we then conclude:
\bee
\mathcal E & \geq & \frac{1}{2}\int\lambda^2(H_{\lambda}w)^2+\delta_1(M)\left[\lambda^2 \int(\nabla \pa_t w)^2+\int \frac{|\nabla \e|^2}{y^2} + \int \dfrac{\e^2}{y^4(1+|\log(y)|)^2} \right]\\
& & - \ b^2M^C\left[ \int\frac{\e^2}{ y^4(1+|\log(y)|)^2 } +  \int \frac{|\nabla \e|^2}{y^2} \right]-C(M)[K(M)]^{\frac14}\frac{b^{4}}{|\log b|^{2}}\\[4pt]
& \geq &  \frac{1}{2}\int\lambda^2(H_{\lambda}w)^2+\delta(M)\lambda^2\left[\int(\nabla \pa_t w)^2+\int \frac{(\pa_rw)^2}{r^2}\right]-C(M)[K(M)]^{\frac14}\frac{b^{4}}{|\log b|^{2}}
\eee
where we used the Hardy bound \fref{harfylog}, and \fref{estlowerenergy} is proved.
This concludes the proof of Lemma \ref{coerceenergy}.

\begin{remark} Note that \fref{estlowerenergy} together with  the Hardy estimate \fref{hardy0}, the coercitivity estimate \eqref{secondordercontrol}
and \fref{estprofuiocalicenibis} yield the following weighted bound on $\e$ which will be extensively used in the paper: let 
\be
\label{defeta}
\eta(s,y)=\lambda^{\frac{N-2}{2}+1}\pa_tw(t,\lambda y)=\pa_s\e(s,y)+b\Lambda  \e(s,y),
\ee 
then:

\bea
\label{estcepueipprec}
\int\frac{\e^2}{y^4(1+|\log y|^{2})}+ 
\int\frac{\eta^2}{y^2}\
+ \int \frac{|\nabla \e|^2}{y^2}+ \int |\nabla \eta|^2
&\lesssim& c(M)\left[ |\mathcal E|+ [K(M)]^{\frac{1}{4}}\frac{b^{4}}{|\log b|^{2}}\right], \\
\label{estcepueip}
&\lesssim& c(M)|\mathcal E|+ \sqrt{K(M)} \frac{b^{4}}{|\log b|^{2}}.
\eea
\end{remark}


\subsection{First bound on $b_s$}
\label{forstboundnbs}

We now derive a crude bound on $b_s$ which appears as an {\it order one} forcing term in the RHS of the equation for $\e$ \fref{eqeqb}. This bound is a simple consequence of the construction of the profile $Q_b$ and the choice of the orthogonality condition \fref{orthoeetb}.

\begin{lemma}[Rough pointwise bound on $b_s$]
\label{roughboundpointw}
There holds the rough pointwise bound\footnote{recall remark \ref{remarkcofm}}:
\be
\label{estbds}
 \left(b_s+\frac{(\e,H\Phi)}{(\Lambda Q,\Phi)}\right)^2\lesssim \frac1{M}|\mathcal E|+ \sqrt{K(M)}\frac{b^{4}}{|\log b|^{2}}.
\ee
 \end{lemma}
 
 \begin{remark}
 \label{remarkbs}
 This is in contrast with \cite{RaphRod} where the $b_s$ term could be treated as degenerate with respect to $\e$ thanks to a specific choice of orthogonality conditions and the factorization of the operator $H$ in the wave map case. This difficulty in our case will be treated using a specific algebra generated by our choice of orthogonality condition \fref{orthoeetb} which gives the {\it right sign} to the leading order terms involving $b_s$ in the energy identity \fref{lemmaenergy}, see \fref{lowerboundrone}, \fref{signtwo}.
 \end{remark}

\noindent{\bf Proof of Lemma \ref{roughboundpointw}.}\\  
Let us recall that the equation for $\e$ in rescaled variables is given by \fref{eqeqb}, \fref{defhbl}, \fref{defnw}. Observe also that from (\ref{adjoinctionfrimula}), the adjoint of  $H_{B}$ with respect to the $L^2(\RR^N)$ inner product is given by: 
\be
\label{deflstar}
H_{B_1}^*=H_{B_1}+2b^2D.
\ee 
To compute $b_s$ we take the scalar product of (\ref{eqeqb}) with $\chi_M\Phi$. Using the orthogonality relations $$(\pa_s^m \e,\chi_M\Phi)=(\pa_s^m (P_{B_1}-Q),\chi_M\Phi)=
0,\qquad \forall \ m\ge 0$$ we integrate by parts to get the algebraic identity:
\begin{multline}
\label{elgebraba}
b_s\left[(\Lambda P_{B_1},\chi_M\Phi)+2 b(\Lambda \partial_b P_{B_1},\chi_M\Phi)+(\Lambda\ebo,\chi_M\Phi)\right]\\
=  -(\Psi_{B_1},\chi_M\Phi) -  (\ebo,H_{B_1}^*(\chi_M\Phi))+2 b(\partial_s\ebo,\Lambda(\chi_M \Phi))
+  (N(\ebo),\chi_M\Phi).
\end{multline}
We first derive from the estimates of Proposition \ref{propqb}: 
\be
\label{estpisbun}
(\Psi_{B_1},\chi_M\Phi)^2\lesssim  \frac{b^4}{|\log b|^{2}}.
\ee 
Similarly, using \fref{estcepueipprec} yields:
\be
\label{cbeoheofonboehfeo}
(\partial_s\ebo,\Lambda(\chi_M\Phi))^2\lesssim C(M)\left[ c(M)|\mathcal E|+\sqrt{K(M)}\frac{b^4}{|\log b|^{2}}\right]
\ee
and
$$ (\ebo,H_{B_1}^*(\chi_M\Phi)) =  (\e,H\Phi)-(H\e,(1-\chi_M)\Phi)+O\left(M^Cb^2\sqrt{c(M)|\mathcal E|+\sqrt{K(M)}\frac{b^4}{|\log b|^{2}}}\right).$$ We then use the improved decay \fref{cancelaationphi} and \fref{estcepueip} to estimate:
\bee
(H\e,(1-\chi_M)\Phi)^2 & \lesssim & \left(\int_{y\geq M} \frac{|H\e|}{1+y^{N}}\right)^2\lesssim \frac{|\mathcal E|}{M}+\sqrt{K(M)}\frac{b^4}{|\log b|^{2}}
\eee
Thus:
\be
\label{opeunkvrogu}
\left| (\ebo,H_{B_1}^*(\chi_M\Phi))-  (\e,H\Phi)\right|^2\lesssim \frac{1}{M}|\mathcal E|+\sqrt{K(M)}\frac{b^{4}}{|\log b|^{2}}.
\ee
similarly,
\begin{multline}
\label{cneokeoonokehoieru}
(\Lambda P_{B_1},\chi_M\Phi)+2 b(\Lambda \partial_b P_{B_1} ,\chi_M\Phi) +(\Lambda\ebo,\chi_M\Phi)\\
\begin{array}{rcl}
&=& (\Lambda Q,\Phi)+O\left(\dfrac{b}{\log(b)} + M^C \sqrt{|\mathcal E|+\sqrt{K(M)}\frac{b^{4}}{|\log b|^{2}}}\right)\\[8pt]
&=&(\Lambda Q,\Phi)+O\left(\dfrac{b}{\log(b)}  \right)\\
\end{array}
\end{multline} 
  where we have used that in the trapped regime $\mathcal{E} \le {K(M)}b^4/[\log(b)]^2.$ 
Finally, on the support of $\chi_M$ and for $b<b_0^*(M)$ small enough, the term $Q$ dominates in $Q_b=Q+b^2T_1$. 
Hence, for the nonlinear term, we have from Sobolev and \fref{estcepueip}: 
 \bee
|(N(\e),\chi_M\Phi)| & \lesssim & \int \left(\dfrac{\e^2}{1+y^6}+\dfrac{\e^{3}}{1+y^4}\right)\lesssim \int \dfrac{|\e|^2}{(1+y^5)} \left[ 1+ \|y \e\|_{L^{\infty}} \right]\\
& \lesssim & C(M)\left[ \mathcal E+\sqrt{K(M)}\frac{b^{4}}{|\log b|^{2}}\right].
\eee
Injecting this together with \fref{estpisbun}, \fref{cbeoheofonboehfeo}, \fref{opeunkvrogu}, \fref{cneokeoonokehoieru} into \fref{elgebraba} yields \fref{estbds}\footnote{recall remark \ref{remarkcofm}} and concludes the proof of Lemma \ref{roughboundpointw}.


\subsection{Global ${\dot{H}}^2$ bound}  
\label{sectionhtwo}


We derive in the section a monotonicity statement for the energy $\mathcal E$ which provides a global $\dot{H}^2$ estimate for the solution. The monotonicity statement involves suitable repulsivity properties of the rescaled Hamiltonian $H_{\lambda}$ in the focusing regime under the orthogonality condition \fref{eqeqb} and the a priori control of the unstable mode \fref{boundunstableboot}, which themselves rely on the positivity of an explicit quadratic form, see Lemma \ref{lemmacoercitivity}.

\begin{proposition}[$H^2$ control of the radiation]
\label{lemmaenergy}
In the trapped regime, there exists a function $\mathcal F$ satisfying 
\be
\label{estimatef}
\mathcal{F} \lesssim \frac{\mathcal E}{M}+\sqrt{K(M)}\frac{b^4}{|\log b|^2}
\ee
and such that, for some $0<\alpha<1$ close enough to 1, there holds:
\be
\label{vhoheor}
\frac{d}{dt}\left\{\frac{\mathcal E + \mathcal F}{\lambda^{2(1-\alpha)}}\right\}
\leq   \frac{b}{\lambda^{3-2\alpha}}\left[\sqrt{K(M)}\frac{b^4}{|\log b|^{2}}\right].
\ee

\end{proposition}

{\bf Proof of Proposition \ref{lemmaenergy}}\\

{\bf step 1} Energy identity.\\

Let $$\vt(t,r)=\frac{N+2}{N-2}Q^{\frac{4}{N-2}}_{\lambda}(r)=\frac{1}{\lambda^2}V\left(\frac{r}{\lambda} \right), \ \ V(y)=\frac{N+2}{N-2}Q^{\frac{4}{N-2}}(y).$$
We first have the following algebraic energy identity which follows by integrating by parts from \fref{oeioehoe}:
\be
\frac{1}{2}\frac{d}{dt}\left\{\int(\pa_{tr}w)^2-\int \vt(\pa_tw)^2+\int(H_{\lambda}w)^2\right\}
=  -\int\pa_t \vt\left[\frac{(\pa_t w)^2}{2}+w H_{\lambda}w\right]+\int\pa_twH_{\lambda}F_{B_1}.
\ee
We now use the $w$ equation and integration by parts to compute:
\bea  -\int \pa_t\vt wH_{\l}w &=&  -\int\pa_t\vt w(F_{B_1}-\pa_{tt}w)\\
& = & \frac{d}{dt}\left\{\int\pa_t\vt w\pa_t w\right\}-\int\pa_t\vt wF_{B_1}-\int\pa_t\vt(\pa_tw)^2-\int 
\pa_{tt}\vt w\pa_tw
\eea
We next pick $0<\alpha<1$ close enough to 1 and combine the above identities to get:
\begin{multline}
\label{computationeergy}
\frac{1}{2\l^{2\alpha}}\frac{d}{dt}\left\{\l^{2\alpha}\left[\int(\pa_{tr}w)^2-\int \vt(\pa_tw)^2+\int(H_{\lambda}w)^2-
 2 \int\pa_t\tilde{V} w\pa_t w \right]\right\}\\
  =  -R_1+ R_2+  \frac{2\alpha b}{\lambda}\int\pa_t\vt w\pa_t w-  \int \pa_{tt}\tilde{V} w\pa_tw
\end{multline}
where $R_1$ collects the quadratic terms:
\bea
\label{firclaucrun}
\nonumber R_1&= & \frac{\alpha b}{\lambda}\left[\int(\pa_{tr}w)^2-\int \vt(\pa_tw)^2+\int(H_{\lambda}w)^2\right]+\frac{3}{2}\int\pa_t \vt(\pa_t w)^2-\nonumber \frac{b_s}{\l^2}\int\pa_t \vt(\Lambda Q)_{\l}w\\
\nonumber & = & \frac{b}{\l^3}\left[\alpha\int(\pa_{y}\eta)^2-\alpha\int V\eta^2+\alpha \int(H\e)^2+\frac{3}{2}\int(2V+y\cdot\nabla V)\eta^2\right .\\
& - & \left .b_s\int \e(2V+y\cdot\nabla V)\Lambda Q\right]
\eea 
and $R_2$ collects the nonlinear higher order terms:
\bea
\label{defrtwo}
R_2= \int\pa_twH_{\lambda}F_{B_1} - \int\pa_t\vt w\left[F_{B_1}+\frac{b_s}{\l^2}(\Lambda Q)_{\l}\right]
\eea

{\bf step 2} Derivation of the quadratic terms and treatment of the $b_s$ term.\\
Let us now obtain a suitable {\it lower bound} for the quadratic term $R_1$. The main enemy is the $b_s$ term which is order one in $\e$ and will be treated using a specific algebra generated by the choice of orthogonality condition \fref{orthoeetb}. 

Observe from $H(\Lambda Q)=0$ that $(\Lambda Q/\lambda)_{\lambda}(y)=(1/\lambda)^{\frac{N}{2}}(\Lambda Q)(y/\lambda)$ satisfies: $$-\Delta (\Lambda Q/\lambda)_{\lambda}(y)-(1/\lambda)^2V(y/\lambda)(\Lambda Q/\lambda)_{\lambda}(y)=0.$$ Differentiating this relation at $\lambda=1$ yields: 
$$H\Phi=H(D\Lambda Q)=(2V+y\cdot\nabla V)\Lambda Q.$$ 

We inject this into the modulation equation \fref{estbds} to get:

\be
\label{signone}
-b_s\int\e(2V+y\cdot\nabla V)\Lambda Q   =    b_s^2(\Phi,\Lambda Q)+|b_s|O\left( \frac{|\mathcal E|}{M}+\sqrt{K(M)}\frac{b^4}{|\log b|^{2}}\right)^{\frac{1}{2}}.
\ee 

We thus conclude using the sign $$(\Phi,\Lambda Q)>0$$ and \fref{firclaucrun}, \fref{estbds} that:
\begin{multline}
\label{lowerboundrone}
R_1 \geq  \frac{b}{\l^3}\left[\alpha\int(\pa_y\eta)^2+\int[(3-\alpha)V+\frac{3}{2}y\cdot\nabla V]\eta^2+\alpha\int(H\e)^2\right .
+   c_1(b_s)^2\\
+\left.
O\left(\frac{|\mathcal E|}{M}+\sqrt{K(M)} \frac{b^4}{|\log b|^{2}}\right)
\right]
\end{multline} 
for some universal constant $c_1>0$ {\it independent of $M$}. \\

{\bf step 3} Coercitivity of the quadratic form.\\
We now claim the following coercitivity property of the quadratic form in $\eta$ appearing in the RHS of \fref{lowerboundrone} in the limit case $\alpha=1$,  see Appendix \ref{app_coerc}:

\begin{lemma}[Coercitivity of the quadratic form]
\label{lemmacoercitivity}
There exists a universal constant $c_0>0$ such that  for all $\eta \in \dot{H}_{rad}^1$, there holds: 
$$\int(\pa_{y}\eta)^2+\int \left[2V+\frac{3}{2}y\cdot\nabla V\right]\eta^2\geq c_0\int (\pa_y\eta)^2-\frac{1}{c_0}\left[(\eta,\psi)^2+(\eta,\Phi)^2\right].$$
\end{lemma}

From a simple continuity argument, there exists $0<\alpha^*<1$ such that given $0<\alpha^*<\alpha\leq 1$, 
for all $\eta\in\dot{H}^1_{rad}$, there holds:
 $$\alpha\int(\pa_{y}\eta)^2+\int \left[(3-\alpha)V+\frac{3}{2}y\cdot\nabla V\right]\eta^2\geq \frac{c_0}{2}\int (\pa_y\eta)^2-\frac{2}{c_0}\left[(\eta,\psi)^2+(\eta,\Phi)^2\right].$$ 
 We now pick once and forall such an $\alpha<1$ and control the negative directions.
 
Using \fref{estprofuiocaliceni} and \fref{estcepueip}, it yields:
 $$
 (\eta,\psi)^2   \lesssim b|\mathcal E|+\sqrt{K(M)}\frac{b^4}{|\log b|^{2}}
 $$
 Similarly, we compute $(\eta,\Phi) =   (\eta, \chi_M \Phi) + (\eta, (1-\chi_M) \Phi)$ for which
  \fref{estprofuiocalicenibis} and \fref{estcepueip} yield
 $$(\eta, \chi_M \Phi)^2  \lesssim b|\mathcal E|+\sqrt{K(M)}\frac{b^4}{|\log b|^{2}}$$  
and we have, applying \fref{hardy0}:
$$
(\eta, (1-\chi_M) \Phi)^2 \le \| y \eta\|_{L^{\infty}}^2  \left[ \int_{y \geq M/2} \dfrac{|\Phi|}{y} \right]^2 \lesssim  \frac{1}{M} \int |\partial_y\eta|^2
$$

 This together with  \fref{lowerboundrone} yields the lower bound on quadratic terms:
 \be
 \label{lowerquadra}
 R_1\geq  \frac{b}{\l^3}\left[c_1((b_s)^2+|\mathcal E|)+O\left(\sqrt{K(M)}\frac{b^4}{|\log b|^{2}}\right)\right]
 \ee
 for some universal constant $c_1>0$. 
Indeed, a straightforward integration by parts in \eqref{poitnwiseboundWbis} yields: 
$$
\mathcal{E} \lesssim \int |\partial_y \eta|^2 + \int |H\e|^2.
$$

{\bf step 4} Control of lower order quadratic terms.\\

The lower order quadratic terms in \fref{computationeergy} are controlled similarly:
\bea
\nonumber \left|\int\pa_t\vt w\pa_t w\right|\lesssim \frac{b}{\l^2}\left[\int\frac{\e^2}{1+y^6}+\int\frac{\eta^2}{y^2}\right]
&\lesssim & \dfrac{1}{\lambda^2}\left(bC(M)|\mathcal E|+\sqrt{K(M)}\frac{b^4}{|\log b|^{2}}\right),\\
\label{commutator} & \lesssim &   \frac{1}{\l^2}\left(\frac{|\mathcal E|}{M}+\sqrt{K(M)}\frac{b^4}{|\log b|^{2)}}\right),
\eea

and, with the help of \eqref{poitwisebsboot},

\bee
\left|\int \pa_{tt}\vt w\pa_tw\right| & \lesssim & \left( \frac{b^2}{\l^3} + \frac{|b_s|}{\lambda^3} \right)\left[\int\frac{\e^2}{1+y^6}+\int\frac{\eta^2}{y^2}\right],
\\
& \lesssim & \frac{b}{\l^3}\left(bC(M)|\mathcal E|+\sqrt{K(M)}\frac{b^4}{|\log b|^{2)}}\right). \\[4pt]
\eee

\begin{remark}
\label{remarkcommutation}
We note here that \fref{commutator} is sufficient for the proof of our theorem. Indeed, the estimated term $\int \pa_t\vt w\pa_t w$ has been integrated by parts with respect to time, so that it becomes a part of $\mathcal F.$
Furthermore, we note that to compute \fref{vhoheor}, we multiply $\mathcal F$ by $\lambda^{2\alpha}.$ Consequently, the commutator $b\alpha/\lambda \int \pa_t\vt w\pa_t w$ appears on the right-hand side.
However, \fref{commutator} yields that, in the trapped regime, this
supplementary term is controlled by $b/\lambda^{3} \sqrt{K(M)} b^4/|\log b|^2.$ 
Similar arguments will be repeated implicitly below for the terms which require an integration by parts with respect to time.\\
\end{remark}

{\bf step 5} Rewriting of the nonlinear $R_2$ terms.\\

It remains to control the nonlinear $R_2$ terms in \fref{computationeergy} given by \fref{defrtwo}. According to \fref{oeioehoe}, this term contains $b_{ss}$ type of terms which cannot be estimated in absolute value and require a further integration by parts in time. Let 
\be
\label{decompFbone}
F_{B_1}=F_1-\pa_tF_2\ \  \mbox{with} \ \ F_2=\frac{1}{\lambda}(\pa_sP_{B_1})_{\lambda}
\ee and rewrite: 
$$
R_2 =  \int\pa_tw H_{\lambda}F_{1}-\int\pa_t\vt w\left[F_{1}+\frac{b_s}{\l^2}(\Lambda Q)_{\l}\right]
 - \int\pa_tw H_{\lambda}\pa_tF_2+\int\pa_t\vt w\pa_tF_{2}
$$
 We now integrate by parts in time to treat the $F_2$ term:
\begin{multline*}
 -\int\pa_twH_{\lambda}\pa_tF_2+\int\pa_t\vt w\pa_tF_{2}\\
 = -\frac{d}{dt}\left\{\int\pa_twH_{\lambda}F_2-\int \pa_t\vt wF_2\right\}
-  \int ( \pa_{tt}\vt w+  2 \pa_t\vt \pa_t w)F_2+ \int\pa_{tt}w H_{\lambda}F_2.
\end{multline*}
The last term is rewritten using \fref{oeioehoe} and integration by parts:
\bee
 \int\pa_{tt}w H_{\lambda}F_2 
&=&\int[F_1-\pa_tF_2-H_{\lambda}w] H_{\lambda}F_2\\
&=&  -\frac{1}{2}\frac{d}{dt}\left\{\int|\nabla F_2|^2-\int  \vt F_2^2\right\}-\frac12\int \pa_t\vt F_2^2
+  \int[F_1-H_{\lambda}w]H_{\lambda}F_2.
\eee
 eventually arrive at a manageable expression for $R_2$: 
\bea
\label{estmanegeable}
R_2&= &  -\frac{d}{dt}\left\{\int\pa_twH_{\lambda}F_2  -\int \pa_t\vt wF_2
+\frac{1}{2}\int|\nabla F_2|^2-\frac 12\int \vt F_2^2\right\}\\
\nonumber & - & \int\pa_t\vt  w\left[F_{1}+\frac{b_s}{\l^2}(\Lambda Q)_{\l}\right]+\int\pa_twH_{\lambda}F_{1}-   \int ( \pa_{tt}\vt w+ 2 \pa_t\vt \pa_t w)F_2\\
\nonumber & - & \frac12\int \pa_t\vt F_2^2+\int[F_1-H_{\lambda}w] H_{\lambda}F_2.
\eea
We now aim at estimating all the terms in the RHS of \fref{estmanegeable}. According to \fref{oeioehoe}, we split $F_1$ into four terms:
\be
\label{decompositionfone}
F_1+\frac{b_s}{\l^2}(\Lambda Q)_{\l}=-\frac{1}{\l^2}\left[\Psi_{B_1}+F_{1,1}+F_{1,2}+N(\e)\right]_{\lambda}
\ee
with 
\be
\label{decompfonone}
F_{1,1}=b\Lambda\partial_sP_{B_1}+b_s(\Lambda P_{B_1}-\Lambda Q), \ \ F_{1,2}=\left[f'(Q)-f'(P_{B_1})\right]\e.
\ee

{\bf step 6} $F_1$ terms.\\

The $F_1$ terms are the leading order terms.\\
{\em $\Psi_{B_1}$ terms}: We first extract from \fref{estpsibloin} the rough bound:
 \be\label{roughpsiboine}
 |\Psi_{B_1}|\lesssim  \frac{b^2}{|\log b|(1+y^2)}+C(M) b^4{\bf 1}_{y\leq 2B_1}
 \ee 
 which yields: $$\int \frac{1+|\log y|^2}{1+y^4}|\Psi_{B_1}|^2\lesssim \frac{b^4}{|\log b|^2}
$$
 and thus from \fref{estcepueip}:
 \bee
 \left|\int \pa_t\vt w\frac{1}{\l^2}(\Psi_{B_1})_{\l}\right| & 
\lesssim & 
\frac{b}{\l^3} \int \dfrac{|\e||\Psi_{B_1}|}{(1+y^4)} 
  \\
 &\lesssim & \frac{b}{\l^3}\frac{b^2}{|\log b|}C(M)\sqrt{|\mathcal E|+\sqrt{K(M)}\frac{b^4}{|\log b|^2}}\\
 & \lesssim & \frac{b}{\l^3}\sqrt{K(M)}\frac{b^4}{|\log b|^{2}}.
 \eee
Next, we use the fundamental cancellation $H(\Lambda Q)=0$ and \fref{estpsibloin} to estimate:
\bee
|H\Psi_{B_1}| &\lesssim&   \frac{b^4}{1+y^2}\left[\frac{1+|\log(by)|}{|\log b|}{\bf 1}_{2\leq y\leq 2B_0}+\frac{1}{b^2y^2|\log b|}{\bf 1}_{\frac{B_0}{2}\leq y\leq 2B_1}+\frac{\log(M)+|\log (1+ y)|}{1+y^2}{\bf 1}_{y\leq 2B_1}\right]\\
 &+&  \frac{b^2}{(1+y^4)|\log b|}{\bf 1}_{ y\geq B_1/2} ,
\eee
and thus 
\be
\label{coeoione}
\int (1+y^2)|H(\Psi_{B_1})|^2\lesssim \frac{b^6}{|\log b|^2}.
\ee  
 Hence:
\bee
 \left|\int\pa_tw H_{\lambda}(\frac{1}{\lambda^2}(\Psi_{B_1})_{\lambda})\right| &\lesssim  &
 \dfrac{b}{\l^3} \|\eta/y\|_{L^2}  \left[ \int \dfrac{1}{b^2} (1+y)^2|H(\Psi_{B_1})|^2 \right]^{\frac{1}{2}}
 \\
&\lesssim&  \frac{b}{\l^3}\sqrt{K(M)}\frac{b^4}{|\log b|^{2}}.
\eee

{\em $F_{1,1}$ terms}: From \fref{esttone}, \fref{estpbcut},  there holds :
$$
|F_{1,1}|   \lesssim  |b_s|{b^2}\left[\frac{1+|\log(by)|}{|\log b|}{\bf 1}_{2\leq y\leq \frac{B_0}{2}}+\frac{1}{b^2y^2|\log b|}{\bf 1}_{\frac{B_0}{2}\leq y \leq 2B_1}+\frac{\log(M)+|\log y|}{1+y^2}\right]
$$
and, recalling that differentiation w.r.t. $y$ acts as a multiplication by $1/(1+y)$ : 
$$
\left|HF_{1,1}\right|\\
 \lesssim  C(M)\frac{|b_s|b^2}{1+y^2}\left[\frac{1+|\log(by)|}{|\log b|}{\bf 1}_{2\leq y\leq \frac{B_0}{2}}+\frac{1}{b^2y^2|\log b|}{\bf 1}_{\frac{B_0}{2}\leq y \leq 2B_1}+\frac{\log(M)+|\log y|}{1+y^2}\right]
$$
from which 
\be
\label{coeoionebis}\int  (1+y^2)|H(F_{1,1})|^2\lesssim |b_s|^2 \frac{b^2}{|\log b|^2},
\qquad 
\int \dfrac{(1+|\log y|^2)}{(1+y^4)} |F_{1,1}|^2 \lesssim |b_s|^2 {b^2}.
\ee
 Hence similar arguments as with the $\Psi_{B_1}$ terms yield:
\bee
\left|\int \pa_t \tilde{V} w F_{1,1}\right| & \lesssim&  \frac{b}{\l^3}b |b_s| C(M)\sqrt{|\mathcal E|+\sqrt{K(M)}\frac{b^4}{|\log b|^2}}
\\
&\lesssim&  \frac{b}{\l^3}\sqrt{K(M)}\frac{b^4}{|\log b|^{2}},
\eee
and

\bee
\left|\int\pa_twH_{\l}F_{1,1}\right| & \lesssim &  \frac{C(M)b}{\l^3}\frac{|b_s|}{|\log b|}\sqrt{|\mathcal E|+\sqrt{K(M)}\frac{b^{4}}{|\log b|^2}}\lesssim \frac{b}{\l^3}\left[\frac{|b_s|^2}{|\log b|}+\frac{\mathcal E}{|\log b|}+\sqrt{K(M)}\frac{b^4}{|\log b|^2}\right]\\
& \lesssim & \frac{b}{\l^3}\sqrt{K(M)}\frac{b^4}{|\log b|^{2}}.
\eee

{\em $F_{1,2}$ terms}: The explicit expansion of the cubic nonlinearity and the bound \fref{esttone} yield: 
\be
\label{estimtevkdljdlrm}
|F_{1,2}|\lesssim \frac{C(M)b^2}{1+y^2}|\e|, \ \ |\nabla F_{1,2}|\lesssim \frac{C(M)b^2}{1+y^3}|\e|+\frac{C(M)b^2}{1+y^2}|\nabla \e|
\ee
from which:
$$
\frac{1}{\l^2}  \left|  \int\pa_t\vt w(F_{1,2})_{\l}\right|\lesssim \frac{C(M)b^3}{\l^3}\int\frac{\e^2}{1+y^6}\lesssim \frac{b}{\l^3}\left(b|\mathcal E|+\sqrt{K(M)}\frac{b^4}{|\log b|^2}\right),
$$
 and, after integration by parts of the laplacian term:
\bee
\frac{1}{\l^2}\left|\int\pa_twH_{\lambda}(F_{1,2})_{\lambda}\right| & \lesssim & 
\frac{C(M)}{\l^3}\left[\int \frac{|\eta|}{1+y^4}\frac{b^2}{1+y^2}|\e|+\int|\nabla \eta|\left(\frac{b^2}{1+y^3}|\e|+\frac{b^2}{1+y^2}|\nabla \e|\right)\right]\\
& \lesssim & \frac{b}{\l^3}\left[\dfrac{|\mathcal E|}{M}+\sqrt{K(M)}\frac{b^4}{|\log b|^2}\right].
\eee
{\em Nonlinear term $N(\e)$}: We expand the nonlinearity: $$N(\e)=3P_{B_1}\e^2+\e^3.$$ This yields using \fref{bootneergynormsd}, \fref{hardy0} the rough bound: $$|N(\e)|\lesssim \frac{\e^2}{1+y}.$$ In what follows, we will use the following bound on $\eta$ which follows from \fref{estcepueipprec}, \fref{hardy0}:
$$ |y\eta|_{L^{\infty}}\lesssim |\nabla \eta|_{L^2}\lesssim \left(c(M)|\mathcal E|+ \sqrt{K(M)} \frac{b^{4}}{|\log b|^{2}}\right)^{\frac 12}.$$
We then estimate:
\bee
\left|\frac{1}{\l^2}\int \pa_t\tilde{V}w(N(\e))_\l\right| & \lesssim & \frac{b}{\l^3}\int \frac{|\e|^3}{1+y^5}\lesssim  \frac{b}{\l^3}|\nabla \e|_{L^2}\left(c(M)|{\mathcal E}|+\sqrt{K(M)}\frac{b^4}{|\log b|^2}\right)\\
& \lesssim & \frac{b}{\l^3}\sqrt{K(M)}\frac{b^4}{|\log b|^{2}}
\eee
for $b_0<b^*(M)$ small enough. We split the second term:
\be
\label{cnehoeoei}
\int \pa_tw H_{\l}\left(\frac{(N(\e))_{\l}}{\l^2}\right)=\int \nabla\pa_tw\cdot\nabla\left(\frac{(N(\e))_{\l}}{\l^2}\right)-\int \tilde{V}\pa_tw\left(\frac{(N(\e))_{\l}}{\l^2}\right).
\ee
The second term is estimated in brute force:
\bee
\left|\int \tilde{V}\pa_tw\left(\frac{(N(\e))_{\l}}{\l^2}\right)\right|& \lesssim & \frac{1}{\l^3}\int \frac{|\eta||\e|^2}{1+y^5}\lesssim \frac{1}{\l^3}|y\eta|_{L^{\infty}}\int \frac{|\e|^2}{1+y^6}\\
& \lesssim & \frac{1}{\l^3}\left(c(M)|\mathcal E|+\sqrt{K(M)}\frac{b^4}{|\log b|^2}\right)^{\frac{3}{2}}\\
& \lesssim & \frac{b}{\l^3}\frac{b^4}{|\log b|^{2}}.
\eee
The first term in \fref{cnehoeoei} is split into two parts:
$$\int \nabla\pa_tw\cdot\nabla\left(\frac{(N(\e))_{\l}}{\l^2}\right)=\int \nabla\pa_tw\cdot\left[\nabla (w^3)+3(P_{B_1})_{\l}\nabla(w^2)\right]+\frac{3}{\l^3}\int\e^2\nabla \eta\cdot\nabla P_{B_1}.$$ The last term is integrated by parts in space and then estimated in brute force:
\bee
 \left|\frac{3}{\l^3}\int\e^2\nabla \eta\cdot\nabla P_{B_1}\right| &= &\frac{3}{\l^3}\left|\int \eta\left[\e^2\Delta P_{B_1}+2\e\nabla P_{B_1}\cdot\nabla \e\right]\right|\\
 & \lesssim & \frac{1}{\l^3}\int |\eta|\left[\frac{\e^2}{1+y^4}+\frac{|\e||\nabla \e|}{1+y^3}\right]\lesssim \frac{1}{\l^3}|y\eta|_{L^{\infty}}\left[\int \frac{\e^2}{1+y^5}+\int\frac{|\nabla\e|^2}{y^2}\right]\\
 & \lesssim & \frac{1}{\l^3}\left(c(M)|\mathcal E|+\sqrt{K(M)}\frac{b^4}{|\log b|^2}\right)^{\frac{3}{2}}\lesssim  \frac{b}{\l^3}\frac{b^4}{|\log b|^{2}}.
\eee
The first term is the most delicate one and requires first a time integration by parts:
\bee
& & \int \nabla\pa_tw\cdot\left[\nabla (w^3)+3(P_{B_1})_{\l}\nabla(w^2)\right]=\frac{d}{dt}\left\{\int |\nabla w|^2\left[\frac{3}{2}w^2+3(P_{B_1})_{\l}w\right]\right\}\\
& - & 3\int w\pa_tw|\nabla w|^2-3\int |\nabla w|^2\left[w\pa_t(P_{B_1})_{\l}+(P_{B_1})_{\l}\pa_tw\right].
\eee
We may now estimate all terms in brute force:
$$
\left|\int |\nabla w|^2\left[\frac{3}{2}w^2+3(P_{B_1})_{\l}w\right]\right| \lesssim  \frac{1}{\l^2}[|y\e|_{L^{\infty}} + |yP_{B_1}|_{L^{\infty}}] |y\e|_{L^{\infty}}  \int \frac{|\nabla \e|^2}{y^2}\lesssim  \frac{1}{\l^2}\frac{b^4}{|\log b|^2},
$$
\bee
\left|\int w\pa_tw|\nabla w|^2\right| & \lesssim & \frac{1}{\l^2}|y\e|_{L^{\infty}}|y\eta|_{L^{\infty}}\int\frac{|\nabla \e|^2}{y^2}\lesssim \left(c(M)|\mathcal E|+\sqrt{K(M)}\frac{b^4}{|\log b|^2}\right)^{\frac{3}{2}}\\
& \lesssim & \frac{b}{\l^3}\frac{b^4}{|\log b|^{2}},
\eee
\bee
\left|\int |\nabla w|^2w\pa_t(P_{B_1})_{\l}\right| & \lesssim &\frac{|yw|_{L^{\infty}}}{\l^3}\int \frac{|\nabla w|^2}{y}\left[\frac{b}{1+y^2}+C(M)b|b_s|{\bf 1}_{y\leq B_1}\right]\\
& \lesssim & \frac{b}{\l^3}|\nabla \e|_{L^2}\left(1+C(M)|b_s|\frac{|\log b|}{b}\right)\int\frac{|\nabla \e|^2}{y^2}\\
& \lesssim & \frac{b}{\l^3}\frac{b^4}{|\log b|^{2}}
\eee
where we used the rough bound extracted from \fref{estpbcut}: $|\pa_bP_{B_1}|\lesssim C(M)b{\bf 1}_{y\leq B_1},$ and finally:
\bee
\left|\int |\nabla w|^2(P_{B_1})_{\l}\pa_tw\right|& \lesssim &\frac{1}{\l^3} |y\eta|_{L^{\infty}}\int\frac{|\nabla \e|^2}{1+y^3}\lesssim \left(C(M)\mathcal E+\sqrt{K(M)}\frac{b^4}{|\log b|^2}\right)^{\frac{3}{2}}\\
& \lesssim & \frac{b}{\l^3}\frac{b^4}{|\log b|^{2}},
\eee
for $b_0<b^*(M)$ small enough. The above chain of estimates together with remark \ref{remarkcommutation} closes the control of the nonlinear term $N(\e)$.\\

{\bf step 9} $F_2$ terms.\\

We  estimate from \fref{estpbcut}:
\be
\label{estpartialpb}
\int\left| \frac{\partial_b P_{B_1}}{(1+y)}\right|^2+\int\left|{\nabla  \partial_b P_{B_1}}\right|^2\lesssim \frac{1}{|\log b|^2}, \ \  \int\frac{1}{1+y^3} \left|{\partial_b P_{B_1}}\right|^2\lesssim \dfrac{b}{|\log b|^2}
\ee
 and hence:
  \bee
 \left|\int\pa_twH_{\lambda}F_2\right|  
 & \lesssim & \frac{|b_s|}{\l^2} \left[\int \dfrac{|\eta| |\partial_b P_{B_1}|}{(1+y^4)}  + \int {|\nabla \eta||\nabla \partial_b P_{B_1}|}\right]
  \\
&\lesssim &  \frac{1}{\l^2}\frac{|b_s|}{|\log b|}\sqrt{c(M)|\mathcal E|+\sqrt{K(M)}\frac{b^4}{|\log b|^2}}\\
 & \lesssim &\frac{1}{\l^2}\left[ \dfrac{|\mathcal E|}{M}+\sqrt{K(M)}\frac{b^4}{|\log b|^2}\right],
 \eee
 
 \bee
 \left|\int \pa_t \vt w F_2  \right| &\lesssim& \dfrac{|b_s| b}{\l^2} \int \dfrac{|\partial_b P_{B_1}||\e|}{(1+y^4)} 
 												\lesssim \dfrac{|b_s| b}{\l^2 |\log b|} \left[c(M) |\mathcal{E}| + \sqrt{K(M)} \dfrac{b^4}{|\log b|^2} \right]^{\frac{1}{2}}\\
 												&\lesssim& \dfrac{1}{\l^2} \left[\dfrac{|\mathcal E|}{M} + \sqrt{K(M)} \dfrac{b^4}{|\log b|^2} \right],
 \eee
 \bee
 \int|\nabla F_2|^2+\left|\int VF_2^2\right|  
 & \lesssim & \frac{|b_s|^2}{\l^2} \left[ \int \dfrac{|\partial_b P_{B_1}|^2}{(1+y^4)} + \int |\nabla \partial_b P_{B_1}|^2 \right]
  \\
&\lesssim &  \frac{1}{\l^2}\frac{(b_s)^2}{|\log b|^2}\lesssim \frac{1}{\l^2}\frac{b^4}{|\log b|^2}.
 \eee
 
 similarly:
 \begin{multline*}
 \left| \displaystyle{\int} ( \pa_{tt}\vt w+  2\pa_t\vt \pa_t w)F_2\right|+\left|\int \pa_t\vt F_2^2\right| 
  \lesssim  
 \dfrac{|b_s|}{\l^3} 
 \left[\displaystyle{\int} \dfrac{((|b_s| + b^2)|\e|+b|\eta|) |\partial_b P_{B_1}|}{(1+y^4)} + |b_s|b \displaystyle{\int} \dfrac{|\partial_b P_{B_1}|^2}{1+y^4} \right]
 \\
 \begin{array}{cl}
 \lesssim &
 \dfrac{|b_s|}{\l^3}\left[\dfrac{(|b_s| + b)}{|\log b|}\sqrt{c(M)|\mathcal E|+\sqrt{K(M)}\dfrac{b^4}{|\log b|^2}}+\dfrac{b^2}{|\log b|^2}|b_s|\right]
 \\[12pt]
 \lesssim & \dfrac{b}{\l^3}\left[\dfrac{|\mathcal E|}{M}+\sqrt{K(M)}\dfrac{b^4}{|\log b|^2}\right].
 \end{array}
 \end{multline*}
Eventually, \fref{coeoione}, \fref{coeoionebis} ensure:
$$\int (1+y^2)|H(\Psi_{B_1}+F_{1,1})|^2\lesssim  \left[ \frac{b^6}{|\log b|^2}+\frac{b^2|b_s|^2}{|\log b|^2} \right]\lesssim \frac{b^6}{|\log b|^2}$$ 
which together with \fref{estpartialpb} yields:
$$\left|\int\frac{1}{\l^2}(\Psi_{B_1}+F_{1,1})_{\l}H_{\lambda}F_2\right| \lesssim  \frac{1}{\l^3}\frac{b^3|b_s|}{|\log b|^2}\lesssim \frac{b}{\l^3}\frac{b^4}{|\log b|^2}.
$$
We similarly estimate from \fref{estimtevkdljdlrm}  and after integration by parts:
\bee
\left|\int\frac{1}{\l^2}(F_{1,2})_{\l}H_{\lambda}F_2\right| 
 &\lesssim& 
 \dfrac{|b_s|}{\l^3} \left[\int \dfrac{b^2 |\e||\partial_b P_{B_1}|}{1+y^6} + \int |\nabla \partial_b P_{B_1}|\left( \dfrac{b^2|\e|}{1+y^3} + \dfrac{b^2|\nabla \e|}{1+y^2}\right)     \right]
\\
	 & \lesssim & C(M) \frac{b^4}{\l^3|\log b|}\left(\int\frac{\e^2}{1+y^{6}}+\int\frac{|\nabla \e|^2}{1+y^4}\right)^{\frac12}\lesssim  \frac{b}{\l^3}\frac{b^4}{|\log b|^2}.
	 \eee
	 
For the nonlinear term, we extract from \fref{estpbcut} the rough bound $$|H(\pa_bP_{B_1})|\lesssim [C(M) + \log(b)] \frac{b}{1+y^2}{\bf 1}_{y\leq B_1},$$ which together with \eqref{hardy0} ensures:
\bee
 \left|\int\frac{1}{\l^2}(N(\e))_{\l}H_{\lambda}F_2\right| & \lesssim & \frac{[C(M) + \log(b)]}{\l^3}|b_s|\int\frac{b}{1+y^2} \frac{\e^2}{1+y} {\bf 1}_{y\leq B_1}\\
 & \lesssim & C(M) \frac{|b_s||\log b|^4}{\l^3}\int \frac{\e^2}{(1+y^4)|\log y|^2}\\
 & \lesssim & \frac{b}{\l^3}\sqrt{b}\left(c(M)|\mathcal E|+\sqrt{K(M)}\frac{b^4}{|\log b|^2}\right)\lesssim \frac{b}{\l^3}\frac{b^4}{|\log b|^2}.
 \eee

{\bf step 10} The remaining $F_2$ term has the right sign.\\

It remains to estimate the term $$-\int H_{\l}wH_{\l} F_2$$ in the RHS of \fref{estmanegeable}. Let us stress onto the fact that this term is a priori no better $O(\frac{1}{\l^3}\mathcal E)$ due to the $b_s$ contribution and the bound \fref{estbds}, recall remark \ref{remarkbs}.\\
We now claim that the {\it main contribution has the right sign again}.
 
Indeed, we first compute from the $T_1$ equation \fref{eqtone}:
\be
\label{vnoiogure}HT_1=-\Phi+c_b\chi_{\frac{B_0}{4}}\Lambda Q, \quad 
H \partial_{b} T_1 =  O\left(\dfrac{1}{b|\log b|} \dfrac{\mathbf{1}_{2 \leq y \leq \frac{B_0}{2}}}{(1+y^2)}\right)
\ee
We then apply the decomposition \eqref{decompdbT1}:
$$
H (\partial_b P_{B_1})  =  H\Big(2b T_1 + 2b (\chi_{B_1}-1) T_1 +  b^2 \partial_b \log(B_1) \rho_{B_1} T_1 + b^2 \chi_{B_1} \partial_b T_1\Big)=-2b\Phi  +  \Sigma
$$ 
and estimate using \fref{estpbcut}, \fref{esttbine}, \fref{vnoiogure}:
$$|\Sigma|\lesssim  \frac{b}{1+y^{2}}\left[\frac{1}{|\log b|}{\bf 1}_{2 \leq y \leq \frac{B_0}{2}}+\frac{1}{b^2y^2|\log b|}{\bf 1}_{\frac{B_0}{2}\leq y}\right].$$

In particular, $$\int \Sigma^2\lesssim \frac{b^2}{|\log b|}$$ and thus using the modulation equation \fref{estbds} : 
\bea
\label{signtwo}
& & \nonumber  -\int H_{\l}wH_{\l} F_2 =  -\frac{b_s}{\l^3}\int (H\e) H(\partial_b P_{B_1})=-\frac{b_s}{\l^3}\int H\e\left(-2b\Phi+\Sigma\right)\\
\nonumber  & = & 2\frac{b}{\l^3}b_s(\e,H\Phi)+\frac{b}{\l^3}O\left(\frac{|b_s|}{\sqrt{|\log b|}}\sqrt{|\mathcal E|+\sqrt{K(M)}\frac{b^4}{|\log b|^2}}\right)\\
\nonumber & = & 2\frac{b}{\l^3}\left[-\frac{(\e,H\Phi)}{(\Lambda Q,\Phi)}+O\left(\sqrt{\frac{|\mathcal E|}{M}+\sqrt{K(M)}\frac{b^4}{|\log b|^2}}\right)\right](\e,H\Phi)+\frac{b}{\l^3}O\left(\frac{b^4}{|\log b|^2}\right)\\
& = & -\frac{2b}{\l^3}\frac{(\e,H\Phi)^2}{(\Lambda Q,\Phi)}+O\left(\frac{|\mathcal E|}{M}+\sqrt{K(M)}\frac{b^4}{|\log b|^2}\right) + \frac{b}{\l^3}O\left(\frac{b^4}{|\log b|^2}\right)\\
& \leq & O\left(\frac{|\mathcal E|}{M}+\sqrt{K(M)}\frac{b^4}{|\log b|^2}\right).
\eea
The recollection of all above estimates yields \fref{vhoheor} and concludes the proof of Proposition \ref{lemmaenergy}.


\subsection{Improved bound}


We now claim that the a priori bound on the unstable direction \fref{boundunstableboot} coupled with the monotonicity property of Proposition \ref{lemmaenergy} imply the following improved bounds:

\begin{lemma}[Improved bounds under the a priori control \fref{boundunstableboot}]
\label{lemmaimprovedbounds}
There holds in $[0,T_1(a_+)]$:
\be
\label{energyboundbootted}
\|(\nabla w(t),\pa_tw(t)+\frac{b(t)}{\lambda(t)}((1-\chi_{B_1(b(t))})\Lambda Q))_{\lambda(t)}\|_{L^2\times L^2} \lesssim  b_0 |\log b_0|,
\ee
\be
\label{eq:idon}
 \frac{b^{4}(t)}{|\log b(t)|^2\lambda^{2(1-\alpha)}(t)}\geq \frac{b^{4}(0)}{|\log b(0)|^2\lambda^{2(1-\alpha)}(0)},
\ee
\be
\label{poitwisebsbootimproved}
|b_s|^2\leq \frac{K(M)}{2}\frac{b^{4}}{|\log b|^2},
\ee
\be
\label{poitnwiseboundbootimproved}
|\mathcal E(t)|\leq  \frac{K(M)}{2}\frac{b^{4}}{(\log b)^2}.
\ee
\end{lemma}

{\bf Proof of Lemma \ref{lemmaimprovedbounds}}\\

{\bf step 1} Energy bound.\\

The energy bound  \fref{energyboundbootted} is a consequence of the conservation of the energy. Indeed, the conservation of the energy and the initial bounds of Lemma \ref{definitionadmissible} ensure $$E(u,\pa_t u)=E(u_0,u_1)=E(Q)+  O(b_0\sqrt{|\log b_0|}), 
$$ 
(see Appendix \ref{definitionadmissible}) and thus: 
\bea
\label{expansionenergyone}
& & E(Q)+O(b_0|\log b_0|) 
\\
\nonumber & = &  \frac 12\int\left[\pa_t(P_{B_1})_{\lambda}+\pa_tw\right]^2+\frac 12\int\left|\nabla (P_{B_1})_{\lambda}+\nabla w\right|^2-\frac{1}{4}\int\left[(P_{B_1})_{\lambda}+w\right]^4.
\eea
We lower bound the first term by  expanding, 
\bee
\partial_t (P_{B_1})_{\lambda} + \partial_ t w &=& \partial_t w + \dfrac{b}{\lambda}((1-\chi_{B_1})\Lambda Q)_{\lambda} + \dfrac{b}{\lambda} (\chi_{B_1}\Lambda Q)_{\lambda} + \dfrac{b^3}{\lambda} (\Lambda [\chi_{B_1}T_1])_{\lambda} + \dfrac{b_s}{\lambda} (\partial_b P_{B_1})_{\lambda}\\
&=&  \partial_t w + \dfrac{b}{\lambda}((1-\chi_{B_1})\Lambda Q)_{\lambda} + \Sigma,
\eee
with 
$$
\int \Sigma^2 \lesssim b_0^2 |\log b_0|,
$$ 
where we used the bootstrap bounds \fref{controllambdaboot}, \fref{poitwisebsboot}. 
Finally :
\be
\label{pouetzero}
 \int\left[\pa_t(P_{B_1})_{\lambda}+\pa_tw\right]^2  \geq  \frac12 \int\left[\dfrac{b}{\lambda}((1-\chi_{B_1})\Lambda Q)_{\lambda}+\pa_tw\right]^2- O(b_0^2|\log b_0|).
\ee

We then expand the second term:
\bee
& & \frac 12\int\left[\nabla (P_{B_1})_{\lambda}+\nabla w\right]^2-\frac{1}{4}\int\left[(P_{B_1})_{\lambda}+w\right]^4=\frac 12\int\left[\nabla P_{B_1}+\nabla \e\right]^2-\frac{1}{4}\int\left[P_{B_1}+ \e\right]^4\\
& = & \frac 12\int\left|\nabla P_{B_1}\right|^2-\frac{1}{4}\int\left|P_{B_1}\right|^4 -(\e,\Delta P_{B_1}+P_{B_1}^3)+\frac12\left(\int |\nabla \e|^2-3\int P_{B_1}^2\e^2\right)\\
& - & \frac14\left(4P_{B_1} \e^3 + \e^4 
\right)
\eee
From the construction of $P_{B_1}$, 

\be
\label{pouetdeux}
\frac 12\int\left|\nabla P_{B_1}\right|^2-\frac{1}{4}\int\left|P_{B_1}\right|^4=E(Q)+
O(b^2 |\log b|).
\ee 

The linear term is treated using \fref{estpsibloin}, the improved decay \fref{cancelaationphi} and \fref{roughpsiboine}:
\bea
\nonumber
\left |(\e,\Delta P_{B_1}+P_{B_1}^3)\right|&=&\left|(\e,b^2D\Lambda P_{B_1}-\Psi_{B_1})\right|\\
\label{pouetun}
&\lesssim& \|\e/y\|_{L^2} \|y(b^2D\Lambda P_{B_{1}} - \Psi_{B_1})\|_{L^2}\lesssim b|\nabla \e|_{L^2} 
\eea
We now rewrite the quadratic term as a small deformation of $H$ and use the coercivity bound \fref{coercH} to ensure:
\be
\label{pouettrois}
\int |\nabla \e|^2-3\int P_{B_1}^2\e^2\geq c_0\int|\nabla \e|^2 + Def,
\ee
with
$$
Def :=  3 \int (Q^2 - P_{B_1}^2) \e^2 - \dfrac{(\e, \psi)^2}{c_0}.
$$
Collecting \eqref{esttone} and \fref{hardy0}, on the one hand, and \fref{estpsiscal} on the other hand, we compute:
\begin{equation} \label{yae}
\left|\int (Q^2 - P_{B_1}^2) \e^2\right| \leq \|y^2(Q^2-P_{{B}_1}^2)\|_{L^{\infty}} \|\nabla \e\|_{L^2}^2 \lesssim b \|\nabla \e\|^2_{L^2},
\qquad 
|(\e,\psi)|^2 \lesssim {b^2}{|\log b|}.
\end{equation}
The nonlinear term is easily estimated from Sobolev:
\be
\label{pouetfour}
\int \left| (3 P_{B_1} + \e)\e^3 \right|\leq \|y P_{B_1} \|_{L^{\infty}}\|y  \e \|_{L^{\infty}} \|\nabla \e\|_{L^2}^2 \lesssim \sqrt{b_0}  \|\nabla \e\|_{L^2}^2
\ee
Injecting \fref{pouetzero}, \fref{pouetun}, \fref{pouetdeux}, \fref{yae}, \fref{pouettrois}, \fref{pouetfour} into \fref{expansionenergyone} yields \fref{energyboundbootted}.\\

{\bf step 2} Lower bound on $b$.\\

We now turn to the proof of \fref{eq:idon}. First observe from the bootstrap estimate \fref{poitwisebsboot} that 
\be
\label{rgouboundbs}
|b_s|\leq \sqrt{K(M)}\frac{b^2}{|\log b|}\leq \frac{1-\alpha}{10}b^2
\ee
This implies:
\bee
\frac{d}{ds}\left(\frac{b^{4}}{(\log b)^2\lambda^{2(1-\alpha)}}\right)=\frac{4b^3}{\lambda^{2(1-\alpha)}(\log b)^2}\left[b_s\left(1-\frac{1}{2\log b}\right)+\frac{1-\alpha}{2}b^2\right]>0
\eee
and \fref{eq:idon} follows.\\

{\bf step 3} Improved $\dot{H}^2$ bound.\\

We now turn to  the proof of  \fref{poitnwiseboundbootimproved}. We integrate \fref{vhoheor} in time and conclude from \fref{estlowerenergy}, \fref{estimatef}:
\bea
\label{chooehewcveieho}
|\mathcal E(t)| & \lesssim & \left(\frac{\lambda(t)}{\lambda(0)}\right)^{2(1-\alpha)}|\mathcal E(0)|\\
\nonumber &+ & (K(M))^{\frac 12}\left[\frac{b^4(t)}{|\log b(t)|^2}+[\lambda(t)]^{2(1-\alpha)}\int_0^t\frac{b(\tau)}{[\lambda(\tau)]^{3-2\alpha}}\frac{b^4(\tau)}{|\log b(\tau)|^2}d\tau\right]
\eea
We then derive from \fref{rgouboundbs}:
\bee
&& \int_0^t\frac{b(\tau)}{[\lambda(\tau)]^{3-2\alpha}}\frac{b^4(\tau)}{|\log b(\tau)|^2}d\tau =-\int_0^t\frac{\lambda_t}{\lambda^{3-2\alpha}}\frac{b^4}{|\log b|^2}d\tau\\
& \leq & \frac{1}{2(1-\alpha)}\frac{b^4(t)}{\lambda^{2(1-\alpha)}(t)|\log b(t)|^2} -\frac{1}{2(1-\alpha)}\int_0^t\frac{b_s}{\lambda^{3-2\alpha}}\frac{b^3}{|\log b|^2}\left[1-\frac{2}{|\log b|^2}\right]\\
& \lesssim & \frac{b^4(t)}{\lambda^{2(1-\alpha)}(t)|\log b(t)|^2} +\sqrt{K(M)}\int_0^t\frac{b(\tau)}{[\lambda(\tau)]^{3-2\alpha}}\frac{b^4(\tau)}{|\log b(\tau)|^2}\frac{1}{|\log b(\tau)|}d\tau
\eee
and hence the bound:
$$\lambda^{2(1-\alpha)}(t)\int_0^t\frac{b(\tau)}{[\lambda(\tau)]^{3-2\alpha}}\frac{b^4(\tau)}{|\log b(\tau)|^2}d\tau\lesssim \frac{b^4(t)}{|\log b(t)|^2}.
$$
Injecting this into \fref{chooehewcveieho} and using the initial bound \fref{initiale0_1},\fref{initiale0_2} and the monotonicity \fref{eq:idon} yields:
\bea
\label{borneimtermemergy}
\nonumber \mathcal{E}(t) & \lesssim & \left(\frac{\lambda(t)}{\lambda(0)}\right)^{2(1-\alpha)}\frac{b^{4}(0)}{|\log b(0)|^2}+(K(M))^{\frac 12}\frac{b^4(t)}{|\log b(t)|^2}\\
& \lesssim & \sqrt{K(M)}\frac{b^4(t)}{|\log b(t)|^2}
\eea
and   \fref{poitnwiseboundbootimproved} follows. \fref{poitwisebsbootimproved} now follows from \fref{roughboundpointw} and \fref{borneimtermemergy}.\\
This concludes the proof of Lemma \ref{lemmaimprovedbounds}.


\subsection{Dynamic of the unstable mode}


We now focus onto the dynamic of the unstable mode. We recall the decomposition 
\be
\label{decompy}
Y(t)=\left|\begin{array}{ll}(\e,\psi)\\(\partial_s\e,\psi)\end{array} \right .=\tilde{a}_+(t)V_++\tilde{a}_-(t)V_-,
\ee
and the variables given by \fref{aplusamoins}:
$$\kappa_+(s)=\tilde{a}_+(s)+\frac{b_s}{2\sqrt{\zeta}}({\partial_b P_{B_1}},\psi), \ \ \kappa_-(s)=\tilde{a}_-(s)-\frac{b_s}{2\sqrt{\zeta}}({\partial_b P_{B_1}},\psi).
$$

\begin{lemma}[Control of the unstable mode]
\label{lemmaunstable}
There holds: for all $t\in [0,T_1(a_+)]$, 
\be
\label{estaminusbis}
|\kappa_-(t)|\leq \frac{1}{2}(K(M))^{\frac18}\frac{b^2}{|\log b|}
\ee
and $\kappa_+$ is {\it strictly outgoing}:
\be
\label{outgoingdirection}
\left|\frac{d \kappa_+}{ds}-\sqrt{\zeta}\kappa_+\right|\leq \sqrt{b}\frac{b^2}{|\log b|}.
\ee
\end{lemma}

{\bf Proof of Lemma \ref{lemmaunstable}}\\

We compute the equation satisfied by the unstable direction $(\e,\psi)$ by taking the inner product of \fref{eqeqb} with the well localized direction $\psi$ to get:
\be
\label{defunstbakle}
\frac{d^2}{ds^2}(\e,\psi)-\zeta(\e,\psi)=E(\e)-(\partial^2_sP_{B_1},\psi)
\ee 
with
\bea
\label{defereste}
\nonumber  E(\e) & = &   -(\Psi_{B_1},\psi)-b_s(\Lambda P_{B_1},\psi)-b(\partial_sP_{B_1}+2\Lambda\partial_sP_{B_1},\psi) -  b(\partial_s\eb+2\Lambda\partial_s\eb,\psi)\\
&-& b_s (\Lambda \eb,\psi)+(N(\eb),\psi)   + b^2(\Lambda \e,D\psi) + ((f'(P_{B_1}) - f'(Q))\e,\psi).
\eea 
Simple algbebraic manipulations using \fref{decompy}, \fref{aplusamoins} and the initial condition yield the equivalent system:
\be
\label{dymniamocsncos}
\frac{d}{ds}\kappa_+=\sqrt{\zeta}\kappa_+(s)+\frac{E_+ (s)}{2\sqrt{\zeta}},
\quad
\frac{d}{ds}\kappa_-=-\sqrt{\zeta}\kappa_-(s)-\frac{E_-(s)}{2\sqrt{\zeta}}\\ \kappa_-(0)
\ee
with 
\be
\label{defetwojvopfj}
E_+(s)=E(s)- \frac{b_s}{2}({\partial_b P_{B_1}},\psi)
\qquad
E_-(s) = E(s)+ \frac{b_s}{2}({\partial_b P_{B_1}},\psi)
\ee
We now have from the explicit formula \fref{defereste}, \fref{defetwojvopfj}, the exponential localization of $\psi$, the orthogonality $$(\psi,\Lambda Q)=0,$$ the estimates of Proposition \ref{propqb} and the bootstrap estimate \fref{poitwisebsboot} the bound:
\be
\label{controleone}
\frac{1}{\sqrt{\zeta}}|E_{\pm}|\lesssim |b|(|b_s|+\sqrt{|\mathcal E|}+\sqrt{K(M)}\frac{b^2}{|\log b|})\leq \sqrt{b}\frac{b^2}{|\log b|},
\ee 
which together with \fref{dymniamocsncos} yields \fref{outgoingdirection}. Let then $$\mathcal G=\kappa_-^2\frac{|\log b|^2}{b^4},$$ then from \fref{dymniamocsncos}, \fref{controleone}, \fref{poitwisebsboot}, we estimate:
\bee
\frac{d\mathcal G}{ds} & = & 2\kappa_-\frac{d\kappa_-}{ds}\frac{|\log b|^2}{b^4}+
\kappa_-^2b_s\left[-\frac{4|\log b|^2}{b^5}+\frac{2\log b}{b^5}\right]\\
& = & 2\frac{|\log b|^2}{b^4}\left[\kappa_-\left(-\sqrt{\zeta}\kappa_--\frac{E_-}{\sqrt{\zeta}} \right)  \right]+\kappa_-^2 \, \frac{|\log b|^2}{b^4}\, O\left(\frac{|b_s|}{b}\right) \\
& \leq & -\frac{\sqrt{\zeta}}{2}\frac{|\log b|^2}{b^4}\kappa_-^2+\frac{|\log b|^2}{b^4}\kappa_-\sqrt{b}\frac{b^2}{|\log b|}\lesssim  -\frac{\sqrt{\zeta}}{2}\mathcal G+1.
\eee
We integrate this in time
$$\mathcal G(s)\leq \mathcal  G(0)e^{-\frac{\sqrt{\zeta}}{2}s}+\int_{0}^s e^{-\frac{\sqrt{\zeta}}{2}(s-\sigma)}d\sigma\lesssim 1 
$$
where we used the initial inequality \fref{initialG} yielding that $\mathcal G(0)\lesssim 1.$ This concludes the proof of \fref{estaminusbis} and of Lemma \ref{lemmaunstable}.


\subsection{Derivation of the sharp law for $b$}
\label{sectionbsharp}

We now turn to the derivation of the sharp law for $b$ which will yield the required monotonicity statement on $b$ to close the smallness bootstrap estimate \fref{controllambdaboot}, and will eventually lead to the derivation of the sharp blow up speed \fref{lawlambda}.\\

\begin{lemma}[Sharp derivation of the $b$ law]
\label{lemmaalgebra}
Let 
\be
\label{defpbtzero}
\pbt=\chi_{\frac{B_0}{4}}Q,
\ee
\be
\label{deffb}
G(b)=b|\Lambda \pbt|_{L^2}^2+\int_0^b \tilde b({\partial_b \pbt},\Lambda \pbt)d\tilde b,
\ee
\be
\label{defiun}{\mathcal I}(s)=(\partial_s\eb,\Lambda \pbt)+b(\e+2\Lambda\eb,\Lambda  \pbt)+b_s({\partial_b  \pbt},\Lambda  \pbt)-b_s \left({\partial_b}(P_{B_1}- \pbt), \Lambda  \pbt)\right ),
\ee
then there holds:
\be
\label{estkeyfb}
G(b)=64b |\log b| +O(b), \ \ |{\mathcal{I}}|\lesssim K(M)b,
\ee
\be
\label{firstcontrol}
\left|\frac{d}{ds}\{G(b)+{\mathcal{I}}(s)\}+32b^{2}\right|\lesssim K(M)\frac{b^{2}}{\sqrt{|\log b|}}.
\ee
\end{lemma}

\begin{remark} Observe that \fref{estkeyfb}, \fref{firstcontrol} essentially yield a pointwise differential equation $$b_s\sim -\frac{b^2}{2|\log b|}	$$
								 which will allow us to derive the sharp scaling law via the 
								 relationship $-\lsl=b$. 
\end{remark}

{\bf Proof of Lemma \ref{lemmaalgebra}}\\

The proof is inspired by the one in \cite{RaphRod}. We multiply (\ref{eqeqb}) with $ \Lambda \pbt$ and compute:

\bea
\nonumber & & (b_s\Lambda \qbtb+b(\partial_s\qbtb+2\Lambda\partial_s\qbtb)+\partial_s^2P_{B_1},\Lambda  \pbt)=-(\Psi_{B_1}, \Lambda \pbt)-(H_{B_1}\e,\Lambda  \pbt)\\
& - &  \left( \partial_s^2\eb+b(\partial_s\eb+2\Lambda\partial_s\eb)+b_s \Lambda \eb,\Lambda  \pbt\right)+(N(\e),\Lambda  \pbt)\notag
\eea
We further rewrite this as follows:
\bea\label{vhoheoheo}
\nonumber & & (b_s\Lambda  \pbt+b(\partial_s  \pbt+2\Lambda\partial_s  \pbt)+\partial_s^2 \pbt,\Lambda  \pbt)=-(\Psi_{B_1},\Lambda \pbt)\\&-&
(b_s\Lambda (P_{B_1}- \pbt)+b(\partial_s (P_{B_1}- \pbt)+2\Lambda\partial_s (P_{B_1}- \pbt))+\partial_s^2(P_{B_1}- \pbt),\Lambda  \pbt)\notag
\\
& - & (H_{B_1}\e,\Lambda \pbt)- \left( \partial_s^2\eb+b(\partial_s\eb+2\Lambda\partial_s\eb)+b_s \Lambda \eb,\Lambda \pbt\right)+(N(\e),\Lambda\pbt).
\eea

We now estimate all terms in the above identity.\\

{\bf step 1} $b$ terms.\\

An integration by parts in time allows us to rewrite the left-hand side of \fref{vhoheoheo} as follows:
\be
\label{rewritnghgllhs}
(b_s\Lambda  \pbt+b(\partial_s  \pbt+2\Lambda\partial_s  \pbt)+\partial_s^2 \pbt,\Lambda  \pbt)\\ 
  =  \frac{d}{ds}\left[G(b)+b_s({\partial_b  \pbt},\Lambda  \pbt)\right]+|b_s|^2|{\partial_b  \pbt}|_{L^2}^2
 \ee
 with $G$ given by \fref{deffb}. Observe from \fref{poitwisebsboot} the bound $$|b_s|^2|{\partial_b  \pbt}|_{L^2}^2\lesssim \frac{|b_s|^2}{b^2}\lesssim (K(M))^2\frac{b^2}{|\log b|^2}\lesssim \frac{b^2}{\sqrt{|\log b|}}.$$
We now turn to the  key step in the derivation of the sharp $b$ law which corresponds to the following outgoing flux computation\footnote{see again \cite{RaphRod} for more details about the flux computation statemement and its connection to the Pohozaev integration by parts formula}:
\be
\label{ourgoingflux}
(\Psi_{B_1},\Lambda \pbt)=32b^{2}\left(1+O\left(\frac{1}{{|\log b|}}\right)\right) \ \ \mbox{as} \ \ b\to 0.
\ee
Indeed, we first estimate from \fref{estpsibloin}:
\bee
\left|(\Psi_{B_1} - c_b b^2\chi_{\frac{B_0}{4}} \Lambda Q,\Lambda  \pbt)\right|  & \lesssim  & b^4\int_{y\leq \frac{B_0}{2}}\left[\frac{1+|\log (by)|}{|\log b|(1+y^2)}+\frac{1+|\log (1+y)|}{(1+y^2)^2}\right]\\
& \lesssim & \frac{b^2}{|\log b|}.
\eee
The remainder term is computed from \fref{computationcb} and the explicit formula for $Q$ \fref{defqexplicite}:
\bee
(c_bb^2\chi_{\frac{B_0}{4}}\Lambda Q,\Lambda \pbt) & = & \frac{b^2}{2|\log b|}\left(1+O\left(\frac{1}{|\log b|}\right)\right)\left[\int_{y\leq \frac{1}{2b}}(\Lambda Q)^2+O(1)\right]\\
& = & 32b^2\left(1+O\left(\frac{1}{|\log b|}\right)\right),
\eee
and \fref{ourgoingflux} follows.\\
We now estimate the lower order terms in $b$ which correspond to the second line of \eqref{vhoheoheo}. One term is reintegrated by parts in time:
$$-(\partial_s^2(P_{B_1}- \pbt),\Lambda  \pbt)=-\frac{d}{ds}\left\{b_s(\partial_b(P_{B_1}-\pbt),\Lambda \pbt)\right\}+b_s^2(\partial_b(P_{B_1}-\pbt),{\partial_b \Lambda \pbt}).$$ The remaining terms are estimated in brute force using \fref{estpbcut} and \fref{poitwisebsboot} which yield:
\bee
& & \left|(b_s\Lambda (P_{B_1}- \pbt)+b(\partial_s (P_{B_1}- \pbt)+2\Lambda\partial_s (P_{B_1}- \pbt))),\Lambda  \pbt)\right|\\
& + & b_s^2\left|({\partial_b}(P_{B_1}-\pbt),{\partial_b \Lambda \pbt})\right| \lesssim  |b_s|+\frac{|b_s|^2}{b^2}\lesssim K(M)\frac{b^2}{|\log b|}.
\eee

{\bf step 2} $\e$ terms .\\

We are left with estimating the third line on the RHS of \fref{vhoheoheo}. We first treat the linear term from \fref{estlowerenergy}, \fref{estcepueip}, \fref{poitnwiseboundboot}:
\be
\left|(H_{B_1}\e,\Lambda \pbt)\right|  \lesssim  |(H\e,\Lambda \pbt)|+\int|\e||P_{B_1}^2-Q^2||\Lambda \pbt| + b^2 \left|(D\Lambda \e, \Lambda \pbt)\right|
\ee
On the one hand, \eqref{estcepueip} together with bootstrap estimates yield: 
\bee
\int|\e||P_{B_1}^2-Q^2||\Lambda \pbt| &\lesssim &  b^2 \int_{y \leq B_0}  \frac{|\e|}{(1+y^2)^2} \leq b^{\frac 32} \left(
\int \dfrac{|\e|^2}{(1+y)^5}\right)^{\frac{1}{2}}\\
							& \lesssim & \dfrac{b^{2}}{|\log(b)|}
\eee
On the other hand, after integration by parts, we repeat the same arguments and apply \eqref{harfysanslog}.
This  yield:
\bee
b^2\left|(D\Lambda \e, \Lambda \pbt)\right| & \leq & b^2 \int_{y\leq B_0} \dfrac{|\e|}{(1+y^4)} +
b^2 \int_{B_0/4 \le y\leq B_0/2}\frac{|\e|}{(1+y^2)} + b^2\int_{y\leq B_0} |\nabla\e|\frac{y}{1+y^2} \\
&\lesssim & b^{\frac 32} \left(
\int \dfrac{|\e|^2}{(1+y^5)}\right)^{\frac{1}{2}} +\left( \int_{B_0/4 \le y\leq B_0/2}\frac{|\e|^2}{(1+y^4)}\right)^{\frac 12}
+  \left(\int_{y\leq B_0} \dfrac{|\nabla\e|^2}{1+y^2}\right)^{\frac 12}\\
&\lesssim & \sqrt{|\log (b)|} \left(c(M) |\mathcal E| + \sqrt{K(M)} \dfrac{b^4}{|\log b|^2} \ \right)^{\frac 12} \\
&\lesssim & \sqrt{K(M)} \dfrac{b^2}{\sqrt{| \log (b)|}}.  
\eee
Finally:
\bee
|(H\e, \Lambda \pbt)|
& \lesssim &  |H\e|_{L^2}\sqrt{|\log b|}+\sqrt{K(M)} \dfrac{b^2}{\sqrt{ \log (b)}} \\
& \lesssim & \sqrt{|\log b|}\sqrt{|\mathcal E|+\sqrt{K(M)}\frac{b^4}{|\log b|^2}}\lesssim \sqrt{K(M)}\frac{b^2}{\sqrt{|\log b|}}
\eee
We further integrate by parts in time to obtain:
\bee
\nonumber \left( \partial_s^2\eb+b(\partial_s\eb+2\Lambda\partial_s\eb)+b_s \Lambda \eb,\Lambda \pbt\right)
&=&\frac{d}{ds}\left[(\partial_s\eb,\Lambda \pbt)+b(\e+2\Lambda\eb,\Lambda \pbt)\right]\\
&-& b_s\left[(\partial_s\eb+b\Lambda \e,\Lambda {\partial_b \pbt})+(\e,\Phi_b)\right]
\eee
with
$$\Phi_b=-\Lambda \pbt-\Lambda^2\pbt-b\Lambda{\partial_b \pbt}-b\Lambda^2{\partial_b \pbt}.
$$
We thus estimate from \fref{estlowerenergy}, \fref{defeta}, \fref{estcepueip}, \fref{poitwisebsboot}, \fref{poitnwiseboundboot}:
\bee
& & |b_s|\left|(\partial_s\eb+b\Lambda \e,\Lambda {\partial_b \pbt})+(\e,\Phi_b)\right|\lesssim  |b_s|\left[\int_{\frac{B_0}{4}\leq y\leq B_0} \frac{|\eta|}{y}+\int_{y\leq B_0}\frac{|\e|}{1+y^2}\right]\\
& \lesssim & \frac{|b_s||\log b|}{b^2}C(M)
\sqrt{|\mathcal E|+\sqrt{K(M)}\frac{b^4}{|\log b|^2}}\lesssim K(M)\frac{b^2}{\sqrt{|\log b|}}
\eee

The non linear term is estimated as previously. Indeed, we have:
\bee
\left|(N(\e),\Lambda \pbt)\right|  & \lesssim & \int (|P_{B_1}| + |\e|)\e^2 |\Lambda \pbt|\\
												& \lesssim & \dfrac{1}{b^2} \|y (|P_{B_1}| +|\e|)\|_{L^{\infty}} \|(1+y^2)\Lambda \pbt\|_{L^{\infty}}\int_0^{B_0} \dfrac{|\e|^2}{y(1+y^4)} \\
												& \lesssim &  \dfrac{C(M)}{b^2} \left[ \mathcal E + K(M)\dfrac{b^4}{|\log b|^2}\right]
												 \lesssim  K(M)\frac{b^2}{\sqrt{|\log b|}}.
\eee

{\bf step 5} Control of $G(b)$ and ${\mathcal{I}}$.\\

Injecting the estimates of step 1 and step 2 into \fref{vhoheoheo} yields  \fref{firstcontrol}. It remains to prove \fref{estkeyfb}. The estimate for $G(b)$ is a straightforward consequence of the choice \fref{defpbtzero} and the explicit formula \fref{defqexplicite}. It remains to control $\mathcal I$. We integrate by parts in space in \fref{defiun} to rewrite:
$${\mathcal I}(s)=(\partial_s\e+b\Lambda \e,\Lambda \pbt)-b(\e,\Lambda \pbt+\Lambda^2\pbt)+b_s({\partial_b \pbt},\Lambda \pbt)-b_s \left({\partial_b}(P_{B_1}-\pbt), \Lambda \pbt)\right ).
$$
The $b$ terms are estimated as in step 1:
$$|b_s|\left|({\partial_b \pbt},\Lambda \pbt)-({\partial_b}(P_{B_1}-\pbt), \Lambda \pbt)\right|\lesssim  \frac{|b_s|}{b}\lesssim b.$$ 
The linear term is estimated using \fref{estlowerenergy}, \fref{defeta}, \fref{estcepueip}, \fref{poitwisebsboot}, \fref{poitnwiseboundboot}:
\bee
& & \left|(\partial_s\e+b\Lambda \e,\Lambda \pbt)-b(\e,\Lambda \pbt+\Lambda^2\pbt)\right|\lesssim   \int_{y\leq B_0}\frac{|\eta|}{y^2}+b\int_{y\leq B_0}\frac{|\e|}{y^2}\\
& \lesssim & \frac{1}{b}\left(\int\frac{|\eta|^2}{y^2}\right)^{\frac{1}{2}}+\frac{|\log b|}{b^2}\left(\int_{y\leq B_0}\frac{|\e|^2}{y^4(1+|\log y|^2)}\right)^{\frac{1}{2}}\lesssim  {K(M)} b
\eee
and \fref{estkeyfb} is proved.\\

This concludes the proof of Lemma \ref{lemmaalgebra}.


\section{Sharp description of the singularity formation}
\label{sectionfour}


We are now in position to conclude the proof of Proposition \ref{propcle} and Theorem \ref{thmmain} as a simple consequence of the a priori bounds obtained in the previous section.The proof relies on a topological argument which closes the bootstrap argument, and then the sharp description of the blow up dynamic is a consequence of the a prori bounds obtained on the solution and in particular the modulation equation \fref{firstcontrol}.\\

{\bf Proof of Proposition \ref{propcle}}\\

We argue by contradiction and assume that for all $a_+\in\left[-\frac{b_0^2}{|\log b_0|},\frac{b_0^2}{|\log b_0|}\right],$ $$T_1(a_+)<T(a_+).$$
In view of the Definition \ref{lemmaimprovedbounds} of the bootstrap regime and the improved bounds of Lemma \ref{lemmaimprovedbounds} and Lemma \ref{lemmaunstable}, a simple continuity argument ensures that $T_1(a_+)$ is attained at the first time $t$ where
\be
\label{cneoheoheoei}
|\kappa_+(t)| = \dfrac{|b(t)|^2}{2|\log(b(t))|}.
\ee
The fundamental fact now is the outgoing behaviour \fref{outgoingdirection} which together with \fref{cneoheoheoei} ensures $$\left|\frac{d\kappa_+}{dt}(T_1(a_+))\right|>0.$$ Thus from standard argument\footnote{see \cite[Lemma 6]{cotemartelmerle} for a complete exposition}, the map $$
\begin{array}{rcl}
\left[-\frac{b_0^2}{|\log b_0|},\frac{b_0^2}{|\log b_0|}\right] & \to & \Bbb R*_+ \\
a_+  & \mapsto & T_1(a_+)
\end{array} \mbox{is continuous}.
$$
We may thus consider the continuous map:
$$
\begin{array}{cccc}
\Phi : & [-\frac{b_0^2}{|\log b_0|}, \frac{b_0^2}{|\log b_0|}] & \to &  \mathbb R\\
	     & a_+ & \to &  \kappa_+(T_1(a_+))\frac{2|\log b(T_1(a_+))|}{b^2(T_1(a_+))}
\end{array}
$$ 
On the one hand, \fref{cneoheoheoei} implies:
$$
\Phi \left(  \left[-\frac{b_0^2}{|\log b_0|}, \frac{b_0^2}{|\log b_0|}\right]\right) \subset \{-1,1\}. 
$$
On the other hand, the outgoing behavior \fref{outgoingdirection} together with the initialization $\kappa_+(0)=a_+$ ensures:
$$
\Phi\left(-\frac{b_0^2}{|\log b_0|}\right)=-1, \ \ \Phi\left(\frac{b_0^2}{|\log b_0|}\right)=1
$$ 
and a contradiction follows.\footnote{This topological argument is of course the one dimensional version of Brouwer's fixed point argument used in \cite{cotemartelmerle}.}This concludes the proof of Proposition \ref{propcle}.\\

{\bf Proof of Theorem \ref{thmmain}}\\

{\bf step 1} Finite time blow up and derivation of the blow up speed.\\

\noindent Let from Proposition \ref{propcle} an initial data with $T_1(a_+)=T(a_+)$. We first claim that $u$ blows up in finite time 
\be
\label{taplusfintie}
T=T(a_+)<+\infty.
\ee
Indeed, from \fref{eq:idon}, $$\lambda^{2(1-\alpha)}\lesssim b^3\ \  \mbox{and thus} \ \ \lambda^{\frac 23}\lesssim \lambda^{\frac{2(1-\alpha)}{3}}\lesssim b=-\lambda_t.$$ Integrating this differential inequation yields $$t\lesssim \lambda^{\frac 13}(0)-\lambda^{\frac 13}(t)\lesssim 1$$ and \fref{taplusfintie} follows. The $(\dot{H}^1\cap \dot{H}^2)\times (L^2\cap \dot{H^1})$ bounds \fref{bootneergynorm2}, \fref{poitnwiseboundboot} on $(\e,\pa_t\e)$ and hence on $(u,\pa_tu)$ in the bootstrap regime and standard $H^2$ local well posedness theory ensure that blow up corresponds to $$\lambda(t)\to  0 \ \ \mbox{as} \ \ t\to T(a_+).$$ We now derive the blow up speed by reintegrating the ODE \fref{firstcontrol} and briefly sketch the proof which follows as in \cite{RaphRod}.\\
First recall the standard scaling lower bound $$\lambda (t)\leq C(u_0)(T-t)$$ which implies that the rescaled time is global: $$s(t)=\int_0^t\frac{d\tau}{\lambda(\tau)}\to +\infty \ \ \mbox{as} \ \ t\to T.$$ Let $$\mathcal J=G+\mathcal I$$ so that from \fref{estkeyfb}:
\be
\label{estmathcaJ}
\mathcal J=64 b|\log b|\left(1+O\left(\frac{1}{|\log b|}\right)\right) \ \ \mbox{ie} \ \ b=\frac{\mathcal J}{64|\log \mathcal J|}\left(1+O\left(\frac{1}{\sqrt{|\log \mathcal J|}}\right)\right)
\ee
and $\mathcal J$ satisfies from \fref{firstcontrol} the ODE: 
$$\mathcal J_s+\frac{\mathcal J^2}{128|\log \mathcal J|^2}\left(1+O\left(\frac{1}{\sqrt{|\log \mathcal J|}}\right)\right)=0.
$$We multiply the above by $\frac{|\log \mathcal J|^2}{\mathcal J^2}$, integrate in time and obtain to leading order:$$\mathcal J=\frac{128(\log s)^2}{s}\left(1+O\left(\frac{1}{\sqrt{|\log s|}}\right)\right) \\  \ \mbox{ie} \ \ -\lsl=b=\frac{2\log s}{s}\left(1+O\left(\frac{1}{\sqrt{|\log s|}}\right)\right).$$ where we used \fref{estmathcaJ}. Integrating this once more in time yields: $$-\log \lambda=(\log s)^2\left(1+O\left(\frac{1}{\sqrt{|\log s|}}\right)\right)$$ and thus $$b=-\lambda_t=\exp\left(-\sqrt{|\log \lambda|}\left(1+O\left(\frac{1}{|\log \lambda|^{\frac 14}}\right)\right)\right).$$ Integrating this from $t$ to $T$ where $\lambda(T)=0$ yields the asymptotic $$\lambda(t)=(T-t)\exp\left(-\sqrt{|\log \lambda(t)|}\left(1+O\left(\frac{1}{|\log \lambda(t)|^{\frac 14}}\right)\right)\right)$$ which yields \fref{lawlambda}.\\

{\bf step 2} Energy quantization.\\

It remains to prove \fref{energyquantization} which can be derived exactly as in \cite{RaphRod}, this is left to the reader.\\ 
This concludes the proof of Theorem \ref{thmmain}.


\begin{appendix}


\section{Modulation theory}
\label{app_modulation}

This appendix is devoted to the proof of Lemmas \ref{definitionadmissible} and \ref{smalldata}.
The arguments are standard in the framework of modulation theory and we briefly sketch the main computations.\\

\subsection{Proof of Lemma \ref{definitionadmissible}}
First note that the bounds  $$|\nabla ({P_{B_1}} -Q) |_{L^2}+b|\Lambda P_{B_1} - b (1-\chi_{B_1}) \Lambda Q|_{L^2}\lesssim b|\log b|$$ ensure that our initial data are of the form 
$$u_0=Q+\tilde{\eta}_0, \ \ u_1=\tilde{\eta}_1$$
for a small excess of energy in the sense that:
\be \label{petitesse}
\|\nabla\tilde{\eta}_0,\tilde{\eta}_1\|_{L^2\times L^2} \lesssim b_0 |\log b_0|, 
\ \ 
\|\nabla^2 \tilde{\eta}_0,\nabla\tilde{\eta}_1\|_{L^2\times L^2}\lesssim b_0,
\ee
Hence the continuity of the flow associated to \eqref{wave} ensures the existence of a time $T_0>0$ (uniform in $\tilde{\eta}_0,\tilde{\eta_1}$)
for which the solution $u$ to \eqref{wave}  $(u_0,u_1)$ satisfies on $[0,T_0]$:
\be \label{boundsol}
\sup_{[0,T_0]} \|\nabla(u-Q),\partial_tu \|_{L^2\times L^2} \lesssim b_0  |\log b_0|, . 
\ee

\noindent{\bf Step 1 :} Modulation near $Q.$ \\ 
The non degeneracy
$(\Lambda Q,\Phi)\neq 0$ ensures\footnote{as a direct consequence of the implicit function theorem and the smoothness of the flow \fref{wave}} that $u$  admits on $[0,T_0]$ a decomposition 
\be
\label{defdecomp}
u(t)=(Q+\tilde{\e}(t))_{\lambda(t)}
\ee 
with: 
\be
\label{orthotildebis}
(\tilde{\e}(t),\chi_M\Phi)=0. 
\ee 
Moreover, $\lambda \in \mathcal{C}^2([0,T_0]; \mathbb R^*_+)$ and noting that $\tilde{\eta}_0$ satisfies 
$$
|(\tilde{\eta}_0,\chi_M \Phi) | \lesssim \frac{b_0^2}{|\log b_0|},
$$
we obtain the bound: 
\be
\label{firstbound}
|\lambda(0)-1|\lesssim\frac{b_0^2}{|\log b_0|}.
\ee
We then let $b(t)=-\lambda_t(t)$ on $[0,T_0].$\\

\noindent{\bf Step 2 :} Positivity of $b.$ \\ 
Straightforward computations yield:
$$
\partial_t \tilde{\e}(t) = \left(\partial_t u - \dfrac{b(t)}{\lambda(t)} \Lambda u \right)_{\frac{1}{\lambda(t)}}  .
$$
Taking the scalar product with $\chi_M \Phi,$ we obtain at the initial time:
\be
\label{nitnisnodbweor}
b(0) = \lambda(0) \dfrac{((u_1)_{\frac{1}{\lambda(0)}}, \chi_M \Phi)}{ ( (\Lambda u_0)_{\frac{1}{\lambda(0)}} , \chi_M\Phi)},
\ee
where  \eqref{defpb} together with \eqref{firstbound} imply:
\bea
((u_1)_{\frac{1}{\lambda(0)}}, \chi_M \Phi) &=& b_0 (\Lambda Q,\chi_M \Phi) + O\left( \frac{b_0^2}{|\log(b_0)|}\right), 
\\
((\Lambda u_0)_{\frac{1}{\lambda(0)}}, \chi_M\Phi) &=& (\Lambda Q, \chi_M \Phi) + O\left( {b_0^2}{|\log(b_0)|}\right).
\eea
This yields the positivity of $b(0)$ and moreover, the positivity of $b(t)$ for small time together with:
\be \label{controlb0precis}
b(t) = b_0 + O\left(\dfrac{b_0^2}{|\log(b_0)|}\right)
\ee
As $b>0,$ we may introduce the decomposition:
\be
\label{defdecompbis}
u(t)=(Q+\tilde{\e})_{\lambda(t)}=(P_{B_1(b(t))}+\e)_{\lambda(t)} \ \ \mbox{ie} \ \ \e(t)=\tilde{\e}(t)-(P_{B_1(b(t))}-Q).
\ee
Observe from \fref{orthotone}, \fref{orthotildebis} that 
\be
\label{orthoeappendix}
\forall \, t\in [0,T_0], \ \  (\e(t),\chi_M\Phi)=0.
\ee
The uniqueness of such a decomposition is guaranteed by the (local) uniqueness of $(\lambda,\tilde{\e}).$\\

\noindent{\bf Step 3 :} Smallness of $\e.$ \\ 
To complete the proof, we obtain the smallness of $\e$ in $\dot{H}^1$ and $\dot{H}^2.$ To this end, we note that:
$$
\e(0) = \left( u_0 \right)_{\frac{1}{\lambda(0)}} - P_{B_1(b(0))} = \left[ \left( P_{B_1(b_0)}\right)_{\frac{1}{\lambda(0)}} -  P_{B_1(b(0))} \right] + \left(\eta_0  +d_+ \psi \right)_{\frac{1}{\lambda(0)}}  .
$$
Simple computations based on the estimates of Proposition \ref{propqb}  yield the expected result :
\be \label{initiale0_1}
\|\nabla \e(0)\|_{L^2} \lesssim b_0|\log(b_0)|\qquad  \left\|\dfrac{\e(0)}{1+y^4}\right\|_{L^2} + \|\nabla^2 \e(0)  \|_{L^2} \lesssim \dfrac{b_0^2}{|\log(b_0)|}.
\ee

\subsection{Proof of Lemma \ref{smalldata}}
 The proof of this lemma is divided into two steps.
First, given $(\eta_0,\eta_1,d_+)$ satisfying smallness condition
\eqref{defmotionintia} for small $b_0,$ we prove that $b,b_s$ and $w$ satisfy \eqref{controllambdaboot}--\eqref{poitnwiseboundboot}. Then, we show that,
given $(b_0,\eta_0,\eta_1)$, we can apply the inverse mapping theorem to  $d_+ \mapsto \kappa_+(0)$  close to $0.$
The arguments are standard and we refer to \cite{cotemartelmerle} for a detailed proof in a similar setting.\\

\noindent{\bf Step 1:} Smallness of initial modulation given $(\eta_0,\eta_1,d_+)$.\\
Given  $(\eta_0,\eta_1,d_+)$ satisfying smallness condition \eqref{defmotionintia} we can apply Lemma \ref{defmotionintia} this yields $T_0$ and $b,\e,w$ such that
\eqref{controllambdaboot} holds and 
\be \label{firstboundw}
\|\nabla w(t)\|_{L^2} \lesssim b_0|\log(b_0)| \quad \|\nabla^2w (t)\|_{L^2} \lesssim \dfrac{b_0^2}{|\log(b_0)|^2} 
\ee 
We emphasize in particular that Lemma \ref{definitionadmissible} implies $ b_0/2<b(0) < 2b_0$ for sufficiently small $b_0.$

As previously, we focus now on bounds satisfied initially. We first compute $b_s(0)$ using \fref{wave} and the orthogonality condition \fref{orthoeappendix}.  
Recalling that $(\partial_b^{k} P_{B_1}, \chi_M \Phi) = (\partial_s^{k-1} \e, \chi_M \Phi) = 0$ for any integer $k,$ we get like for \fref{elgebraba}:
\begin{multline*} b_s\left[(\Lambda P_{B_1},\chi_M\Phi)+2 b( \Lambda \partial_b P_{B_1},\chi_M\Phi)+(\Lambda\ebo,\chi_M\Phi)\right]\\
=  -(\Psi_{B_1},\chi_M\Phi) -  (\ebo,H_{B_1}^*(\chi_M\Phi))+b(\partial_s\ebo, \Lambda (\chi_M\Phi)) + (N(\ebo),\chi_M\Phi)
\end{multline*}
where, denoting LHS and RHS the left-hand and right-hand side at initial time, we compute, for  sufficiently small $b_0$ w.r.t. $M$ : 
\be
\label{cwouweofieo}
|RHS| \leq  C(M) \left( \dfrac{b_0^2}{ |\log(b_0)|} + \|\partial_s \e\|_{L^2(y < M)} \right) , \quad \frac{|b_s(0)|}{2} (\Lambda Q, \chi_M \Phi) \leq |LHS|. 
\ee
On the other hand, after time-differentiation, we obtain :
\be
\label{trhidfrimula}
\partial_s \e (0) = \lambda(0)\partial_t\e(0) =  -b_s(0){\partial_b P_{B_1(b(0))}}-b(0)\Lambda u_0
 +  \lambda(0) \left(b_0\Lambda P_{B_1(b_0)}\right)_{\frac{1}{\lambda(0)}}.
\ee
Observe now from \fref{estpbcut} that $$\left\|{\partial_b P_{B_1(b_0)}}\right\|_{L^2(y\leq 2M)}\lesssim C(M) b_0\leq \sqrt{b_0}$$ 
which together with  \fref{firstbound}, \fref{controlb0precis} and \fref{defmotionintia} yields: 
\be
\label{cnekoheioheof}
\|\pa_s\e(0)\|_{L^2(y\leq 2M)}=\lambda(0)\|\partial_t\e(0)\|_{L^2(y\leq 2M)} \lesssim \frac{b_0^2}{|\log b_0|} + |b_s(0)|\sqrt{b_0},
\ee 
which together with \fref{cwouweofieo} concludes the proof of the initial bound \eqref{poitwisebsbootsd} on $b_s$.

Then, we compute :
$$
\partial_t w(0) = u_1 - \left( \dfrac{b_s(0)}{\lambda(0)} \partial_b P_{B_1(b(0))} + \dfrac{b(0)}{\lambda(0)} \Lambda P_{B_1(b(0))} \right)_{\lambda(0)}
$$ 
so that, introducing \fref{trhidfrimula} and previous estimates on $b(0),$ we compute : 
$$
\|\partial_tw (0)+\dfrac{b(0)}{\lambda(0)}((1-\chi_{B_1(b(0))})\Lambda Q)_{\lambda(0)}\|_{L^2}\lesssim b_0 |\ln(b_0)| \leq \sqrt{b_0}. 
$$
and 
\be \label{initiale0_2}
\|\nabla\partial_t w(0)\|_{L^2} \lesssim \frac{b_0^2}{|\log b_0|},
\ee
Together with \eqref{firstboundw}, this yields \eqref{bootneergynormsd} and  \eqref{poitnwiseboundbootsd}.

Finally, straightforward computations yield:
$$
\kappa_- = \dfrac{1}{2} (\e, \psi ) - \frac{1}{\zeta}(\partial_s \e, \psi) - \frac{b_s}{2 \zeta} \left(\partial_b P_{B_1},\psi \right),
$$  
Consequently, we apply \eqref{poitnwiseboundbootsd}, noting that $w(t) = (\e(t))_{\lambda(t)}$, 
and \eqref{trhidfrimula} because of the exponential decay of  $\psi$ to compute
\be \label{initialG}
|\kappa_-(0)| \lesssim \dfrac{b_0^2}{|\log b_0|}
\ee

\noindent{\bf Step 2:} Computation of $d_+.$\\
We now claim from an explicit compuation that given $a_+$, the initialization \fref{intialisaifjeo} can be reformulated in the form 
\be
\label{movemtndplus}
F(d_+)=a_+ \ \ \mbox{with} \ \ \frac{\partial F}{\partial d_+}|_{d_+=0}=\frac{\|\psi\|_{L^2}^2}{2}+O(b_0)
\ee which from the implicit function theorem concludes the proof of Lemma \ref{smalldata}.\\
Let us briefly justify \fref{movemtndplus}. We want to study the mappping
$$
\begin{array}{rcl}
\mathcal{V} & \rightarrow & \mathbb R^4 \\
d_+ &\longmapsto& \left[b(t), b_s(t), (\varepsilon(0),\psi), (\partial_s \varepsilon(t),\psi) \right] 
\end{array}
$$
where $\mathcal{V}$ is a neighborhood of $0.$ To this end, it is necessary to study the dependencies of all initial
parameters on $d_+.$ For conciserness, we denote by $d$ differentiation w.r.t. $d_+$ in what follows\\
{\noindent {\bf Computation of $(\lambda(0),\tilde{\varepsilon}(0))$}}. As a first step in the modulation theory, we proved that $(\lambda(0),\tilde{\varepsilon}(0)) = \Phi(u_0)$ where
$\Phi$ is a smooth mapping $\dot{H}^1(\mathbb R^N) \to \mathbb R \times \dot{H}^1(\mathbb R^N)$ defined in a neighborhood
of $Q.$ Due to the exponential decay of $\psi \in \mathcal{C}^{\infty}(\mathbb R^N)$ we thus have that $\lambda(0)$
is a smooth function of $d_+$ with differential $d\lambda(0) = d\lambda \in \mathbb R.$  We have the same result for $\varepsilon$
with differential $d\tilde{\varepsilon}(0) = d\tilde{\varepsilon} \in \dot{H}^1(\mathbb R^N).$ By definition, we have 
$$
\tilde{\varepsilon}(0) = u_0 - Q_{\frac{1}{\lambda}} 
$$
so that:
$$
d \tilde{\varepsilon} = \psi + \frac{d\lambda}{\lambda(0)} (\Lambda Q)_{\frac{1}{\lambda(0)}}. 
$$
{\noindent {\bf Computation of $b(0)$}}: From \fref{nitnisnodbweor}, $b(0)$ is a $\mathcal C^1$ mapping with:
\bee
db(0)  & = &  d\lambda \left[  \dfrac{((u_1)_{\frac{1}{\lambda(0)}}, \chi_M \Phi)}{ ( (\Lambda u_0)_{\frac{1}{\lambda(0)}} , \chi_M\Phi)}
			+ \dfrac{( (\Lambda^2 u_0)_{\frac{1}{\lambda(0)}} , \chi_M\Phi) - ((\Lambda u_1)_{\frac{1}{\lambda(0)}}, \chi_M \Phi)}{ ( (\Lambda u_0)_{\frac{1}{\lambda(0)}} , \chi_M\Phi)^2}\right]\\
			&
			- & \lambda(0)  \dfrac{((u_1)_{\frac{1}{\lambda(0)}}, \chi_M \Phi) ( (\Lambda \psi)_{\frac{1}{\lambda(0)}} , \chi_M\Phi)}{ ( (\Lambda u_0)_{\frac{1}{\lambda(0)}} , \chi_M\Phi)^2}
\eee
where (A.6) and (A.7) ensure that, for some $db \in \mathbb R,$ there holds : 
$$
db(0) = db + O(b_0).
$$

{\noindent {\bf Computation of $\varepsilon(0)$}}: Next,
$$
\varepsilon(0)=\tilde{\varepsilon}(0)-(P_{B_1(b(0))}-Q)
$$
Consequently, $(\varepsilon(0),\psi)$ is also a smooth function of $d_+$ with derivative $dps_1(0)$ satisfying
$$
d ps_1(0)= (d\tilde{\varepsilon},\psi) - db(0) (\partial_b P_{B_1(b(0))},\psi)
$$
Replacing $d\tilde{\varepsilon}$ by its values, and applying that $(\Lambda Q ,\psi) = 0$ together with $|\lambda(0)-1| \lesssim b_0^2/|\log(b_0)|,$
we get:
$$
(d\tilde{\varepsilon},\psi) = \|\psi\|_{L^2}^2  + O (b_0) 
$$
so that:
$$
d ps_1(0)= \|\psi\|^2_{L^2} + O(b_0).
$$
{\noindent {\bf Computation of $\partial_s\varepsilon(0) + b_s(0){\partial_b P_{B_1(b(0))}}$}}: From \fref{trhidfrimula},$$
\partial_s\varepsilon(0) =  -b_s(0){\partial_b P_{B_1(b(0))}}-b(0)\Lambda u_0 +  \lambda(0) \left(b_0\Lambda P_{B_1(b_0)}\right)_{\frac{1}{\lambda(0)}}
$$
so that  $(\partial_s\varepsilon(0) + b_s(0){\partial_b P_{B_1(b(0))}},\psi) $ is a smooth function of $d_+$ with derivative :
$$
d ps_2(0) = -db(0) (\Lambda u_0,\psi) + d\lambda \left( \left[ \left(b_0\Lambda P_{B_1(b_0)}\right)_{\frac{1}{\lambda(0)}} + \left(b_0\Lambda^2 P_{B_1(b_0)}\right)_{\frac{1}{\lambda(0)}}\right] ,\psi\right) -
 b(0)(\Lambda \psi,\psi),
$$
where, for the same orthogonality reason $(\Lambda Q,\psi) = 0,$ we have: 
$$
(\Lambda u_0, \psi) = (\Lambda Q,\psi) + O(b_0) = O(b_0)
$$
Consequently $d ps_2(0) = O(b_0).$\\
{\noindent {\bf Conclusion}}: Finally, there holds 
$$
\kappa_+(0)  = \dfrac{1}{2} \left[ (\varepsilon(0), \psi) + \frac{1}{\sqrt{\zeta}} (\partial_s \varepsilon(0) + b_s (0) \partial_b P_{B_1(b(0))}, \psi) \right]. 
$$
and $\kappa_+ (0) = a_+$ reduces to a simple 1D equation $F(d_+) = a_+$ with $F$ computed as combination of the above functions so that
it is smooth in a neighborhood of $0.$ Moreover, there holds:
$$
d F = \dfrac{1}{2} \left[ d ps_1(0)  + \dfrac{1}{\sqrt{\zeta}} d ps_2(0)\right] = \dfrac{\|\psi\|_{L^2}^2}{2} + O(b_0),
$$
and \fref{movemtndplus} is proved. This concludes the proof of Lemma \ref{smalldata}.


\section{Coercivity estimates} \label{app_coerc}
The aim of this section is a proof of the coercivity properties
of the quadratic form:
$$
B(\eta,\eta) =(\mathcal Bv,v)= \int_{\mathbb{R}^4} |\partial_r \eta|^2  + \int_{\mathbb{R}^4} W \eta^2
$$
where
$$
W(r) = 2V  + \frac{3}{2} \ r  V' = \dfrac{6}{(1+{r^2}/{8})^2}-\dfrac{9}{4}  \dfrac{r^2}{(1+{r^2}/{8})^3}
$$
We use the elementary method developed in \cite{FMR}. The coercitivity property of Lemma \ref{lemmacoercitivity} is a consequence of the two following facts. First the index of $B$ on $\dot{H}^1_{r}=\{u\ \ \mbox{radial} \ \ \mbox{with} \ \ \int|\nabla u|^2+\int\frac{u^2}{r^2}<+\infty\}$ is at most 2. From standard Sturm Liouville oscillation theorems, see Theorem XIII.8 \cite{ReedSimon}, this is equivalent to counting the number of zeroes of \begin{equation} \label{eq_Z}
\left\{
\begin{array}{ll}
\mathcal BU= 0 & \text{on $(0,\infty)$}, \\[10pt]
U(0) =1 \quad U'(0)=0, & 
\end{array}
\right.
\end{equation}
and this can be analytically reduced to counting the number of zeroes of a Bessel function. Then we need to show that the orthogonality conditions $(\eta,\psi)=(\eta,\Phi)=0$ are enough to treat the two negative directions. Arguing exactly as in \cite{FMR}, see also \cite{MS}, this is equivalent to first invert the operator $\mathcal B$ on $\dot{H}^1_{rad}$, and then show that $B$ restricted to $Span\{\mathcal B^{-1}\psi, \mathcal B^{-1} \Phi\}$ is definite negative, which is an elementary numerical check. We shall check these two facts below and refer to \cite{FMR} for the proofs that this implies the claimed coercitivity property. Note that the proofs in \cite{FMR} are given for exponentiallly decaying functions and potentials, but one checks easily that the decay of the potential $|W(r)|\sim \frac{1}{r^4}$ at infinity and $|\Phi(r)|\sim \frac{1}{r^4}$ are more than enough to have all proofs go through.\\

\begin{figure}[h] 
\begin{center}
\includegraphics[scale=.7]{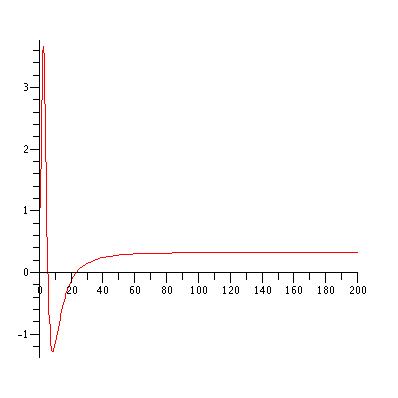}
\end{center}
\caption{Solution to \eqref{eq_Z} computed by MAPLE}
\label{fig_U}
\end{figure}

\subsection{Computation of the index of $B$}
We claim:
\begin{lemma}[Derivation of the index]
\label{lemmaindex}
The index of $\mathcal B$ on $\dot{H}^1_r$ is at most 2..
\end{lemma}
\begin{proof}
First,  we note that $W(r) \geq \hat{W}(r)$ where:
$$
\hat{W}(r) = -\frac{3}{2}  \dfrac{r^2}{(1+{r^2}/{8})^3}.
$$
Hence, classical Sturm-Liouville theory ensures that $U$ has less zeros than $\hat{U}$ 
the unique solution to :
\begin{equation} \label{eq_Zhat}
\left\{
\begin{array}{ll}
-\dfrac{1}{r^3} \dfrac{\textrm{d}}{\textrm{dr}} \left[ r^3 \dfrac{\textrm{d}}{\textrm{dr}} \hat{U}\right] + \hat{W} \hat{U} = 0 & \text{on $(0,\infty)$}, \\[10pt]
\hat{U}(0) =1 \quad \hat{U}'(0)=0, & 
\end{array}
\right.
\end{equation}
Second, we look for $\hat{U}$ of the form:
$$
\hat{U}(r) = \frac{2}{r^2}\overline{U}(r^2/2),
$$
with $\overline{U}$ a sufficiently smooth function.
Denoting by $s$ the new variable $r^2/2,$ straightforward calculations yield that
$\overline{U}$ is a solution to :
\begin{equation} \label{eq_Zbar}
\left\{
\begin{array}{ll}
- \dfrac{\textrm{d}^2}{\textrm{d$s^2$}} \overline{U} + \overline{W} \ \overline{U} = 0 & \text{on $(0,\infty)$}, \\[10pt]
\overline{U}(0) =0 \quad \overline{U}'(0)=1, & 
\end{array}
\right.
\end{equation}
where:
$$
\overline{W}(s) = -\frac{3}{2}  \dfrac{1}{(1+s/{4})^3}.
$$
Setting then $\overline{U}(s) = \sqrt{1 + s/4} \ \widetilde{U}(1/\sqrt{1+s/4}),$
we obtain that $\overline{U}$ is a solution to \eqref{eq_Zbar} if and only if 
$\widetilde{U}$ is a solution to 
$$
\left\{
\begin{array}{ll}
\tau^2 \dfrac{\textrm{d}^2}{\textrm{d$\tau^2$}} \widetilde{U} +  \tau \dfrac{\textrm{d}}{\textrm{d$\tau$}} \widetilde{U}  + (96 \tau^2 - 1)  \widetilde{U} = 0 & \text{on $(0,1)$}, \\[10pt]
\widetilde{U}(1) =0 \quad \widetilde{U}'(1)=-8, & 
\end{array}
\right.
$$Hence, $\widetilde{U}$ is a combination of Bessel functions:
$$
\widetilde{U}(\tau) = C_1 J(1, 4 \sqrt6\tau)+C_2 Y(1, 4 \sqrt6 \tau)
$$
We compute $(C_1,C_2)$ and draw the explicit combination with MAPLE. We obtain Figure \ref{fig_Bessel}. The computed solution $\widetilde{U}$
has two zeros on $(0,1).$ Moreover, it diverges in $0$ so that $\widetilde{U}(\tau) \sim K/\tau$ close to $0$ with $K\neq 0$ As a consequence 
$$
\hat{U}(r) \sim \frac{K}{4} \neq 0\ \ \mbox{when $r \to \infty$},
$$ 
and thus the index of $-\Delta+\hat{W}$ on $\dot{H}^1_{rad}$ is exactly two. Hence the index of $\mathcal B$ is at most 2. This completes the proof of Lemma \ref{lemmaindex}.
\begin{figure}[h] 
\begin{center}
\includegraphics[scale=.7]{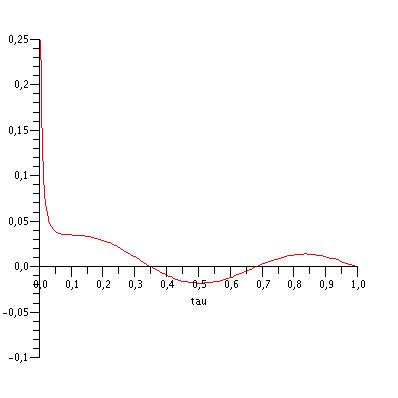}
\end{center}
\caption{Solution to \eqref{eq_Zbar} computed by MAPLE}
\label{fig_Bessel}
\end{figure}
\end{proof}

\subsection{Choice for the orthogonality conditions}
We now invert $\mathcal B$. We first check numerically that the solution $U$  does not vanish at infinity ie $$\lim_{r\to+\infty}U(r)>0,$$ see {Figure \ref{fig_U}}.

 Hence $U$ is not a resonance -note that if $U$ had been a resonance, we could have removed the resonance by diminishing a bit the potential and getting a potential with index 2 and no resonance-, and thus from standard ODE arguments, \cite{FMR}, there exists unique smooth solution in $\dot{H^1}_{rad}$ of:
\begin{equation} \label{eq_antepsi}
\left\{
\begin{array}{ll}
\mathcal BU=-\dfrac{1}{r^3} \dfrac{\textrm{d}}{\textrm{d$r$}} \left[ r^3 \dfrac{\textrm{d}}{\textrm{d$r$}} U\right] + W U = \psi & \text{on $(0,\infty)$}, \\[10pt]
U'(0)=0, & 
\end{array}
\right.  \ \ \mbox{with} \ \ (1+r^2)U\in L^{\infty}
\end{equation}
and
\begin{equation} \label{eq_anteDLQ}
\left\{
\begin{array}{ll}
\mathcal BU=-\dfrac{1}{r^3} \dfrac{\textrm{d}}{\textrm{d$r$}} \left[ r^3 \dfrac{\textrm{d}}{\textrm{d$r$}} U\right] + W U = \Phi & \text{on $(0,\infty)$}, \\[10pt]
U'(0)=0, & 
\end{array}
\right.  \ \ \mbox{with} \ \ (1+\frac{r^2}{\log r})U\in L^{\infty}
\end{equation}
We denote $\mathcal B^{-1}\psi$ and $\mathcal B^{-1}\Phi$ the respective solutions to these systems.
We recall the explicit formula 
$$
\Phi(r)=D\Lambda Q(r) = \dfrac{2 - 3r^2/4}{(1+r^2/8)^3}.
$$
In the remainder of this section we, check numerically that the restriction of $B$ to
$\mbox{Span}(\mathcal B^{-1}\psi,\mathcal B^{-1}\Phi)$ is definite negative, or equivalently:

\begin{lemma}[Numerical check of the orthogonality conditions]
\label{checknumerical}
The symmetric matrix $$
\mathbb{B} = 
\left[
\begin{array}{cc}
(\mathcal B^{-1}\psi ,\psi) & (\mathcal B^{-1}\Phi,\psi) \\
(\mathcal B^{-1}\Phi,\psi) & (\mathcal B^{-1}\Phi,\Phi)
\end{array}
\right]
$$
satisfies: 
\be
\label{numericalchecl}
(\mathcal B^{-1}\psi,\psi)<0 \ \ \mbox{and} \ \ \det \mathbb{B}>0,
\ee
and is thus definite-negative.
\end{lemma}

{\bf Numerical proof of Lemma \ref{checknumerical}} We use  standard MATLAB routines  for the computation of solutions 
to (\ref{eq_antepsi},\ref{eq_anteDLQ}). We note that we only fixed the initial value for $U'(0).$
The value $U(0)$ is left open in order to achieve the expected decay at infinity which characterizes the inverse.
In order to obtain $\mathcal B^{-1}\psi,$ we first compute $\psi.$ We obtain that the corresponding eigenvalue
is approximatively $l=-0.5860808922.$
Because $\psi$ decays exponentially, we only need to obtain an approximation on a short time-range.
We computed our solutions until $T_{\psi,max} = 30.$ 
We emphasize here that we use an explicit scheme. As a drawback, the accumulation of errors tends to make the numerical solution to become negative when the exact solution is exponential small. Hence, our scheme becomes unstable after time $\tilde{T}_{\psi,max}=18.$ 
Nevertheless, we extend our numerical solution with $0$ after this time.
This induces an exponentially small error.  
The pictures in Figure \ref{fig_psi} illustrate this computation.
On the left-hand side is drawn the obtained solution. On the right-hand side, we draw 
$\psi_{test}(r) = \psi(r)\exp(\sqrt{-l}r).$ We observe here that our solution enters the exponential asymptotic regime before the instability comes into play. 

\begin{figure}[h] 
\begin{center} 
\includegraphics[scale=.2]{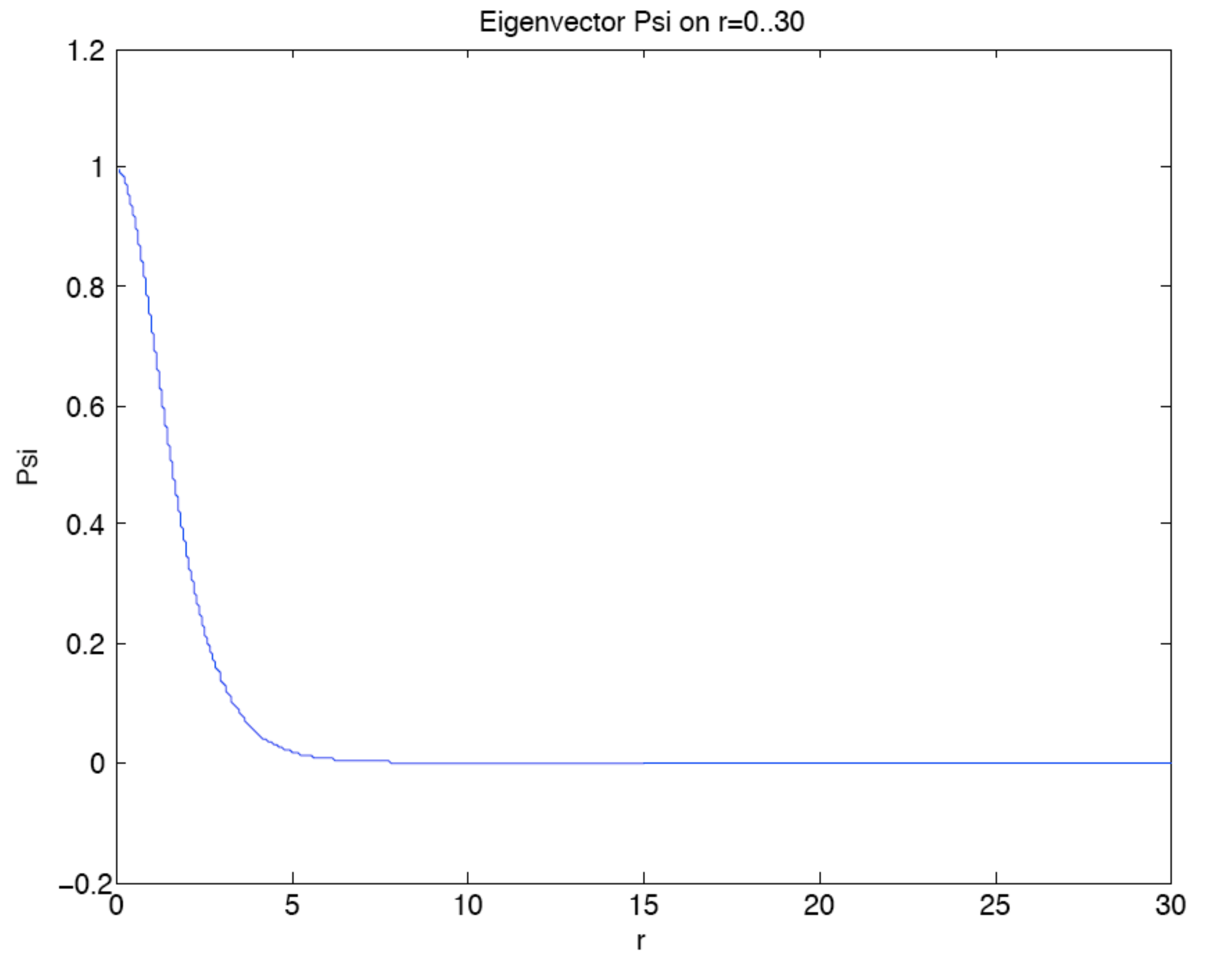}
\includegraphics[scale=.2]{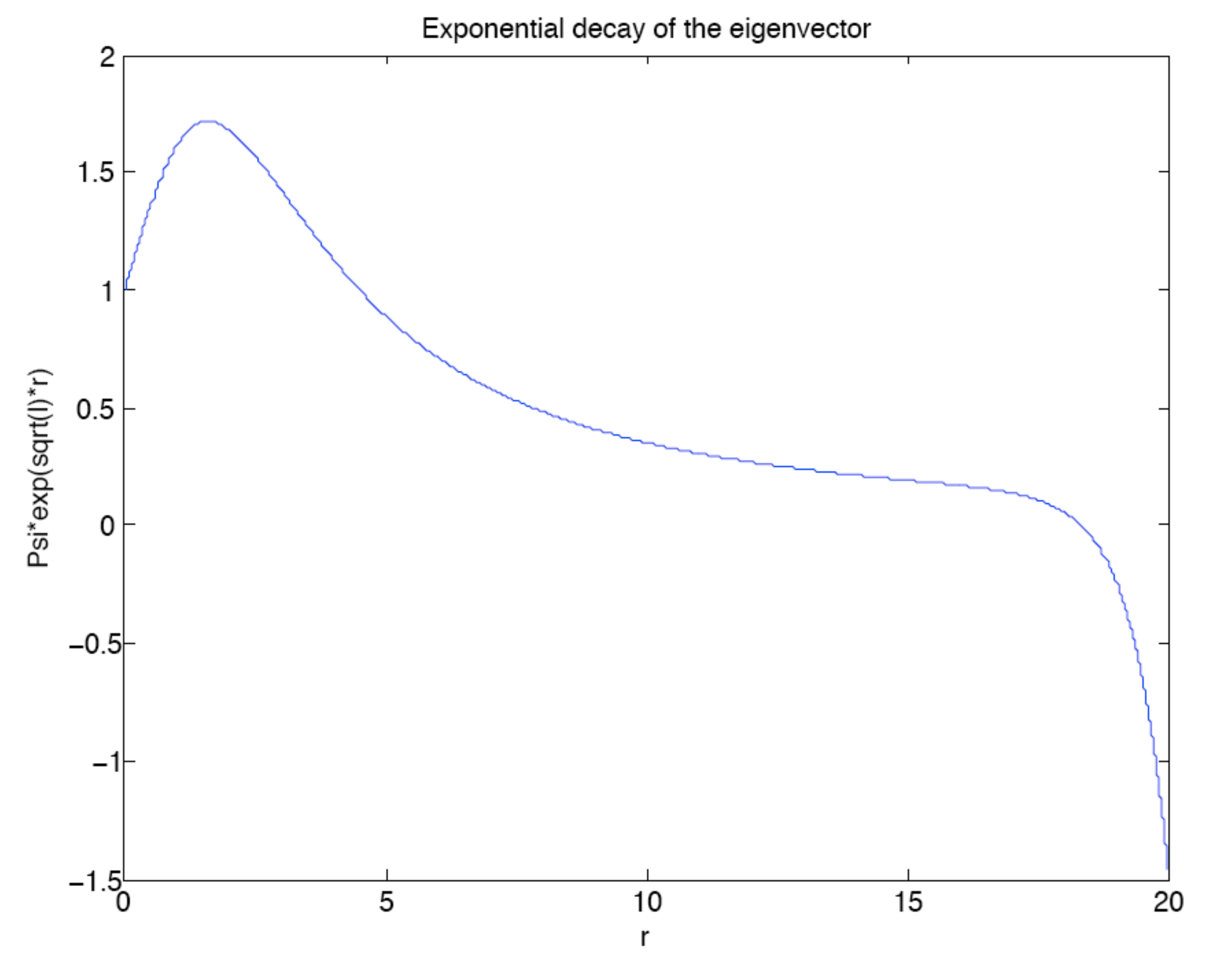}
\end{center}
\caption{Numerical simulations for $\psi$}
\label{fig_psi}
\end{figure}

The solution $\mathcal B^{-1}\psi$ is computed with the extension of $\psi.$ Straightforward ode analysis
shows that the unique solution decaying fast at infinity behaves like $1/r^2$
asymptotically. The choice of $U(0)$ is made with respect to this criterion. Figure \ref{fig_Apsi+test} illustrates that we obtained a solution with the suitable decay. 
As previously, on the left-hand side is a picture of the numerical solution. On the right-hand side
we plot $\mathcal B^{-1}\psi_{test}(r) = r^2 \mathcal B^{-1}\psi(r).$ In the latter computations, this solution is involved in scalar products with $\psi$. Hence even if drawn until $T_{max} = 300$, we only need a precise computation of this solution until $T_{\mathcal B^{-1}\psi,max} = 18.$

\begin{figure}[h] 
\begin{center} 
\includegraphics[scale=.2]{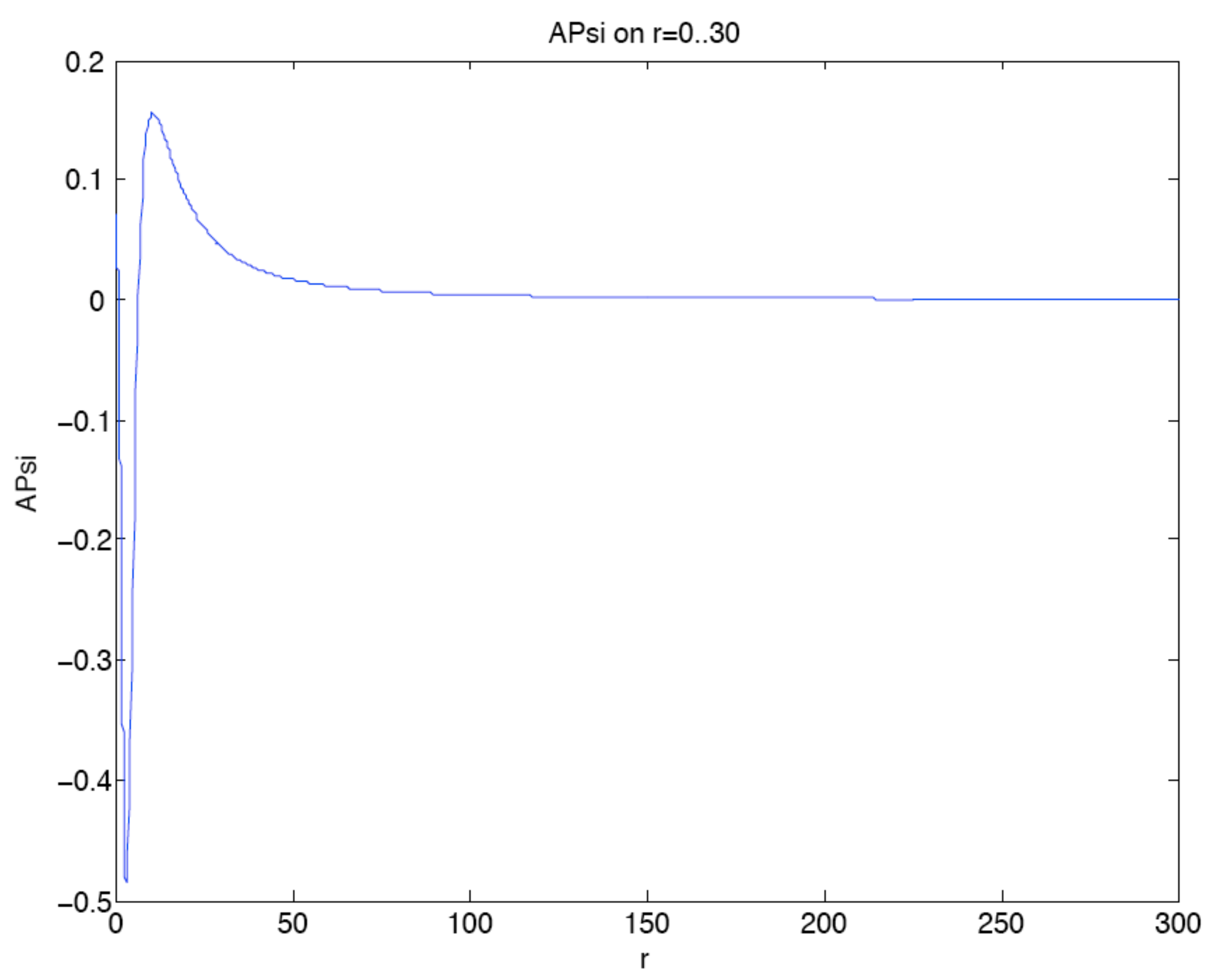}
\includegraphics[scale=.2]{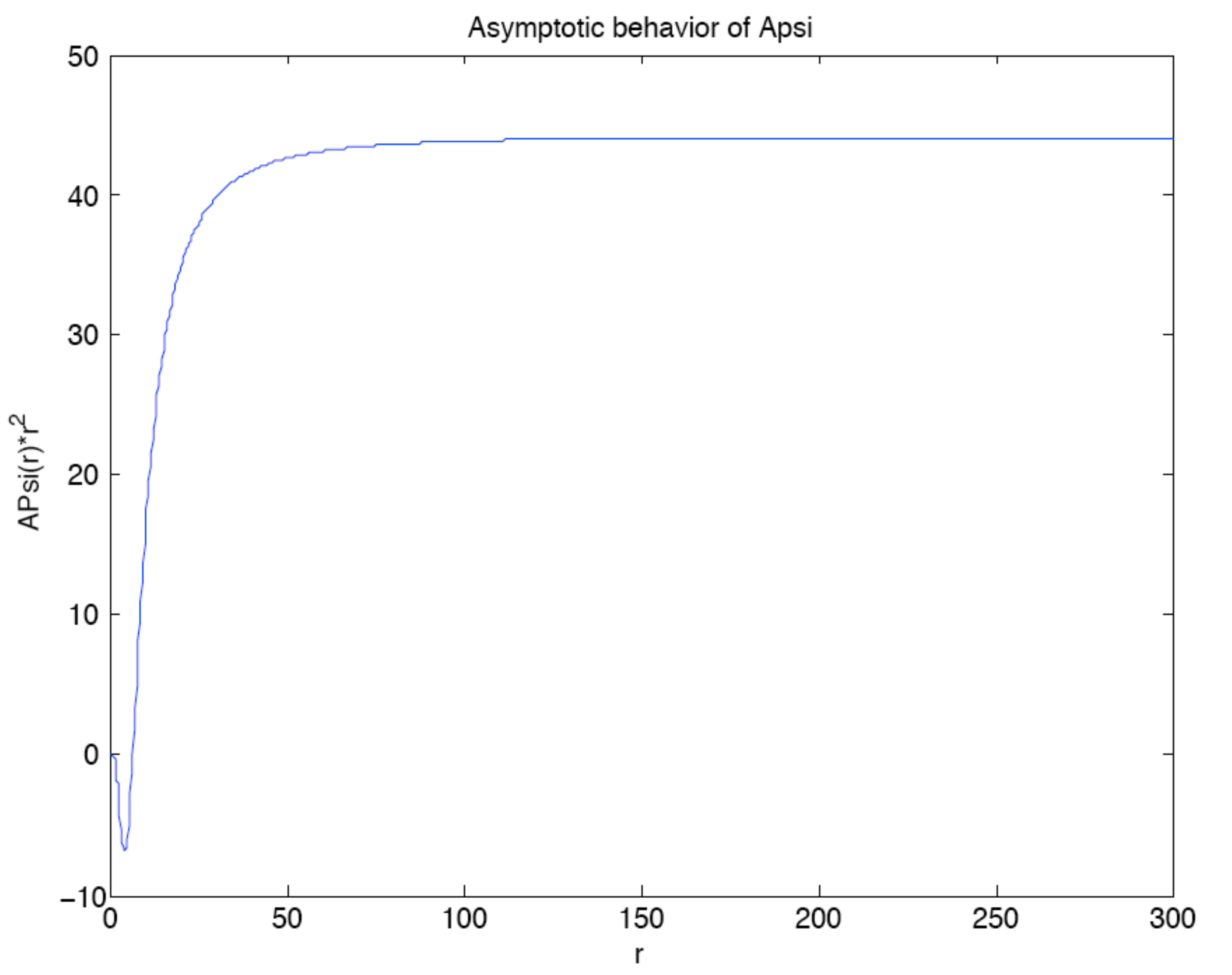}
\end{center}
\caption{Numerical simulations for $\mathcal B^{-1}\psi$}
\label{fig_Apsi+test}
\end{figure}

The last solution $\mathcal B^{-1}\Phi$ is computed with the same method. In this second case, the expected decay of the solution is $log(r)/r^2.$ Figure \ref{fig_AQ+test}  illustrates that we obtained a solution with the suitable decay. The picture on the right-hand side restricts to the time-interval $r=0..100$ because this
is the significant region.
In the latter computations, this solution is involved in integrals which converge slowly. Hence, we compute this solution until $T_{\mathcal B^{-1}\Phi,max} = 1000.$

\begin{figure}[h]\begin{center} 
\includegraphics[scale=.2]{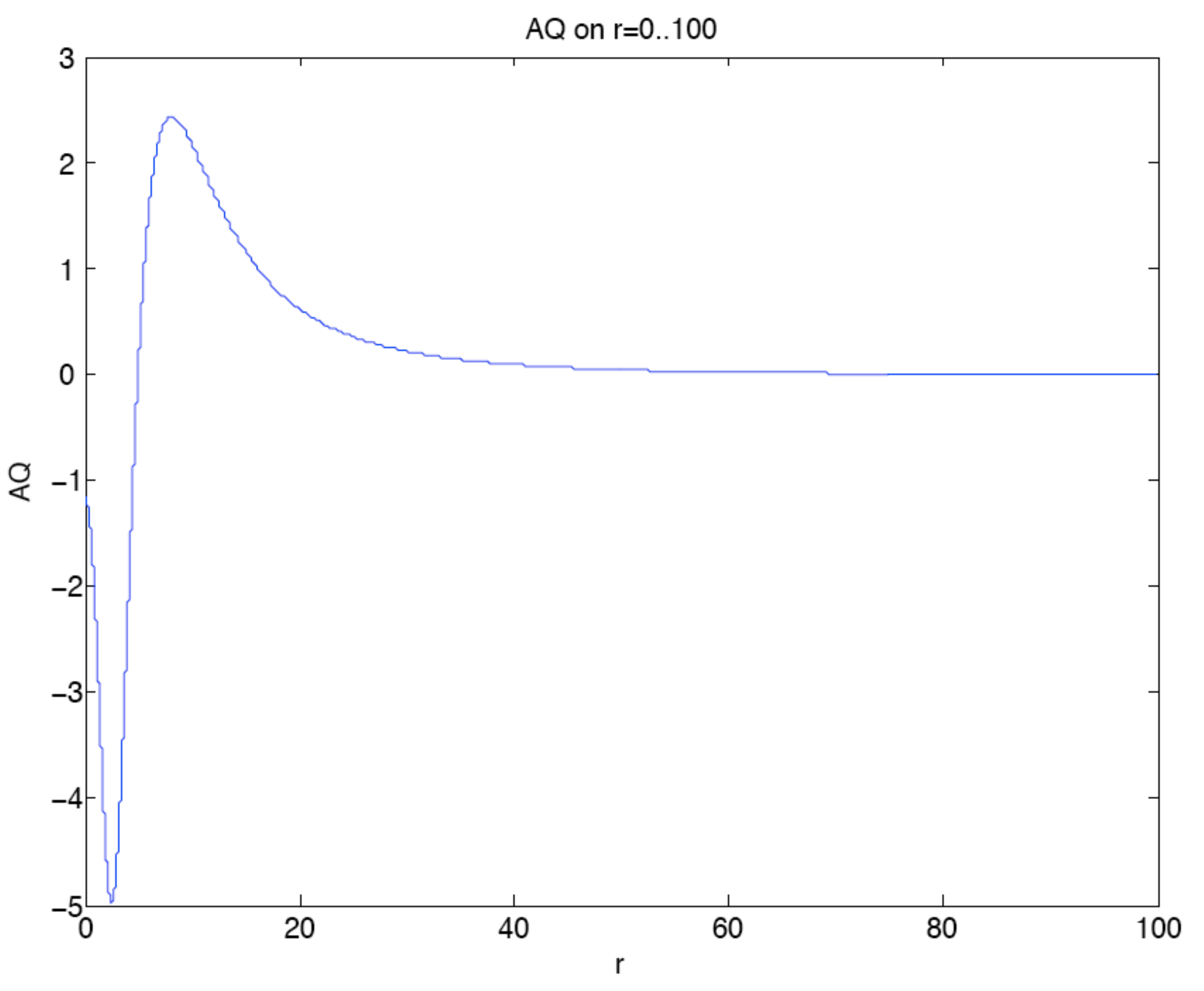}
\includegraphics[scale=.2]{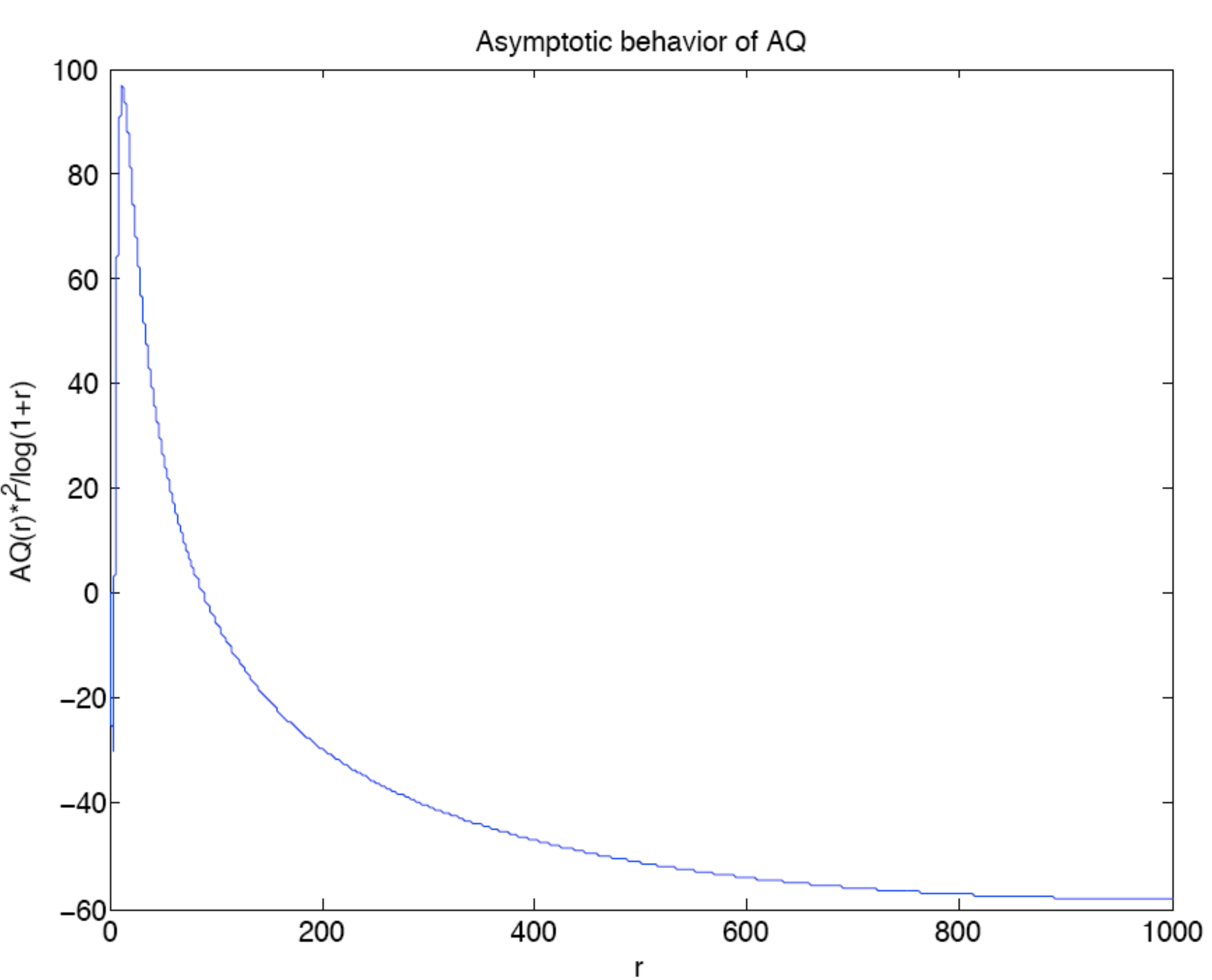}
\end{center}
\caption{Numerical simulations for $\mathcal B^{-1}\Phi$}
\label{fig_AQ+test}
\end{figure}

We now compute numerically the entries of the matrix $\mathbb{B}$. We first compute $(\mathcal B^{-1}\Phi,\psi)  = (\mathcal B^{-1}\psi,\Phi).$ 
The exponential decay of $\psi$ implies that we need to compute the first integral 
$(\mathcal B^{-1}\Phi,\psi)$ on a shorter time-interval. Hence, we prefer this computation to the second one.
We compute the $L^2$-scalar products with a standard trapezoidal method. 
Changing the time-interval and the time-step, the computations are stable up to an error of 
$10^{-2}.$ We get the following approximations for the integrals involving $\psi$. 
$$
(\mathcal B^{-1}\psi, \psi) = -4.63 \pm 10^{-2} \qquad (\mathcal B^{-1}\Phi,\psi) = 32.65 \pm 10^{-2}. 
$$

The last integral is a more involved computation. Indeed, standard real analysis 
implies that there holds : 
$$
I(M) := \int_{0}^M \mathcal B^{-1}\Phi Q(r) \Phi(r) r^{3} \text{d$r$} = (\mathcal B^{-1}\Phi,\Phi) + err(M)
$$
with a remainder satisfying $err(M) = (K+o(1))\ln(M)/M^2$
for some constant $K.$ 
This remainder going slowly to $0,$ we see numerically that our computations
has not converged even when integrating until $T_{\mathcal B^{-1}\Phi,max} = 1000$
(see Figure \ref{fig_PS22}, red crosses). In order to improve the rate of convergence
we compute an approximation of coefficient $K$ and substract the estimated error term
of our computations. This yields Figure \ref{fig_PS22}, blue circles. On this second computation we 
obtain a very good rate of convergence. Hence, we provide the approximation
$$
(\mathcal B^{-1}\Phi,\Phi) = -574.25 \pm 10^{-2}
$$
Hence $$\det(\mathbb{B})=1591\pm 10$$ which concludes the numerical proof of Lemma \ref{checknumerical}.

\begin{figure}[h] 
\begin{center} 
\includegraphics[scale=.4]{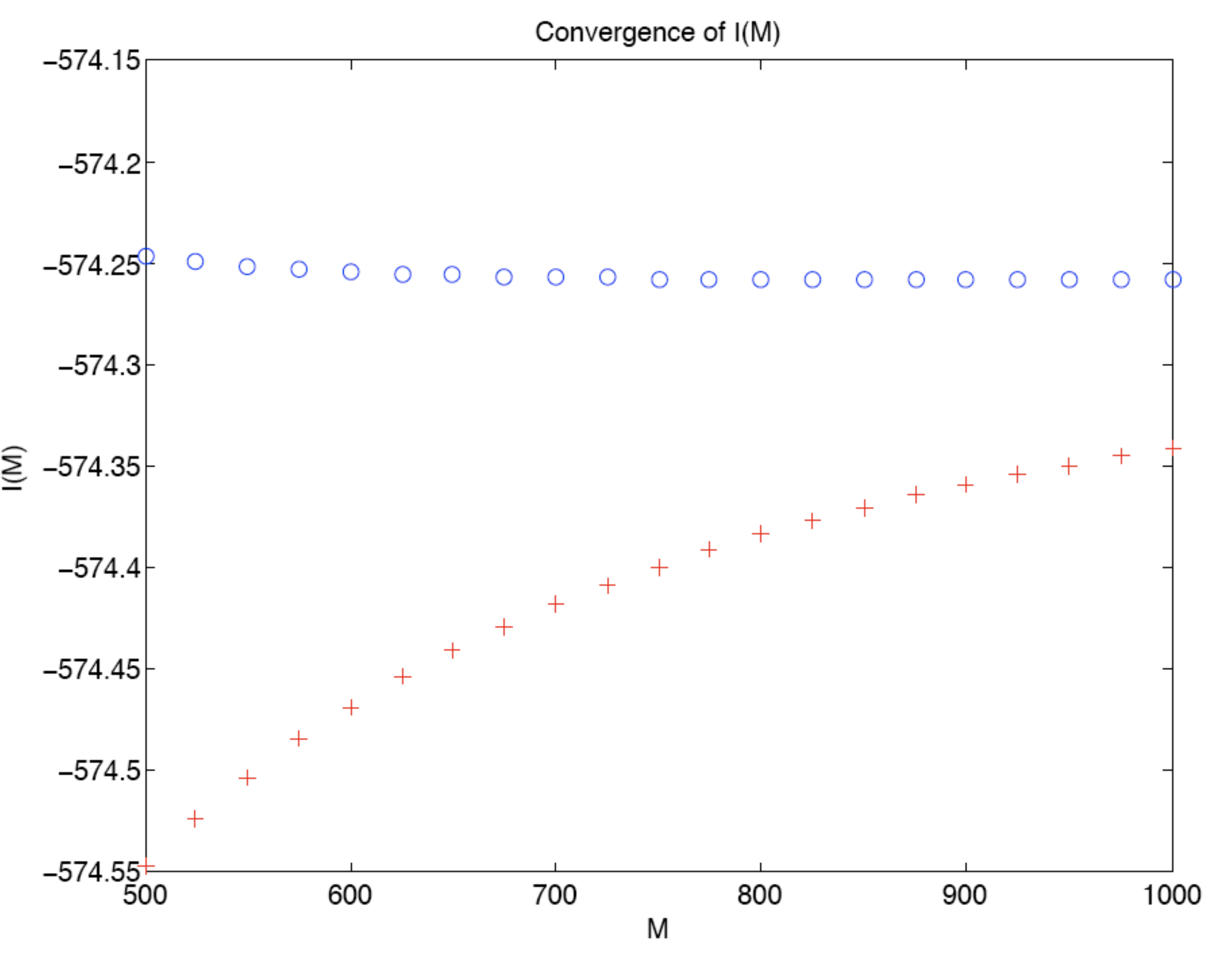}
\end{center}
\caption{Computations for $(\mathcal B^{-1}\Phi,\Phi)$}
\label{fig_PS22}
\end{figure}



\section{Some linear estimates}


We start by recalling some obvious integration-by-part results :
\begin{lemma} \label{lemmaHardy0}
For any $N \geq 3,$ there exists  a constant $C$ for which there holds, for any $v \in H^{1}_{rad}(\mathbb R^N)$
\be \label{hardy0}
\left[\int_{\mathbb R^N} \dfrac{|v(y)|^2}{|y|^2}\right]^{\frac{1}{2}} + 
\sup_{y \in \mathbb R^N} \left( |y|^{\frac{N-2}{2}} |v(y)|\right) \leq C 
\left[ \int_{\mathbb R^n} |\nabla v(y)|^2\right]^{\frac{1}{2}}.
\ee
\end{lemma}
Looking for control on further derivatives, we prove:

\begin{lemma}[Hardy inequalities]
\label{lemmaloghrdy}
Let $N=4$. Then $\forall R>2$, $\forall v\in H^2_{rad}(\RR^N)$, there holds the following controls:
\be
\label{hardyhtwo}
\int  \frac{|\pa_y v|^2}{y^2}\lesssim \int (\Delta v)^2,
\ee
\be
\label{harfylog}
\int_{y\le R} \frac{|v|^2}{y^4(1+|\log y|)^2}\lesssim  \int_{y\le R} \frac{|\pa_y v|^2}{y^2}+\int_{y\leq 2}|v|^2.
 \ee
\end{lemma}
\be 
\label{harfysanslog}
\int_{R \le y\le 2R}\frac{|v|^2}{y^4} \lesssim {\log R}\int_{y\le R} \frac{|\pa_y v|^2}{y^2}+\int_{y\leq 2}|v|^2.
\ee

\begin{proof} Let $v$ smooth. \eqref{hardyhtwo} follows from the explicit formula after integration by parts $$\int(\Delta v)^2=\int (\pa_{yy}v+\frac{N-1}{y}\pa_yv)^2=\int (\pa_{yy}v)^2+(N-1)\int\frac{|\pa_y v|^2}{y^2}.$$ To prove \fref{harfylog}, let 
\be
\label{ofueofu}
a\in [1,2] \ \ \mbox{such that} \ \ |v(a)|^2\leq \int_{1\leq y\leq 2}|v|^2.
\ee
Let $f(y)=-\frac{{\bf e}_y}{y^3(1+\log (y))}$ so that $\nabla \cdot f=\frac{1}{y^4(1+|\log y|)^2}$, and integrate by parts to get: 
\bea
\label{stepwzofp}
\nonumber \int_{a\le y\le R} \frac{|v|^2}{y^4(1+\log y)^2} &=&  \int_{a\leq y\le R} |v|^2\nabla \cdot f\\
\nonumber & = &- \left[\frac{|v|^2}{1+\log (y)}\right]_{a}^R +2\int_{y\le R}  \frac{v\partial_y v}{y^3(1+\log y)}\\
& \lesssim &  |v(a)|^2+\left(\int_{y\le R} \frac{|v|^2}{y^4(1+|\log y|)^2}\right)^{\frac{1}{2}}\left(\int_{y\le R} \frac{|\pa_yv|^2}{y^2}\right)^{\frac{1}{2}}.
\eea
similarly, using $\tilde{f}(y)=\frac{{\bf e}_y}{y^3(1-\log (y))}$, we get:
\bea
\label{stepwzofpbis}
\nonumber  \int_{\e\le y\le a} \frac{|v|^2}{y^4(1-\log y)^2} & =&  \int_{a\leq y\le R} |v|^2\nabla \cdot \tilde{f} \\
\nonumber & = & \left[\frac{|v|^2}{1-\log (y)}\right]_{\e}^a +2\int_{y\le a} v\partial_y v \frac{1}{y^3(1-\log y)}\\
& \lesssim &  |v(a)|^2+\left(\int_{y\le R} \frac{|v|^2}{y^4(1+|\log y|)^2}\right)^{\frac{1}{2}}\left(\int_{y\le R} \frac{|\pa_yv|^2}{y^2}\right)^{\frac{1}{2}}.
\eea
\fref{ofueofu}, \fref{stepwzofp} and \fref{stepwzofpbis} now yield \fref{harfylog}.
The last inequality \eqref{harfysanslog} is a straightforward variant of \cite[Lemma B.1, (B.4)]{RaphRod}
and is left to the reader.
\end{proof}

\begin{lemma}[Coercitivity estimates with H]
\label{lemmahardy1}
Let $\psi$ be the first eigenvector of $H$. Then there exists $c>0$ and $M_0\geq 1$ such that for $M\geq M_0$, there exists $\delta(M)>0$ such that given  $u\in H^1_{rad}(\RR^N)$,
there holds 
\be
\label{coercH}
(Hu,u) \geq  c\int(\pa_y u)^2 -  \frac{1}{c}\left[(u,\psi)^2+(u,\chi_M\Phi)^2\right]
\ee
\be
\label{secondordercontrol}
\int (Hu)^2 \geq  \delta(M)\left[\int\frac{(\pa_yu)^2}{y^2}+\int\frac{u^2}{y^4(1+|\log y|)^2}\right]- \frac{1}{\delta(M)}(u,\chi_M\Phi)^2.
\ee
\end{lemma}

\begin{proof} \fref{coercH} is a standard consequence of the coercitivity of the linearized energy which admits exactly $\psi$ as bound state  and $\Lambda Q$ as resonance at the origin, the good enough localization of $\Phi$ \fref{defphi} and the nondegeneracy \fref{cancelaationphi}. The detailed proof is left to the reader.\\
To prove \fref{secondordercontrol}, we first observe the key subcoercivity property:
\bea
\label{subceorcuio}
\nonumber \int(Hu)^2& = & \int(\Delta u+Vu)^2=\int(\Delta u)^2-2\int V(\pa_yu)^2 +\int (\Delta V+V^2)u^2\\
& \geq & c\left[\int (\Delta u)^2+\int\frac{u^2}{1+y^6}\right]-\frac{1}{c}\left[\int\frac{(\pa_yu)^2}{1+y^4}+\int\frac{u^2}{1+y^8}\right].
\eea
where we used the asymptotic value $$V(y)=\frac{N(N+2)(N-2)}{y^4}\left[1+O(\frac{1}{y^2})\right] \ \ \mbox{as} \ \ y\to+\infty.$$ \fref{secondordercontrol} now follows by contradiction. Let $M>0$ fixed and consider a sequence $u_n$ such that 
\be
\label{normlaization}
\int\frac{(\pa_yu_n)^2}{y^2}+\int\frac{u_n^2}{y^4(1+|\log y|)^2}=1
\ee
and 
\be
\label{seeumporgpo}
\int(Hu_n)^2\leq \frac{1}{n}, \ \  \ (u_n,\chi_M\Phi)=0.
 \ee
Then by semicontinuity of the norm, $u_n$ weakly converges on a subsequence to $u_{\infty}\in H^1_{loc}$ solution to $Hu_{\infty}=0.$ $u_{\infty}$ is smooth away from the origin and hence the explicit integration of the ODE and the regularity assumption at the origin $u_{\infty}\in H^1_{loc}$ implies $$u_{\infty}=\alpha \Lambda Q.$$ On the one hand, the uniform bound \fref{normlaization} together with the local compactness of Sobolev embeddings ensure up to a subsequence: $$\int\frac{(\pa_yu_n)^2}{1+y^4}+\int \frac{|u_n|^2}{1+y^8}\to \int\frac{(\pa_yu_{\infty})^2}{1+y^4}+\int \frac{|u_{\infty}|^2}{1+y^8} \ \ \mbox{and} \ \ (u_n,\chi_M\Phi)\to (u_{\infty},\chi_M\Phi)$$ thanks to the $\chi_M$ localization. We thus conclude that  $$\alpha(\Lambda Q,\chi_M\Phi)=(u_{\infty},\chi_M\Phi)=0 \ \ \mbox{and thus} \ \ \alpha=0.$$ On the other hand, the subcoercivity property \fref{subceorcuio}, the Hardy control \fref{hardyhtwo}, \fref{harfylog} and \fref{normlaization}, \fref{seeumporgpo} ensure $$\int\frac{(\pa_yu_n)^2}{1+y^4}+\int\frac{u_n^2}{1+y^8}\geq C>0$$ from which
$$\alpha^2\left[ \int\frac{(\pa_y\Lambda Q)^2}{1+y^4}+\int \frac{|\Lambda Q|^2}{1+y^8}\right]= \int\frac{(\pa_yu_{\infty})^2}{1+y^4}+\int \frac{|u_{\infty}|^2}{1+y^8}\geq C>0 \ \ \mbox{and thus} \ \ \alpha\neq 0.$$ A contradiction follows. This concludes the proof of \fref{secondordercontrol} and Lemma \ref{lemmahardy1}. \end{proof}

Straightforward computations show that the coercitivity estimates with $H$ can be adapted to any of the operator $H_{\lambda}$ yielding,
for any $\lambda >0$ and $u \in H^1_{rad}(\mathbb R^N),$
\be
\label{coercHlambda}
(H_{\lambda}u,u)  \geq  c \int(\pa_y u)^2   -  \frac{1}{c\lambda^4}\left[(u,(\psi)_{\lambda})^2+(u,(\chi_M\Phi)_{\lambda})^2\right]
\ee
for the same $c$ and $\delta(M)$ as in Lemma \ref{lemmahardy1}. 

\end{appendix}



\begin{thebibliography}{10}

\bibitem{Alinhac} Blow up for nonlinear hyperbolic equations, volume 17 of Progress in Nonlinear Differential Equations and their Applications, Birkh\"auser Boston Inc., Boston, MA, 1995.

\bibitem{papierheat} van den Berg, J. B.; Hulshof, J.; King, J.R., Formal asymptotics of bubbling in the harmonic map heat flow, SIAM J. Appl. Math. 63 (2003), no. 5, 1682--1717.

\bibitem{bizon}  Bizo\'n, P.; Chmaj, T.; Tabor, Z., Formation of singularities for equivariant $(2+1)$-dimensional wave maps into the 2-sphere, Nonlinearity 14 (2001), no. 5, 1041--1053.

\bibitem{ChrsitZadeh} Christodoulou, D.; Tahvildar-Zadeh, A. S., On the regularity of spherically symmetric wave maps, Comm. Pure Appl. Math. 46 (1993), no. 7, 1041--1091.

\bibitem{cotemartelmerle} Cote, R.; Martel, Y.; Merle, F, Construction of multisolitons solutions for the $L^2$-supercritical gKdV and NLS equations,  arXiv:0905.0470 (2010).

\bibitem{DMKquanta} Duyckaerts, T.; Kenig, C.E.; Merle, F., Universality of blow up profile for small radial type II blow up solutions of energy critical wave equation, arXiv:0910.2594 (2009)

\bibitem{DMKquantabis} Duyckaerts, T.; Kenig, C.E.; Merle, F., Universality of blow up profile for small type II blow up solutions of energy critical wave equation: the non radial case, arXiv:1003.0625 (2010)

\bibitem{FMR} Fibich, G.; Merle, F.; Raphael, P., Numerical proof of a spectral property related to the singularity formation for the $L^2$ critical nonlinear Schr\"odinger equation,  Phys. D  220  (2006),  no. 1, 1--13.

\bibitem{grillakis}  Grillakis, M.G., Regularity and asymptotic behaviour of the wave equation with a critical nonlinearity. Ann. of Math. (2) 132 (1990), no. 3, 485--509.


\bibitem{Jurgens} J\"orgens, K., Das Anfangswertproblem im Grossen f?r eine Klasse nichtlinearer Wellengleichungen. (German) Math. Z. 77 1961 295--308. 

\bibitem{karastrauss}  Karageorgis, P.; Strauss, W. A., Instability of steady states for nonlinear wave and heat equations, J. Differential Equations 241 (2007), no. 1, 184--205.

\bibitem{KaWe} Kavian, O.; Weissler, F.B., Finite energy self-similar solutions of a nonlinear wave equation, Comm. Partial Differential Equations 15 (1990), no. 10, 1381--1420.

\bibitem{MS} Marzuola, J.; Simpson, G.. Spectral Analysis for matrix Hamiltonian operators, arXiv:1003.2474 (2010)

\bibitem{KMacta} Kenig, C.E.; Merle, F., Global well-posedness, scattering and blow-up for the energy-critical focusing non-linear wave equation. Acta Math. 201 (2008), no. 2, 147--212.

\bibitem{KSwave} Krieger, J.; Schlag, W., On the focusing critical semi-linear wave equation, Amer. J. Math. 129 (2007), no. 3, 843--913.MK

\bibitem{KSnls} Krieger, J.; Schlag, W., Non-generic blow-up solutions for the critical focusing NLS in 1-D. J. Eur. Math. Soc. (JEMS) 11 (2009), no. 1, 1--125.

\bibitem{KSTinvent}  Krieger, J.; Schlag, W.; Tataru, D., Renormalization and blow up for charge one equivariant critical wave maps, Invent. Math. 171 (2008), no. 3, 543--615.

\bibitem{KSTwave}  Krieger, J.; Schlag, W.; Tataru, D., Slow blow-up solutions for the $H^1(\Bbb R^3)$ critical focusing semilinear wave equation, Duke Math. J. 147 (2009), no. 1, 1--53. 

\bibitem{KSTyangmills} Krieger, J.; Schlag, W.; Tataru, D., Renormalization and blow up for the critical Yang-Mills problem, Adv. Math. 221 (2009), no. 5, 1445--1521.

\bibitem{levine} Levine, H.A., Instability and nonexistence of global solutions to nonlinear wave equations of the form $Pu_{tt}=-Au+F(u)$, Trans. Amer. Math. Soc. 192 (1974), 1--21.

\bibitem{martelmulti} Martel, Y., Asymptotic $N$-soliton-like solutions of the subcritical and critical generalized Korteweg-de Vries equations, Amer. J. Math. 127 (2005), no. 5, 1103--1140.

\bibitem{MartelMerlekdv} Martel, Y.; Merle, F., Stability of blow-up profile and lower bounds for blow-up rate for the critical generalized KdV equation, Ann. of Math. (2) 155 (2002), no. 1, 235--280.

\bibitem{MR1} Merle, F.; Rapha\"el, P., Blow up dynamic and upper bound on the blow up rate for critical nonlinear Schr\"odinger equation, Ann. Math. 161 (2005), no. 1, 157--222.

\bibitem{MR2} Merle, F.; Rapha\"el, P., Sharp upper bound on the blow up rate for critical nonlinear Schr\"odinger equation, Geom. Funct. Anal. 13 (2003), 591-642.

\bibitem{MR3} Merle, F.; Rapha\"el, P., On universality of blow up profile for $L^2$ critical nonlinear Schr\"odinger equation, Invent. Math. 156, 565-672 (2004).

\bibitem{MR4} Merle, F.; Rapha\"el, P., Sharp lower bound on the blow up rate for critical nonlinear Schr\"odinger equation, J. Amer. Math. Soc. 19 (2006), no. 1, 37--90.

\bibitem{MR5} Merle, F.; Rapha\"el, P.,  Profiles and quantization of the blow up mass for critical nonlinear Schr\"odinger equation, Comm. Math. Phys.  253  (2005),  no. 3, 675--704.

\bibitem{Merlezaag1} Merle, F.; Zaag, H., Determination of the blow-up rate for the semilinear wave equation, Amer. J. Math. 125 (2003), no. 5, 1147--1164.

\bibitem{Merlezaag2} Merle, F.; Zaag, H., Openness of the set of non-characteristic points and regularity of the blow-up curve for the 1 D semilinear wave equation, Comm. Math. Phys. 282 (2008), no. 1, 55--86.

\bibitem{R1} Rapha\"el, P., Stability of the log-log bound for blow up solution to the critical non linear Schr\"odinger equation,  Math. Ann.  331  (2005),  no. 3, 577--609.

\bibitem{RaphRod} Rapha\"el, P.; Rodnianski, I., Stable blow up dynamics for the critical corotational wave maps and equivariant Yang Mills problems, submitted.

\bibitem{RaphSzef} Rapha\"el, P.; Szeftel, J. Existence and uniqueness of minimal blow up solutions to an inhomgeneous mass critical NLS, arXiv:1001.1627 (2009).

\bibitem{ReedSimon} Reed, M.; Simon, B., Methods of modern mathematical physics, vol I-IV, Academic Press, New York, 1972-1979.

 \bibitem{RodSter} Rodnianski, I., Sterbenz, J., On the formation of singularities in the critical $O(3)$ $\sigma$-model, to appear Ann. Math.

\bibitem{sogge} Sogge, C.D., Lectures on nonlinear wave equations, Monographds in Analysis, II, International Press, Boston, MA, 1995.

\bibitem{shattahtahvildar} Shatah, J.; Tahvildar-Zadeh, A. S., On the Cauchy problem for equivariant wave maps, Comm. Pure Appl. Math. 47 (1994), no. 5, 719--754.

\bibitem{strauss} Strauss, W. A., Nonlinear wave equations, CBMS Regional Conference Series in Mathematics, 73, AMS, Providence, RI, 1989. 

\bibitem{struweone}  Struwe, M., Globally regular solutions to the $u^5$ Klein-Gordon equation. Ann. Scuola Norm. Sup. Pisa Cl. Sci. (4) 15 (1988), no. 3, 495--513.

\bibitem{Struwewm}  Struwe, M., Equivariant wave maps in two space dimensions. Dedicated to the memory of J\"urgen K. Moser, Comm. Pure Appl. Math. 56 (2003), no. 7, 815--823


\end{thebibliography}
\end{document}